%% file: BEsurvey_revised.tex
\documentclass[a4paper]{article}
\usepackage{mathtools,amsthm,amsfonts,amssymb,mathrsfs,epsfig,bm,xspace,bbm}
\usepackage{array}
\usepackage{enumerate}
\usepackage{comment}    
\usepackage{microtype}  
\usepackage[T1]{fontenc}        
\usepackage{lmodern}

\usepackage{datetime}

\usepackage[top=3cm, bottom=3cm, left=3cm, right=3cm, heightrounded,
 marginparwidth=2.5cm, marginparsep=3mm]{geometry}
\parskip        \smallskipamount

\usepackage[usenames,dvipsnames,named]{xcolor}
\usepackage[colorlinks=true,urlcolor=black]{hyperref}

\usepackage[margin=1cm]{caption}        
\usepackage[usenames,dvipsnames,named]{xcolor}
\usepackage[colorlinks=true]{hyperref}
\def\blue{\color{blue}}
\def\red{\color{red}}

\numberwithin{equation}{section}
\newtheorem{theorem}{Theorem}[section]
\newtheorem{lemma}[theorem]{Lemma}
\newtheorem{corollary}[theorem]{Corollary}
\newtheorem{proposition}[theorem]{Proposition}

\newtheorem{conjecture}[theorem]{Conjecture}
\newtheorem*{theorem*}{Theorem}

\theoremstyle{definition}
\newtheorem{definition}[theorem]{Definition}

\theoremstyle{remark}
\newtheorem{remark}[theorem]{Remark}
\newtheorem{example}[theorem]{Example}

\def\N{\mathbb{N}}
\def\Z{\mathbb{Z}}
\def\Q{\mathbb{Q}}
\def\R{\mathbb{R}}

\def\P{\mathbb{P}}
\def\E{\mathbb{E}}

\def\FF{\mathscr{F}}

\def\SSS{\mathscr{S}}
\def\RRR{\mathscr{R}}

\def\NN{\mathcal N}

\newcommand{\calF}{\mathcal{F}}

\newcommand{\calU}{\mathcal{U}}
\newcommand{\calT}{\mathcal{T}}
\newcommand{\calG}{\mathcal{G}}

\newcommand{\calB}{\mathcal{B}}

\renewcommand{\phi}{\varphi}
\renewcommand{\epsilon}{\varepsilon}

\newcommand{\1}{{\text{\Large $\mathfrak 1$}}} 
\newcommand{\ind}[1]{{\text{\Large $\mathfrak 1$}_{\left\{#1\right\}}}}

\newcommand{\floor}[1]{{\left\lfloor #1 \right\rfloor}}

\newcommand{\comp}{\raisebox{0.1ex}{\scriptsize $\circ$}}

\newcommand{\var}{\operatorname{var}}

\renewcommand{\limsup}{\varlimsup}
\renewcommand{\liminf}{\varliminf}
\newcommand{\dd}{\mathrm{d}}

\newcommand{\eqdist}{\overset{\mathrm d}{=}}          

\newcommand{\e}{e}

\def\gl{\mathsf{L}}
\def\gr{\mathsf{R}}

\usepackage{lipsum}

\renewcommand\vec{\overrightarrow}

\usepackage{cancel}

\def\K{\mathbb{K}}
\def\CC{\mathscr{C}}
\def\Ai{\operatorname{Ai}}

\def\fG{\mathfrak G}
\def\fH{\mathfrak H}
\def\crit{\text{\rm crit}}
\def\esssup{\operatorname{esssup}}

\def\fm{\mathfrak m}
\def\fM{\mathfrak M}

\def\fG{\mathfrak G}
\def\fH{\mathfrak H}

\def\WW{\mathcal W}
\def\+.{+ \text{\tiny \raisebox{0.5mm}{$\bullet$}}}

\def\convd{\xrightarrow{\text{\rm d}}}

\usepackage{tikz}
\usetikzlibrary{positioning,arrows.meta}
\usetikzlibrary{decorations.markings}
\tikzstyle directed=[postaction={decorate,decoration={markings,
    mark=at position .65 with {\arrow{latex}}}}]

\usepackage{subcaption}

\newcommand{\configS}{\mathbb{S}}

\usepackage[plain,noend]{algorithm2e}

\begin{document}

\newdateformat{UKvardate}{\THEDAY\ \monthname[\THEMONTH] \THEYEAR}
\UKvardate

\title{\Large \bf Last passage percolation and limit theorems
in Barak-Erd\H{o}s directed random graphs and related models}
\author{\sc Sergey Foss \and \sc Takis Konstantopoulos \and
\sc Bastien Mallein \and \sc Sanjay Ramassamy}
\date{\today}
\maketitle

\begin{abstract}
We consider directed random graphs, the prototype of which being the
Barak-Erd\H{o}s graph $\vec G(\Z, p)$, and study the way that long
(or heavy, if weights are present) paths grow.
This is done by relating the graphs to certain particle systems
that we call Infinite Bin Models (IBM).
A number of limit theorems are shown.
The goal of this paper is to present results along with techniques
that have been used in this area.
In the case of $\vec G(\Z, p)$ the last passage percolation constant $C(p)$
is studied in great detail. It is shown that $C(p)$ is analytic
for $p>0$, has an interesting asymptotic expansion at $p=1$ and that
$C(p)/p$ converges to $e$ like $1/(\log p)^2$ as $p \to 0$.
The paper includes the study of IBMs as models on their own as well
as their connections to stochastic models of branching processes
in continuous or discrete time with selection. Several proofs herein
are new or simplified versions of published ones.
Regenerative techniques are used where possible, exhibiting random sets
of vertices over which the graphs regenerate.
When edges have random weights we show how the last passage percolation
constants behave and when central limit theorems exist. When the underlying
vertex set is partially ordered, new phenomena occur, e.g., there
are relations with last passage Brownian percolation.
We also look at weights that may possibly take negative values
and study in detail some special cases that require combinatorial/graph
theoretic techniques that exhibit some interesting non-differentiability
properties of the last passage percolation constant.
We also explain how to approach the problem
of estimation of last passage percolation constants by means of perfect
simulation.
\end{abstract}

\hypersetup{linkcolor=Brown}
\tableofcontents

\section{Introduction}
The well-known Erd\H{o}s-R\'enyi graph
\cite{BOLbook}
admits a loopless directed version
where an edge is oriented according to an a priori order on the
set of vertices (see Figure~\ref{fig:beImg}). We call this a Barak-Erd\H{o}s graph due to
the 1984 paper \cite{BE84} by Amnon Barak and Paul Erd\H{o}s
that studied the size of the
maximal subset of vertices with the property that no two
of them are connected by a directed path and showed that it grows
like the square root of the number of vertices of the graph.
One of the most well-studied questions regarding of the  Barak-Erd\H{o}s graph
and related models is the maximum path length
or the maximum path weight if edges and vertices are given random weights.
As such, the question is closely related
to last passage percolation (LPP) problems appearing in statistical
physics dealing with maximum weight paths in random environments.

Motivations for such a quantity
come from performance evaluation of computer systems \cite{GEL86,ISONEW94},
from biology \cite{NEWCOH86,COHNEW91,CBN90} and from physics \cite{Itoh,IK12}.
In the latter field, especially in mathematical ecology, one is
interested in the survival of a species based on information of food chains.
The graphs we study are also models of food chains.
In computer networking applications, the maximum time for an information
packet to reach a destination can also be approached by random directed graphs.
Likewise, the growth of a network can also be modeled
via directed random graphs.
What is interesting is the connections of the Barak-Erd\H{o}s graph with
models in random geometry, in the theory of random matrices (originating
from the study of Schr\"{o}dinger's equation in a very complex ``random''
potential), and in statistical physics. The class of models exhibiting
convergence to Tracy-Widom distribution is growing and, in this paper,
we point out yet another one that falls in this category.
Therefore, we believe that the reader interested in theoretical
or applied research will find something of interest in this paper.
We strive to explain some connections, as above, but also present concrete
results with full proofs in most cases.

\begin{figure}[ht]
\centering
\begin{tikzpicture}
  \draw[>=latex,directed] (1,0) to[bend left] (4,0);
  \draw[>=latex,directed] (1,0) to[bend left] (7,0);
  \draw[>=latex,directed] (2,0) to[bend left] (5,0);
  \draw[>=latex,directed] (2,0) to[bend left] (6,0);
  \draw[>=latex,directed] (5,0) to[bend left] (7,0);
  \draw [thick, color = red,>=latex,directed] (1,0) to (2,0);
  \draw [thick, color = red,>=latex,directed] (2,0) to (3,0);
  \draw [thick, color = red,>=latex,directed] (3,0) to (4,0);
  \draw [thick, color = red,>=latex,directed] (4,0) to[bend left] (6,0);
  \draw [thick, color = red,>=latex,directed] (6,0) to (7,0);
  \draw [color = red] (1,0) node{$\bullet$};
  \draw [color = red] (2,0) node{$\bullet$};
  \draw [color = red] (3,0) node{$\bullet$};
  \draw [color = red] (4,0) node{$\bullet$};
  \draw [color = red] (7,0) node{$\bullet$};
  \draw [color = red] (6,0) node{$\bullet$};
  \draw (5,0) node{$\bullet$};
  \draw (1,0) node[below]{1};
  \draw (2,0) node[below]{2};
  \draw (3,0) node[below]{3};
  \draw (4,0) node[below]{4};
  \draw (5,0) node[below]{5};
  \draw (6,0) node[below]{6};
  \draw (7,0) node[below]{7};
\end{tikzpicture}
\caption{Realisation of a Barak-Erd\H{o}s graph with seven edges. For every pair $i < j$, a directed edge from $i$ to $j$ is present with probability $p$, independently from the presence of every other edge. The main objective of this survey is to describe the asymptotic properties of the longest path in this graph (here marked in red).}
\label{fig:beImg}
\end{figure}
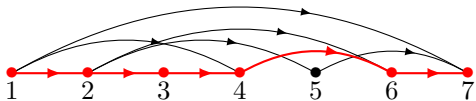

This paper offers a survey of results on the Barak-Erd\H{o}s graph
and related models.
Starting from a relatively simple static model,
we will see how it relates to discrete and continuous time
particle systems and Markov processes and, in particular,
to the Infinite Bin Model (introduced in \cite{FK03}) that has also appeared
in several papers, often in disguise \cite{ALDPIT}, and often arising as a
byproduct of other random models.
We shall also explore connections with branching processes
and random walks. In particular, we will see the emergence
of a continuous time branching random walk that is often known
as a Poisson-weighted infinite random tree \cite{AS} or Poisson
cascade model \cite{IK12}
in the statistical physics literature. The growth of the longest
path will be explained and various analytical properties of it
will be studied. In particular, we will see how the rate of convergence
relates to questions around the F-KPP equation \cite{BD09}.

We will deal with several stochastic models and notation will be introduced
little by little. For now, given an ordered (or partially ordered) set $V$,
let us define $\vec G(V,p)$ to be a random graph
on a set of vertices $V$ such that each edge $(i,j)$, where $i$ is smaller
than $j$ in the order of $V$, exists with probability $p$,
independently from edge to edge. Having said that, we shall have
the occasion to make $p$ depend on the edge and we shall discuss
situations where independence is replaced by invariance under translations.

The paper offers a survey of results aiming at exposing the main ideas.
We often (but not always) give proofs, sometimes sketches of them.
Our aim is not to provide an exhaustive bibliographical survey but rather
an exposition of results, ideas, and main proof techniques,
sometimes compromising with a simpler than a more general model.

The first part of Section \ref{sec:BEIBM} deals with a random directed graph
on $\Z$ where the edge probabilities are not even independent but, rather,
stationary and ergodic, in some sense. The aim is to show right from the start
that the maximum length of all paths from $1$ to $n$ satisfies a law of
large numbers (LLN), that is, it has a deterministic linear growth rate,
denoted by the letter $C$, and referred to as the last passage percolation
constant;
this is due to a subadditive ergodic theorem. Throwing in some
independence assumptions (but still remaining at a level more general than that
of $\vec G(\Z,p)$) shows that a regenerative structure can be obtained:
the random graph can be split into independent pieces that occur at
a computable rate. The set of the end-vertices of these
pieces is called skeleton of the graph. For $\vec G(\Z, p)$, the rate
$\lambda(p)$ of the skeleton (plotted in Figure  \ref{fig:lambda})
equals $\phi(1-p)^2$, where $\phi(x)=(1-x)(1-x^2)
\cdots$, a well-known function that bears Euler's name and has a wealth of
combinatorial and number-theoretic interpretations.

The second part of Section \ref{sec:BEIBM} explains how to
grow $\vec G(\Z,p)$ little by little.
If we let $\vec G_n$ be $\vec G (\{1,\ldots,n\},p)$
then the sequence $(\vec G_n)$ is mapped into a Markovian particle system $(X_n)$,
where $X_n$ can be thought of as a configuration of particles on $\Z$
(a balls-in-bins model) that we call Infinite Bin Model (IBM).
This was introduced in \cite{FK03} where it was shown that it converges
in distribution to a stationary state, say $X_\infty$, a particle
configuration supported on the whole of $\Z$. Paper \cite{FK03}
was mostly concerned with an extension of Borovkov's theory of
renovating events \cite{BOR78,BOR80,BORFOS92,FOS83}. This approach enabled
facilitate good explicit bounds for $C=C(p)$, the LPP constant for $\vec G(\Z,p)$.

Rather than repeating the arguments of \cite{FK03} we use the IBM derived from
$\vec G(\Z,p)$ as a motivation for the more general IBM$(\mu)$ that is
introduced in Section \ref{sec:ibm}: particles are placed on
integers so that there is always a front, that is, a largest bin after
which there are no particles. A particle is selected according to the
distribution $\mu$ and the configuration changes by
placing a daughter particle one position to its right.
This simple particle system and was introduced and studied
in detail in \cite{MR16,MR19}. The techniques and results of these papers
are exposed in Sections \ref{sec:ibm}, \ref{sec:analytic} and \ref{sec:sparse}.
A particular instance of the general IBM was considered by Aldous and
Pitman \cite{ALDPIT}.
The IBM$(\mu)$ travels to the right at asymptotic speed $v_\mu$
and this is shown in Section \ref{sec:ibm}.
If $\mu$ is geometric random variable with parameter $p$ then $v_\mu=C(p)$.
The view of the IBM in this part of the survey is that of a symbolic
dynamical system. Successive draws of integers from $\mu$ are
viewed as words from the alphabet of positive integers.
Such words can be ``$k$-coupling'' in that the content of
the $k$ rightmost bins forgets the initial configuration.
These are used to show, by coupling, the existence
of a stationary version of the IBM$(\mu)$.
By delving into the structure of specific sets
of words, new expressions for $v_\mu$ can be obtained.
Specializing to the $\mu$=geometric$(p)$ case, Section \ref{sec:analytic}
uses these expressions to obtain sharp upper and lower bounds
for $C(p)$, which are sequences rational functions that converge
to $C(p)$ uniformly over $\epsilon \le p \le 1$ for all $\epsilon >0$.
Moreover, it is shown that $C(p)$ is analytic away from $p=0$
and its power series expansion at $p=1$ has integer coefficients
that admit some combinatorial interpretation.

We then examine the behavior of $C(p)$ in a neighborhood of $0$.
One one hand, we have that $C(p)/p \to e$, as $p \to 0$,
that is $C'(0)=e$. On the other hand, $C''(0)$ does not exist.
In fact, the convergence of $C(p)/p$ to $e$ is very slow.
It was shown in \cite{MR16} that
$C(p)/p = e-\frac12 \pi^2 e (\log p)^{-2} (1+o(1))$, as $p \to 0$.
This is explained in Section \ref{sec:sparse} using somewhat different proofs.
We refer to this as ``Brunet-Derrida behavior'' as this slow convergence
phenomenon appeared in the physics literature \cite{BD09} in the following form.
Consider the classical F-KPP partial differential
equation \cite{FIS37,KPP37} arising in the modeling of reaction-diffusion
systems. This has a traveling wave solution with asymptotically constant
speed $v_\infty$, say.
An $N$-particle stochastic approximation to it is described by a certain
model that moves with constant speed $v_N$, say. It was first observed
in \cite{BD09} that $v_N \approx v_\infty - c (\log N)^{-2}$, and this was
later proved rigorously in \cite{BeG10}.
The similarity of the two results is not fortuitous. Indeed, the IBM
is compared to a branching random walk with selection, that is,
by killing particles. Results for the speed and rate of convergence
were obtained in \cite{Mal18b} and these can be used to establish
the rate of convergence of $C(p)/p\to e$.
We use the so-called Poisson-weighted infinite tree (PWIT) of Aldous and Steele
\cite{AS} which, if interpreted time-wise, is a Markovian branching process
of immortal particles that reproduce in continuous time.
We then produce a novel embedding of the IBM in the PWIT (or, rather,
a coupling between the two) which is used to obtain the rate of convergence.
In the last part of Section \ref{sec:sparse} we also take a first look at
LPP on random graphs with geometry and present, in passing, some results
on shortest paths as well for $\vec G(\{1,\ldots,n\},p_n)$.
We also note that \cite{NEWM92} proved among other results,
using branching processes,
that if $L_n$ is the maximum length of all paths in $\vec G(\{1,\ldots,n\},p_n)$
with $p_n \to 0$ and $n p_n \to \infty$, then $L_n/np_n \to e$, as $n \to \infty$,
in probability.

In Section \ref{secreg} we move on to graphs $\vec G(\Z,p_k)$ where
the probability that an edge between $i$ and $i+k$ exists equals $p_k$,
$k \in \N$. We take a closer look at the skeleton $\mathscr S$
and the regeneration properties, exhibiting a construction of elements
of $\mathscr S$ that allows us to study moment properties.
In particular, we show that the distance between successive points
has a $p$-th moment if and only if $\sum_{k=1}^\infty k^p Q_k<\infty$,
where $Q_k = (1-p_1) \cdots (1-p_k)$.
We thus obtain necessary and sufficient conditions for a
central limit theorem for the quantity $L_{n}$ in terms of the $p_k$.
In particular, a CLT always holds when the $p_k$ are identical.
The results of \ref{secreg} have been obtained in \cite{DFK12}.

Disclaimer: the term CLT (Central Limit Theorem) in this article
will refer to a limit obtained by considering deviations from
an average behavior of a random sequence, regardless of whether
the limit is Gaussian or not.

In Section \ref{sec:posets} we consider the graph $\vec G(\Z \times I, p)$,
where $I$ is a partially ordered set, say a finite set $I=\{1,\ldots,M\}$.
Then order $\Z \times I$ in component-wise fashion and place
an edge directed from $(u,i)$ to $(v,j)$ with probability $p$
if $(u,i)$ is below $(v,j)$.
More general conditions are studied in  \cite{DFK12}.
We show, in particular, that if $L_n$ is
the maximum of all paths in $[1,n] \times I$ then
a CLT holds but the limit is not Gaussian if $I$ has at least 2 points.
A functional central limit theorem for the sequence $(L_{[nt]}, t \ge_0)$
of processes establishes convergence to the Brownian LPP process
whose marginal has a distribution proportional to the largest eigenvalue
of a random $M\times M$ GUE matrix.
When $M=\infty$, we have, in particular, the graph $\vec G(\Z \times \Z, p)$.
It was shown in \cite{KT13} that a certain scaling of $L_{[nt]}$ yields
convergence, in distribution to the Tracy-Widom law $F_2$.
The proofs here are technical and we only outline the results
and refer the reader to \cite{KT13} for details.
We point out that in the finite $I$ case there is a way to obtain
a skeleton for the graph (by taking the intersection of $|I|$ skeleton sets),
whereas in the infinite $I$ case this is not possible.

Section \ref{sec:wei} takes a look at a version of
the Barak-Erd\H{o}s graph when random weights are introduced.
The material is taken from \cite{FMS14} and \cite{FK18}.
Even though negative weights on both edges and vertices can be allowed,
we focus only in the positive weights case in order to make ideas clear.
We measure the weight of a path by the sum of the weights of its edges
(and, if vertices have weights too, we add those weights as well; see
\cite{FK18}). If $u$ is a random variable representing an edge weight,
then we show that $\E u^2$ is required for the law of large numbers,
that is, the convergence of $W_n/n$, where $W_n$ is the maximum weight
of all paths from $1$ to $n$, to a constant $C=C(F)$ that depends on the
distribution $F$ of $u$.
For the CLT, we need $\E u^3 < \infty$.
When $\E u^2=\infty$ some new phenomena occur because $W_n$ grows faster than
linearly. When we put the graph on $(1/n) \Z$ we show convergence to a certain
random graph whose vertices are constructed by means of i.i.d.\ uniform
random variables.

When weights are introduced one can ask the question of the behavior
of $C(F)$. Deep properties of it have been investigated when
$F=\delta_p + (1-p) \delta_{-\infty}$.
(the case of the standard Barak-Erd\H{o}s graph)
and exposed in earlier sections. Continuity of $C(F)$ for a large set
of distributions $F$ has been investigated in a recent paper by Terlat
\cite{terlat}. In this section we focus exclusively on very simple
weight distributions with $2$ atoms: $F=p \delta_1 + (1-p) \delta_x$.
That is, every pair $(i,j)$, with $i<j$, of integers is given a weight
that has distribution $F$, independently.
What can we say about $C(p,x) \equiv C(p \delta_1 + (1-p) \delta_x)$
as a function of $x$?
We refer to this graph as ``random charged graph'' because we allow $x$
to be negative (and hence a charge rather than weight).
We still want to maximize total charge. Paths with negative charge
exist, however, $C(p,x)>0$.
The results in this section have been obtained in \cite{FKP18}
and show some interesting behavior: whereas $C(p,x)$ is a convex increasing
function of $x$, it is not everywhere differentiable.
A number of combinatorial arguments allow us to establish that $C(p,x)$ is
nondifferentiable if and only if $x$ is a negative rational or
equal to $n$ or $1/n$ for some positive integer $n \ge 2$.
Due to lack of space, the section only offers an outline of the results.

In Section \ref{sec:perfectSimu} we ask how to obtain more
information about $C(F)$ experimentally, that is, by simulation.
When $F=\delta_p + (1-p) \delta_{-\infty}$
(the standard Barak-Erd\H{o}s graph) we can employ
Markovian methods (MCMC). But we want to do better and devise a perfect
simulation method, that is, a way to perfectly (and not approximately)
simulate a random variable whose expectation is $C(F)$.
To deal with the general $F$ case, we first assume that that $F$
is supported on a semi-infinite interval $(-\infty, 1]$, say, such
that it places positive mass to any left neighborhood of $1$.
Using this assumption, we generalize the IBM particle system
to something that we call Max Growth System (MGS) that is a Markovian
process in a space of point measures (configurations of particles)
on the real line. We then construct renovation events, use them
to construct a stationary process, and then extract a random variable
that can be perfectly simulated and which has expectation $C(F)$.
Based on this, we offer a method for experimenting with various
weight distributions. We only ran simulations in a simple case, and
even present the algorithm for it.

We conclude the paper by an overview and some open problems.

\section{From the Barak-Erd\H{o}s graph to the infinite bin model}
\label{sec:BEIBM}
Consider a loopless directed graph $\vec G$ on the set $\Z$ of integers
whose edges are oriented in a way compatible with the ordering of
the integers: if $\{i,j\}$ is an edge then it is oriented from
$\min(i,j)$ to $\max(i,j)$.
Fix two integers $i, j$ such that $j-i=n > 0$.
There are four maximal quantities of interest:
\begin{equation}
\label{LR}
\begin{split}
L^{\gl, \gr}_{i,j}
& := \text{ the maximal length of all paths from $i$ to $j$ ;}
\\
L^\gl_{i,j} & :=
\text{ the maximal length of all paths from $i$ to some $j'\le j$;}
\\
L^\gr_{i,j} & :=
\text{ the maximal length of all paths from some $i' \ge i$ to $j$;}
\\
L_{i,j} & :=
\text{ the maximal length of all paths from some $i' \ge i$ to some $j' \le j$.}
\end{split}
\end{equation}
(Superscripts $\gl$, $\gr$ indicate left-tied, right-tied paths, respectively.)
Clearly, the first quantity is the smallest and the last the largest,
while the other two are in-between.

\subsection{Ergodic arguments}
\label{ergosec}
If $\vec G$ is the Barak-Erd\H{o}s graph $\vec G(\Z,p)$,
it will be seen that all these quantities satisfy the same
strong law of large numbers (SLLN).
But it is easier to see that the largest of these quantities satisfies
a SLLN, as a consequence of Kingman's
subadditive ergodic theorem \cite{Kingman}.
This has nothing to do with independence {\em per se}
and this becomes more general in the context of the following lemma.
Instead of insisting that the edge-defining random variables
are i.i.d.\ we merely assume stationarity and ergodicity.
In what follows, we shall consider a collection $\alpha$ of random variables
$\alpha_{i,j}$, indexed by pairs $(i,j)$ of integers with $i<j$,
and taking values in $\{1,-\infty\}$.
We shall then speak of the random graph $\vec G(\Z,\alpha)$
with whose set edges is
\begin{equation}
\label{GZxi}
\{(i,j): i< j, \alpha_{i,j}=1\}.
\end{equation}
Choosing $-\infty$ rather than $0$ is convenient because
if we take any sequence $i_0 < i_1 < \cdots < i_\ell$ of integers,
for some $\ell \in \N$, then the quantity
$(\alpha_{i_0, i_1} + \alpha_{i_1,i_2} + \cdots + \alpha_{i_{\ell-1},i_\ell})^+$
takes values $0$ or $\ell$;
it takes value $\ell$ if and only if $(i_0, i_1, \ldots, i_\ell)$
forms a path in $\vec G(\Z,\alpha)$.
Using this trick, we can easily express the maximal lengths \eqref{LR}
as maxima of these quantities over deterministic increasing sequences of integers.
For example,
$L^{\gl,\gr}_{1,3} = \max\{\alpha_{1,3}^+, \,(\alpha_{1,2}+\alpha_{2,3})^+\}$.
By saying that a probability measure $\P$ is defined on the canonical space
$\Omega$
we mean that $\P$ is defined on the set $\Omega$ consisting of
all collections $\alpha = (\alpha_{i,k})_{i,k \in \Z}$.

\begin{lemma}
\label{lemmix}
Let $\alpha=(\alpha_{i,j}, i< j, i,j \in \Z)$, be a collection of random variables
with values in $\{1, -\infty\}$ with distribution $\P$ on its
canonical space $\Omega$.
Define $\theta: \Omega \to \Omega$ by \footnote{Note that $\theta$
is a bijection from $\Omega$ onto itself with both $\theta$ and $\theta^{-1}$
measurable when $\Omega$ is given its natural product $\sigma$-algebra.
Let $\theta^0$ be the identity.
Then $\theta^n$, $n \in \Z$, is a group.
We say that $(\theta, \P)$ is stationary if $\P(\theta A)=\P(A)$ for all
measurable $A$. In this case, we say that it is ergodic if every set $A$
such that $\theta A = A$ a.s., actually has $\P(A)$ equal to $0$ or $1$.}
\begin{equation}
\label{thetadef}
(\theta \alpha)_{i,j} = \alpha_{i+1, j+1}.
\end{equation}
Assume that $(\theta, \P)$ is stationary and ergodic.
Let
\[
L_{i,j} := \max_{\substack{i\le i_0<i_1<\cdots < i_\ell\le j\\\ell \in \N}}
(\alpha_{i_0, i_1} + \alpha_{i_1,i_2} + \cdots + \alpha_{i_{\ell-1},i_\ell})^+.
\]
Then there is a deterministic $C$ such that
\[
C=\lim_{n\to \infty}L_{0,n}/n  \quad \text{as $n \to \infty$ $\P$-a.s.\ and
in $L^1$},\quad C=\inf_n \E L_{0,n}/{n}.
\]
\end{lemma}
\begin{proof}
Noticing that $L_{i,j}$ is the maximum length of all paths in $\vec G(\Z,\alpha)$
with endpoints between $i$ and $j$ (consistent with the last of \eqref{LR})
we have
\[
L_{i,k} \le L_{i,j} + L_{j,k}+1, \quad i<j<k,
\]
for if we consider a maximum length path between two vertices on $[i,k]$
then its length is at most the length of its restriction on
$[i,j]$ plus the length of its restriction on $[j,k]$ plus 1 if $j$
is not a vertex of the maximum length path.
The stationarity and ergodicity of $(\theta, \P)$ together
with the last inequality shows that the $L_{i,j}+1$ satisfy the assumptions of
Kingman's subadditive ergodic theorem \cite{Kingman} and so $\lim_{n \to \infty}
L_{i,i+n}/n$ exists $P$-a.s.\ and in $L^1$ and equals $C=\inf_n \E L_{0,n}/n$.
\end{proof}

\begin{remark}
It will turn out that all four quantities in \eqref{LR} have the same growth
rate as the largest of them. This is not entirely obvious at this moment
because, for example, attempting to establish that
$\lim_{n \to \infty} L^{\gl,\gr}_{0,n}/n$ exists a.s., one might be tempted to
use the obvious superadditivity
\[
L^{\gl,\gr}_{i,k} \geq L^{\gl,\gr}_{i,j} + L^{\gl,\gr}_{j,k},
\]
But, according to the extension of the subadditive ergodic theorem
of Liggett, see \cite[Theorem 2.6]{LIGG}
this would require that $\E L_{i,j}^- < \infty$ which is false here.
The fact that $C^{\gl,\gr} = C^{\gl} = C^{\gr} =C$ is discussed below;
see Corollary \ref{aLL}.
\end{remark}

Returning to the graph $\vec G(\Z, \alpha)$, whose edge set
is as in \eqref{GZxi},  let us define
\[
i \leadsto j \iff
\text{ there is a path in $\vec G(\Z, \alpha)$ from $i$ to $j$}
\]
and identify a certain random subset of $\Z$, that we shall refer to as
the {\em skeleton of the graph}, as follows.
For each $j \in \Z$ let
\begin{equation}
\label{AAA}
A_j=\{\alpha \in \Omega:\, \text{for all $i \in \Z$
there is a path in $\vec G(\Z, \alpha)$ from $\min(i,j)$ to $\max(i,j)$}\}.
\end{equation}
The {\em skeleton} $\SSS$ is the random set of all $j$ such that $A_j$ occurs:
\begin{equation}
\label{skeleton}
\SSS=\SSS(\alpha) = \{j \in \Z:\, \alpha \in A_j\}.
\end{equation}
The elements of $\SSS$ are called {\em skeleton points} or skeleton
vertices of the graph $\vec G(\Z,\alpha)$.
Notice that $\theta A_j = A_{j+1}$ for all $j \in \Z$.
Hence, if $(\theta, \P)$ is stationary we have
$\P(A_0)=\P(A_n)$ for all $n \in \Z$
and the random sets $\SSS\comp \theta^n$ have all the same law.
\begin{definition}[rate of skeleton]
Assume that $(\theta, \P)$ is stationary. Then the quantity
\begin{equation}
\label{SKELRATE}
\lambda:=\P(A_0)
\end{equation}
is referred to as the {\em rate} or {\em density} of the skeleton $\mathscr S$.
\end{definition}

\begin{lemma}
\label{lemoll}
Assume that $(\theta, \P)$ is stationary and ergodic.
Then
$\SSS^+:=\SSS \cap [0,\infty)$ and $\SSS^-:=\SSS \cap (-\infty,0]$ are both infinite
sets $\P$-a.s. if and only if $\lambda >0$.
Moreover, conditional on $A_0$, the expected distance
between two successive elements of $\SSS$ is $1/\lambda$.
\end{lemma}
\begin{proof}[Sketch of proof.]
The first claim is due to the Poincar\'e recurrence theorem \cite{Caratheodory}.
The second claim is from basic properties of stationary point processes.
\end{proof}

\begin{remark}
\label{remrem}
If $\SSS^+$ and $\SSS^-$ are both infinite
then any two far apart vertices $i$ and $j$ will contain a skeleton
point between them. This implies that the there is at
least one path from $i$ to $j$ (and this path passes through
the skeleton point).
\end{remark}

The following is taken from \cite{DFK12}.
\begin{lemma}
\label{lela}
Consider $\vec G(\Z,\alpha)$ and assume that
$\alpha_{i,j}$, $i<j$, $i, j \in \Z$,
are all independent with
\[
\P(\alpha_{i,j}=1)=p_{j-i},
\]
where $p_n$, $n \in \N$, is a sequence of probabilities
such that
\begin{equation}
\label{summa}
\sum_{n=1}^\infty (1-p_1)\cdots(1-p_n)<\infty.
\end{equation}
Then the rate $\lambda$, defined by \eqref{SKELRATE},
of the skeleton $\mathscr S$ is positive and given by
\begin{equation} \label{lprod}
\lambda = \prod_{j>0} (1-(1-p_1) \cdots (1-p_j))^2.
\end{equation}
\end{lemma}
\begin{proof}
The independence assumption implies that $(\theta, \P)$ is stationary
and ergodic, where $\theta$ is as in  \eqref{thetadef}.
We will argue that the summability assumption \eqref{summa} implies that
$\lambda >0$ which, by Lemma \ref{lemoll}, will imply that
$\SSS^-$ is an infinite set.
Consider the random variables
\begin{align*}
\ell(j) &:= \max\{k>0:\, \alpha_{j-k,j}=1\}, \quad j \in\Z,
\end{align*}
which have the same distribution:
\[
\P(\ell(j) > k) = (1-p_1) \cdots (1-p_k).
\]
Condition \eqref{summa} implies that $\ell(j)<\infty$ a.s.
Consider also the event
\begin{equation*}
\label{A+}
A^+_{0,m} = \{0 \leadsto j, \text{ for all } j=1,\ldots,m\},
\end{equation*}
noticing that
\[
A^+_{0,m} = \bigcap_{j=1}^m \bigcup_{i=0}^{j-1} \{i \leadsto j\}
= \bigcap_{j=1}^m \{\ell(j) \le j\}.
\]
Therefore,
\[
A^+_0 := \{0 \leadsto 1, 0 \leadsto 2, \ldots\}
=  \bigcap_{j=1}^\infty \{\ell(j) \le j\},
\]
and so
\[
\P(A_0^+) = \prod_{j=1}^\infty (1-(1-p_1)\cdots(1-p_j))>0.
\]
Similarly,
\[
A_0^-:=\{-1 \leadsto 0, -2 \leadsto 0, -3 \leadsto 0, \ldots\}
\]
has the same probability as $A_0^+$. Noticing that the event
$A_0$, defined by \eqref{AAA}, is the intersection of $A_0^+$ and $A_0^-$,
two independent events, we obtain
\[
\lambda = \P(A_0) =  \P(A_0^+)\P(A_0^-)
= \prod_{j>0} (1-(1-p_1) \cdots (1-p_j))^2.
\]
\end{proof}

\begin{remark}
If one of the $p_j$ equals $1$ then letting $k=\min\{j: p_j=1\}$
we have $\lambda = \prod_{j=1}^{k-1} (1-(1-p_1) \cdots (1-p_j))^2$.
The case $p_1=1$ is uninteresting.
\end{remark}

Combining all of the above we conclude that the length of longest paths in
Barak-Erd\H{o}s graphs grows linearly, a result first observed by Newman
\cite{NEWM92}.

\begin{corollary}
\label{aLL}
Consider the four quantities defined by \eqref{LR} for a Barak-Erd\H{o}s
graph $\vec G(\Z,p)$.
Then there is a constant $C=C(p)$ such that
\[
\lim_{n\to \infty} \frac{L_{0,n}}{n}
= \lim_{n\to \infty} \frac{L^\gl_{0,n}}{n}
= \lim_{n\to \infty} \frac{L^\gr_{0,n}}{n}
= \lim_{n\to \infty} \frac{L^{\gl,\gr}_{0,n}}{n} = C(p) \text{ a.s.}
\]
\end{corollary}
\begin{proof}[Sketch of proof]
If $p=0$ then the graph has no edges and the above limits
hold trivially with $C(0)=0$.
Assume $p>0$ and
note that condition \eqref{summa} of Lemma \ref{lela} holds because
\eqref{summa} holds: $\sum_{n=1}^\infty (1-p)^n < \infty$.
We thus have $\lambda>0$.
By Lemma \ref{lemoll}, the random sets
$\SSS^+$, $\SSS^-$ are a.s.\ infinite with positive rate $\lambda$.
We can easily see that $L_{0,n}=L^{\gl,\gr}_{0,n}+o(n)$,
as $n \to \infty$, a.s.
The conclusion now follows from Lemma \ref{lemmix}.
\end{proof}

\begin{remark}
\label{elleuler}
For a $\vec G(\Z,p)$ with $0 \le p \le 1$, we have that the rate of
its skeleton is given by
\begin{equation}
\label{ellq}
\lambda = \prod_{j=1}^\infty (1-(1-p)^j)^2
\end{equation}
This follows from \eqref{lprod}.
We now give a number-theoretic interpretation.
Consider Euler's function
\[
\phi(q) := \prod_{k=1}^\infty (1-q^k), \quad |q|<1.
\]
Clearly,
\[
\lambda(p) = \phi(1-p)^2.
\]
It is easy to see that $1/\phi(q)$ is the generating function
of the sequence $p(n)$ of integer partitions of the positive integer $n$,
that is,
\[
\sum_{n=1}^\infty p(n) q^n = \frac{1}{\phi(q)}.
\]
To see this, recall that $p(n)$ is defined as the number of ways to write
$n= \ell_1+2\ell_2+3\ell_3+\cdots$, where the $\ell_i$ are nonnegative integers.
So
$\sum_{n=1}^\infty q^n p(n) = \sum_{n=1}^\infty q^n \sum_{\ell_1, \ell_2,
\ldots} \1_{n=\ell_1+2\ell_2+\cdots}
= \sum_{\ell_1} q^{\ell_1} \sum_{\ell_2} q^{2\ell_2} \cdots
= (1-q)^{-1} (1-q^2)^{-2} \cdots=1/\phi(q)$.
Euler's pentagonal number theorem relates Euler's function to pentagonal
numbers, that is numbers of the form $(3n^2-n)/2$ (pentagonal
numbers are ``Pythagorean'' numbers in the sense that they
can be represented using pentagons, analogously to triangular and square
numbers that were actually known by Pythagoras).
The theorem says that
\[
\phi(q)
= \sum_{n =-\infty}^\infty (-1)^n q^{\frac{3n^2-n}{2}}, \quad |q|<1.
\]
A beautiful bijective proof of this is due to Franklin (1881) \cite{FRA1881};
see Andrews \cite{AND} for a more modern account.
Other algebraic proofs are due to Jacobi, Euler, and others;
see P\'olya and Szeg\H{o} \cite[$\mathsection$4, 50-54]{PS78} for these proofs.
\end{remark}

\begin{remark}
It is interesting to see that for a sparse $\vec G(\Z,p)$ graph, the average
distance between two successive skeleton points is huge,
whereas for a dense graph every second point is a skeleton point.
For the symmetric $p=0.5$ case, roughly every $12$th point
is a skeleton point.
\begin{table}[h]
\centering
\caption{Average distance between two successive skeleton points in the infinite
Barak-Erd\H{o}s graph with parameter $p$.}
\setlength\extrarowheight{2pt}
\begin{tabular}{|c|ccccccc|}
\hline
$p$ & $0.01$ & $0.1$ & $0.3$ & $0.5$ & $0.6$ & $0.8$ & $0.9$  \\
\hline
$1/\lambda$ &$10^{139}$ &$6\times10^{11}$ & $558.46$ &$11.99$ &$4.9$ &$1.73$ &$1.26$\\
\hline
\end{tabular}
\end{table}
\begin{figure}[ht]
\centering
\includegraphics[height=4.5cm]{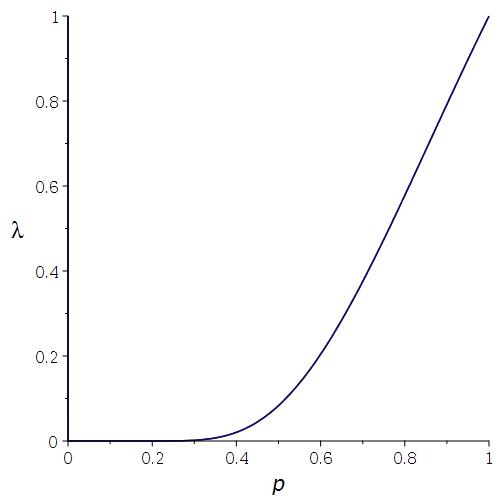}
\caption{Plot of the rate $\lambda(p)$ of skeleton points against
the connectivity probability $p$}
\label{fig:lambda}
\end{figure}
We used the  formula $\lambda(p) = \left(\sum_{n=-\infty}^\infty
(-1)^n (1-p)^{\frac{3n^2-n}{2}}\right)^2$ to perform these computations,
since the series converges much faster than the product.
This, together with the regenerative properties (Section \ref{secreg})
provides a method for constructing an accurate picture of $\vec G(\Z,p)$.
\end{remark}

\subsection{The infinite bin model corresponding to the Barak-Erd\H{o}s graph.}
\label{BEIBM}
The general Infinite Bin Model (IBM) is a particle system that will
be introduced in Section \ref{sec:ibm}. In the current section,
we will motivate the need to study it by explaining how to obtain
an IBM by growing a Barak-Erd\H{o}s $\vec G(\Z, p)$ graph dynamically.
For the construction, we shall keep in mind that we are interested
in longest paths.

Suppose that we have created
\[
\vec G_n=\vec G(\{1,\ldots,n\},p).
\]
To construct $\vec G_{n+1}$
we need to add all edges $(i,n+1)$ for which $\alpha_{i,n+1}=1$.
Conditional on $\vec G_n$,
the distribution of $\vec G_{n+1}$
will not change if we permute the variables
$\alpha_{1,n+1}, \ldots, \alpha_{n,n+1}$.
We choose to order the variables $\alpha_{1,n+1}$, $i=1,\ldots,n$
by ordering the vertices $i=1,\ldots,n$, according to decreasing
values of $(L^{\gr}_{1,i})_{1 \le i \le n}$:
\[
j \preceq i \iff L^{\gr}_{1,j} \ge L^{\gr}_{1,i},
\]
with ties resolved arbitrarily.
Introduce a state (or configuration) vector $X_n=(X_n(0), X_n(1),\ldots)$
for $\vec G(\{1,\ldots,n\},p)$ by letting
\[
  X_n(\ell) := |\{1\le j \le n:\, L^{\gr}_{1,j}=\ell\}|.
\]
In particular, $X_n(0)$ is the number of vertices $j$ for which none
of the edges $(i,j)$ exist, for any $1\le i < j$, and $X_n(1)$ is
the number of vertices $j$ such that there exists an edge $(i,j)$ for some $i$
that is counted in $X_n(0)$ and there are no other incoming edges to $j$.
We now let
\[
L_n :=  \max_{1\le j \le n} L^{\gr}_{1,j},
\]
for notational convenience.
Hence $X_n(L_n)$ is the number of vertices $1\le i\le n$
such that $L^{\gr}_{1,i}$ is maximal.
For technical reasons, we shall extend $X_n(\ell)$ on
negative integers $\ell$ too and let $X_n(\ell)=\infty$ if $\ell<0$.
Therefore, our state vector is of the form
\begin{equation}
\label{Xform}
X_n=[\ldots, \infty, \infty, X_n(0), X_n(1), \ldots, X_n(L_n), 0,0,\ldots].
\end{equation}
Note that $X_1 = [\ldots, \infty, \infty, 1, 0,0,\ldots]$, that is, $X_1(0)=1$,
$X_1(\ell)=\infty$ for $\ell <0$, and $X_1(\ell)=0$ for $\ell > 0$. We observe
that the sequence $(X_n)$ of state vectors is a Markov chain.

Indeed, changing the point of view, we shall think of $X_n$
as a configuration of a number of identical balls (corresponding to the
vertices) into labeled bins,
so that $X_n(k)$ is the number of balls in bin $k$.
The Markovian evolution is then easily obtained.
When we add a new vertex $n+1$ then the state will change by the addition
of a new ball into a bin.
Recalling the ordering $\prec$ of the vertices
of $\vec G_n$, we remark that the balls are placed in the bin in an
increasing fashion. In other words, for all $i < j$, if the ball corresponding
to vertex $i$ is in a bin to the left of the ball corresponding to vertex $j$, then
$i \prec j$.

To construct $X_{n+1}$ we only need to monitor the largest vertex $i$, for that
order, such that $\alpha_{i,n+1}=1$.
We then add a new ball to the bin immediately to the right of the bin
containing ball $i$.
If such an $i$ does not exist we are in the situation that vertex
$n+1$ is not the endpoint of a path starting from vertex $1$
and, in this case, a ball is added in bin $0$:
$X_{n+1}(0)= X_n(0)+1$.
Thus, first let $B_n$ be the nonnegative integer uniquely specified by
\begin{equation}
\label{binning}
X_n(B_n+1) + \cdots + X_n(L_n) \le I < X_n(B_n)+X_n(B_n+1) + \cdots + X_n(L_n),
\end{equation}
where $I$ is the rank of the largest vertex $i$ for $\prec$ such that
$\alpha_{i,n+1} = 1$, or $I = n+1$ if there is no such vertex.
Then, we construct $X_{n+1}$ as $X_n + \delta_{B_n}$, where $\delta_{k}$
is the configuration with a single ball in bin $k$, so that $\delta_k(\ell) =
\ind{k=\ell}$.
Noting the form of the state \eqref{Xform} we see that such a $B_n$
always exists and, because negative bins contain an infinite number
of balls we have $B_n \ge 0$.

We can equivalently describe the transition from $X_n$ to $X_{n+1}$
by the stochastic recursion
\begin{equation}
\label{begbin}
X_{n+1} = X_n + \delta_{B(X_n,\xi_{n+1})+1},
\end{equation}
where $\xi_{n+1}$ is a geometric random variable with parameter $p$,
i.e.\, $\P(\xi_{n+1}=i)=(1-p)^{i-1}p$, $i \in \N$,
independent of $X_n$ and
\begin{equation}
  \label{eqn:firstB}
  B(X_n,z) = \inf\left\{k \in \Z : \sum_{p = k+1}^\infty X_n(p) < z\right\},
\end{equation}
i.e.\ $B(X_n, \xi_{n+1})$ is the $B_n$ satisfying \eqref{binning} with
$\xi_{n+1}$ in place of $I$. The stochastic recursion \eqref{begbin} then
proves that $(X_n)$ is a Markov chain.

The process $(X_n)$ is a particular case of an {\em infinite bin model}
that will be studied in more detail in the next section.
This bin model was introduced in \cite{FK03} even under more
general stationary and ergodic assumptions. It was shown that
a stationary version of a spatially-shifted version of it exists.
Under independence assumptions, it was possible to write balance
equations for the stationary version. These led
to sharp bounds on $C(p)$ showing, in particular, the
asymptotics of $C(p)$ as $p \to 0$ that had previously obtained
by Newman \cite{NEWM92}.  Moreover, it became possible to obtain analytically
expressible upper and lower bounds for the whole function $p\mapsto C(p)$.

We study in the next section the infinite bin model as a stochastic process per se, dropping the geometric distribution assumption. This allows us to obtain general formula for the speed at which the index of the rightmost occupied bin is growing. Specifying these formulas for the geometric distribution will allow us to obtain the precise analytic properties of the function $C$ in Section~\ref{sec:analytic}.

\section{The general infinite bin model}
\label{sec:ibm}
We consider in this section the infinite bin model as a ``ball in bins'' process. It can be constructed as a Markov chain in which at each step, a new ball is added in one of the bins similarly to the process defined in Section~\ref{BEIBM}, but with an arbitrary law for the placement of the new ball. This model has appeared in different forms in several areas of probability. Among others, Aldous and Pitman \cite{ALDPIT} took interest in an infinite bin model in which at each step, a new ball is added to the right of one of the $N$ rightmost balls, chosen uniformly at random. The present general setting was introduced by Foss and Konstantopoulos \cite{FK03}, however in this article a stationary version of the process (defined in Section~\ref{subsec:stationaryversion}) is considered.

This section is organized as follows: we first introduce a formal definition of
the generalized infinite bin model in Section~\ref{subsec:defIBM} as a Markov
chain on the space of functions $\Z \to \Z_+$ with support of the form $\{n \leq
a\}$ for some $a \in \Z$. In Section~\ref{subsec:speed}, we show that the index
of the rightmost occupied bin in an infinite bin model grows linearly over time,
at a certain speed $v$. The rest of the section is devoted to various ways to
compute this constant.

We introduce the notion of coupling words in Section~\ref{subsec:couplingwords},
which allows us to introduce renewal events in the evolution of infinite bin
model. Thanks to these renewal events, we can define a stationary version of the
infinite bin model in Section~\ref{subsec:stationaryversion}, which allows us to
obtain several analytic formulas for its speed in Section~\ref{speedexpress}.

\subsection{Definition and first properties of the infinite bin model}
\label{subsec:defIBM}

In order to give a general description of infinite bin models, we first describe
the state space on which this Markov chain will evolve.
A configuration (or state) $X$ is a map $k \mapsto X(k)$
from $\Z$ (the set of bins) into
\[
\overline\Z_+ := \{0,1,2,\ldots\} \cup \{\infty\}
\]
such that $X(k)=0$ eventually.
We let
\[
\mathbb S := \left\{X \in \overline{\Z}_+^\Z: \text{ there is $f \in \Z$
such that ($X(k)=0 \text{ iff } k > f$) }\right\}
\]
be the set of configurations.
We think of $k \in \Z$ as a bin and of $X(k)$ as a number of
indistinguishable balls placed in this bin.
Given $X \in \mathbb S$ we let
\[
F(X)=\max\left\{ k\in\Z: X(k)>0\right\},
\]
a quantity called {\em the front} (bin) of $X$.
Thus, each bin contains some balls (either no ball or a
positive finite number of balls or an infinite number of balls)
such that every bin to the right of $F(X)$ is empty
and every bin to the left of $F(X)$ is nonempty.
\footnote{This last assumption may be relaxed, allowing empty bins to the left
of the front, in which case the proofs become more technical, with some
absorbing states being created. However, the main results
stated in this section still hold true under quite general conditions.}.

\paragraph{System dynamics.}
Consider a configuration $X$ and a positive integer $\xi$ that we will
refer to as the {\em selection number}.
The rightmost nonempty bin is $F(X)$. Enumerate the balls in the
nonempty bins of $X$ starting from the right and moving to the left,
select the $\xi$-th ball, and let
$B(X,\xi)$ be the bin containing it.
The next state is obtained by simply adding a single ball to the
bin to its right, indexed $B(X,\xi)+1$.

\begin{figure}[ht]
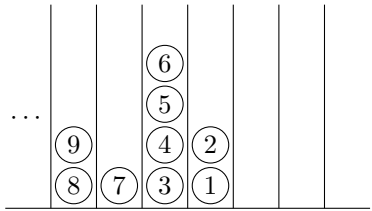

\centering
\include{figTex/configuration}
\caption{Representation of the state $X = [\ldots, 2,1,4,2, 0,0, \ldots]$ in
terms of balls in bins. Balls are enumerated from right to left.}
\label{fig:configuration}
\end{figure}

\begin{example}
Let $X=[\ldots, 2,1,4,2, 0,0, \ldots]$ be the configuration pictured in Figure
\ref{fig:configuration}.
If $\xi=1$ then the next state is $X=[\ldots, 2,1,4,2, 1,0, \ldots]$.
If $\xi=3$ then the bin that contains the third ball from the right
is the second bin from the right, so the next state is
$X=[\ldots, 2,1,4,3, 0,0, \ldots]$. The new ball will be placed in the second
bin from the right if $\xi = 7$, in the third bin from the right if $\xi \in
\{8,9\}$, etc.
\end{example}

For technical reasons, we will also allow for the possibility that $\xi = 0$ or
$\xi = \infty$. If $\xi = \infty$, the bin that contains the $\xi$-th ball is
formally at $-\infty$, so no ball is added. If $\xi = 0$, the new
configuration is obtained by shifting the position of every ball
in the configuration one step
to its right, i.e.\ replacing $X$ by $k \mapsto X(k-1)$.

In symbols, we let
\begin{equation}
\label{Bdef}
B(X,\xi):=\inf\bigg\{k\in\Z:\, \sum\limits_{j>k} X(j)<\xi\bigg\},
\end{equation}
writing $B(X,\infty) = -\infty$ by convention. This definition generalizes the one given in \eqref{eqn:firstB}. Then, for $\xi \in \N \cup \{\infty\}$, we define $\Phi_\xi : \configS \to
\configS$ by
\begin{equation}
\label{Phidef}
\Phi_\xi(X) := X + \delta_{B(X,\xi)+1},
\end{equation}
similarly to \eqref{begbin}, where $\delta_k$ is the element of $\Z_+^\Z$ with
$\delta_k(i)=1$ if $i=k$ and $0$ otherwise, so that $\delta_{-\infty} \equiv
0$ (in particular $\Phi_{\infty}(X) = X$). We also define $\Phi_0$ via
\[
  \Phi_0(X) (k) := X(k-1), \quad k \in \Z,
\]
corresponding to shifting $X$ by a unit step to the right.
We notice that $B=B(X,\xi)$ is uniquely specified by the inequality
\[
X(B+1)+\cdots+X(F) < \xi \le X(B)+X(B+1)+\cdots+X(F), \text{ where $F=F(X)$.}
\]
We also have, for all $1 \le \xi \le \infty$,
\[
F(X)-\xi+1 \le B(X,\xi) \le F(X)=B(X,1).
\]

\begin{definition}[infinite bin model]
\label{def31}
Given a probability measure $\mu$ on $\overline\Z_+ =\{0,1,2,\ldots,\infty\}$,
and an i.i.d.\ random sequence $(\xi_n)_{n \ge 1}$ with common law $\mu$,
we define IBM($\mu$) to be the Markov process $(X_n)$ with values
in $\mathbb S$ given by
\[
X_{n+1} = \Phi_{\xi_{n+1}}(X_n).
\]
We leave the initial configuration $X_0$ unspecified.
\end{definition}

\begin{remark}
In Section \ref{subsec:couplingwords} we will write $\Phi_\xi(X)$ simply as
$\xi X$ or $\xi(X)$ for reasons that will become apparent there.
Keeping this in mind,
we refer the reader to the Example \ref{weex} below.
\end{remark}

Note that $X_0$ does not reflect in the notation IBM($\mu$). This is
justified by the fact that the asymptotic properties of the IBM that we are
interested in are independent of the choice of $X_0$.
In relation with the
Barak-Erd\H{o}s graphs, we will be mostly interested in the case when $\mu$ is a
geometric random variable on $\N$, but it will be useful to consider more
general measures on $\overline\Z_+$ in order to obtain estimates
on some quantities of interest. However, as several lemmas are easier to prove
under the assumption that $\mu$ is supported on $\N$,
let us first remark that this condition
is usually enough to study the asymptotic properties of the infinite bin model.

\begin{lemma}
\label{lem:connectionIbm}
Let $\mu$ be a probability distribution on $\bar{\Z}_+$ with $\mu(\N)>0$. We
denote by $\nu$ the law $\mu$ conditioned to be on $\N$, i.e.\
\[
\forall k \in \N, \quad \nu(\{k\}) = \frac{\mu(\{k\})}{\mu(\N)}.
\]
There exists a coupling between the IBM($\mu$) $(X_n)$ and a couple
$(Y_n,(A_n,B_n))$, where $(Y_n)$ is an IBM($\nu$) started from $X_0$ and
$(A_n,B_n)$ is an independent $\Z^2$-valued random walk with step distribution
\begin{multline*}
\P(A_{n+1} = A_n+1,B_{n+1}=B_n) = \mu(0), \text{ } \P(A_{n+1}=A_n,B_{n+1}
= B_n+1) = \mu(\infty)
\\
\text{ and } \P(A_{n+1}=A_n,B_{n+1}=B_n) = \mu(\N)
\end{multline*}
such that
\[
\forall n \in \Z_+,  \forall k \in \Z,
X_n(k) = Y_{n-A_n-B_n}(k-A_n) \quad \text{a.s.}
\]
\end{lemma}

\begin{proof}
Letting $(\xi_n, n \geq 0)$ be a sequence of i.i.d.\
random variables with law $\mu$, we set
\[
A_n = \sum_{k=1}^n \ind{\xi_k = 0} \quad \text{and} \quad
B_n = \sum_{k=1}^n \ind{\xi_{k} = \infty},
\]
and observe that $(A_n,B_n)$ is the same random walk as defined in the lemma. Then, we relabel $(\zeta_n)$ the random sequence $(\xi_n)$ removing all
terms equal to $0$ or $\infty$, in increasing order of their indices. We observe
that $(\zeta_n)$ is a sequence of i.i.d.\ random variables with law $\nu$, independent from $(A_n,B_n)$.

We define $X_n$ and $Y_n$ by setting $Y_0 = X_0$ and
\[
X_{n+1} = \Phi_{\xi_{n+1}}(X_n) = \Phi_{\xi_{n+1}} \circ \cdots \circ
\Phi_{\xi_1} (X_0) \quad \text{and} \quad Y_{n+1} = \Phi_{\zeta_{n+1}}(Y_n) =
\Phi_{\zeta_{n+1}} \circ \cdots \circ \Phi_{\zeta_1} (X_0).
\]
Observing that $\Phi_0$ and $\Phi_\infty$ both commute with all $\Phi_k$ for
$k \in \N$, we can rewrite
\[
X_{n} = \Phi_0^{A_n} \Phi_{\zeta_{n-A_n-B_n}} \circ \cdots \circ
\Phi_{\zeta_1}(X_0) = Y_{n-A_n-B_n}(k-A_n),
\]
using that there are exactly $A_n$ elements of $(\xi_1,\ldots, \xi_n)$ equal to
$0$ and $B_n$ equal to $-\infty$, the rest being given, in the same order, by
$(\zeta_1,\ldots,\zeta_{n-A_n-B_n})$, which completes the proof.
\end{proof}

\begin{theorem}[Speed of the IBM [\cite{FK03,MR16}]
\label{thm:speed}
Let $\mu$ be a probability measure on $\bar{\Z}_+$.
Let $(X_n)$ be an IBM($\mu$) with initial configuration $X_0$.
Then there exists a constant $0 \le v_\mu\le 1$, not dependent of $X_0$,
such that
\[
\lim_{n\rightarrow\infty} \frac{F(X_n)}{n}=v_\mu \quad a.s.
\]
\end{theorem}

The quantity $v_\mu$ is called the speed of the IBM($\mu$). This theorem is
proved in Section \ref{subsec:speed} by bounding the IBM($\mu$) by two
sequences of IBMs with laws having finite supports and using an increasing
coupling between these processes.

\begin{remark}
Applying Lemma~\ref{lem:connectionIbm}, we observe that if $\mu$ is a
probability distribution on $\bar{\Z}_+$ with $\mu(\N) > 0$, then setting $\nu =
\mu(\cdot|\N)$, we have (by law of large numbers)
\[
v_\mu = \mu(0) + \mu(\N) v_\nu.
\]
In particular, it is enough to prove Theorem~\ref{thm:speed} for measures
supported by $\N$.
\end{remark}

Beyond the position of the front, we will generally be interested in the content
of a finite number of bins at a fixed distance from the front. In the case of an
IBM($\mu$) where $\mu$ has finite support, one can reduce the study of the IBM
to a finite state space Markov chain having a stationary distribution, by
considering some finite-dimensional projection of the process, see
Section \ref{subsec:speed}. For a general $\mu$ the content of the rightmost
$K$ non-empty bins also has a stationary distribution but the arguments are more
involved, see Section \ref{subsec:stationaryversion}.

\begin{definition}[partial order on $\configS$]
For any $X,Y\in \configS$, we set $X\preceq Y$ if for every
$\xi \in \N$, $B(X,\xi)\leq B(Y,\xi)$, that is, if for every
$\xi$ the $\xi$-th ball of $X$ is to the left of the $\xi$-th ball of $Y$.
\end{definition}

\begin{lemma}
The relation $\preceq$ is a partial order that is preserved by addition.
Moreover,
\begin{equation}
\label{eq:monotonicity}
\text{ if } 0\le \xi \le \xi' \le \infty \text{ and } X \preceq Y
\text{ then }\Phi_{\xi'}(X)\preceq \Phi_{\xi}(Y).
\end{equation}
\end{lemma}
\begin{proof}
Simply notice that
\[
X\preceq Y \iff
\sum_{k=\ell}^\infty X(k)
\le \sum_{k=\ell}^\infty Y(k)
\text{ for all } \ell \in \Z.
\]
This implies that $X\preceq Y$ and $X'\preceq Y'$ implies $X+X' \preceq Y+Y'$.

For the second assertion, assume first that $1 \le \xi \le \xi'\le \infty$.
We then have, by \eqref{Bdef}, $B(X,\xi') \le B(X, \xi)$ and so
$\delta_{B(X, \xi')+1} \preceq \delta_{B(X, \xi)+1}$;
and if $X \preceq Y$ then
$X + \delta_{B(X, \xi')+1} \preceq Y+\delta_{B(X, \xi)+1}$,
and so $\Phi_{\xi'}(X) \preceq \Phi_{\xi}(Y)$ by \eqref{Phidef}.

The case $\xi=\xi'=0$ being straightforward, we are left with the case $\xi=0 <
\xi' \le \infty$. Assume again $X \preceq Y$. Then, by the argument above, and
since $\xi' \ge 1$,
we have $\Phi_{\xi'}(X) \le \Phi_{\xi'} (Y) \le \Phi_1(Y)$ and we can
easily see that $\Phi_1(Y) \le \Phi_0(Y)$.
\end{proof}

This partial order can be used to define an
increasing coupling between two IBMs when the step distribution of the first IBM
is dominated by the step distribution of the second IBM.

\begin{proposition}[increasing coupling]
\label{prop:coupling}
Let $\mu$ and $\nu$ be two probability measures on $\overline{\Z}_+$ such that
for
every $i\geq0$ we have $\mu([0,i])\leq\nu([0,i])$. Then if $X_0 \preceq Y_0$ are
two configurations in $S$, we can construct a coupling of $(X_n) \sim
\text{IBM}(\mu)$ and $(Y_n) \sim \text{IBM}(\nu)$ such that $X_n \preceq Y_n$
for every $n\geq0$ a.s.
\end{proposition}

\begin{proof}
The assumption that $\mu([0,i])\leq\nu([0,i])$ for all $i$
allows us to define random $\xi, \xi'$, with laws $\mu, \nu$,
respectively, such that $\xi' \le \xi$ a.s.
Hence we can build an i.i.d.\ sequence of pairs
$(\xi_n,\xi'_n)_{n\geq1}$ such that $\xi_n$
has law $\mu$, $\xi'_n$ has law $\nu$ and $\xi'_n\leq\xi_n$ for all $n\geq1$.
Defining the IBM($\mu$) $(X_n)$ using $(\xi_n)$ and the IBM($\nu$) $(Y_n)$ using
$(\xi'_n)$, we obtain that $X_n \preceq Y_n$ for every $n\geq0$ a.s.\ by
applying
\eqref{eq:monotonicity} inductively.
\end{proof}

\subsection{The speed of the general infinite bin model}
\label{subsec:speed}
This section is devoted to the proof of Theorem \ref{thm:speed}
for the existence of the speed $v_\mu$ for IBM$(\mu)$.
We first examine the case where $\mu$ has finite support and
then develop a coupling technique to deal with the general case.

\paragraph{The finite support case.}
Let $\xi$ be a random element of $\bar{\Z}_+$, distributed according to the law $\mu$. Assume there is an integer $k$ such that
\begin{equation}
\label{eqn:kSupport}
\xi \in \{0,1,\ldots,k\} \cup \{\infty\} \text{ a.s.}
\end{equation}
If $k=1$ then, by the definition of $\Phi_\xi$, the front of $\Phi_\xi(X)$ is one unit to the front of $X$ iff
$\xi \in\{0,1\}$. Hence, in this case,
$F(X_n)-F(X_0)=\sum_{m=1}^n \1_{\xi_m \in \{0,1\}}$
and so $v_\mu=\mu(0)+\mu(1)=1-\mu(\infty)$.

We now assume in the rest of the section that
\[
\text{there exists } k \ge 2 \text{ such that } \mu(k)>0.
\]
Given a configuration $X$ define
\begin{equation}
\label{PPP}
\widetilde\Pi_k(X) := \big[ X(B(X,k)+1),\ldots, X(F(X)) \big],
\end{equation}
which is interpreted as a \emph{finite word}\footnote{A finite word in an
alphabet $A$ is a finite (possibly empty) sequence of elements of $A$. We denote
by $A^* = \bigcup_{\ell = 0}^\infty A^\ell$ the set of finite words in $A$, with
the convention $A^0 = \{\varnothing\}$. We denote by $|w|$ the length of the
word $w \in A^*$, defined as the unique $\ell \in \Z_+$ such that $w \in
A^\ell$. In particular, $\varnothing$ is the only word with length $0$.} in
$\N$.
Let $|\widetilde \Pi_k(X)|$ be its
length and $\|\widetilde \Pi_k(X)\|$ its content, that is,
the total number of balls, corresponding to the sum of the letters of that word.
We have
\[
0 \le |\widetilde \Pi_k(X)| = F(X)-B(X,k) \le k-1, \quad
0\le \|\widetilde \Pi_k(X)\| = \sum_{j=X(B,k)+1}^{F(X)}X(j) \le k-1,
\]
and $|\widetilde \Pi_k(X)|=0$ if and only if $\|\widetilde \Pi_k(X)\|=0$.
The set
\begin{align*}
\widetilde \Pi_k(\configS)
&=\{\widetilde \Pi_k(X): X \in \configS\}
\\
&=\{\varnothing\} \cup
\{[a_1, \ldots,a_\ell]:\, 1 \le \ell \le  k-1,\, a_1,\ldots,a_\ell \ge 1,\,
a_1+\cdots+a_\ell \le k-1\}
\end{align*}
is finite\footnote{More precisely $|\widetilde \Pi_k(\configS)| = 2^{k-1}$. Indeed, recall that the set of $\ell$-tuples $(a_1, \ldots, a_\ell)$ of strictly positive integers summing to $m$ has cardinality $\binom{m-1}{\ell-1}$. Thus the number of $\ell$-tuples of positive integers summing at most to $k-1$ is given by $\sum_{m=\ell}^{k-1} \binom{m-1}{\ell-1} = \binom{k-1}{\ell}$. Summing over all possible values of $\ell$ yields that $\widetilde \Pi_k(\configS)$
has cardinality $\sum_{\ell=0}^{k-1} \binom{k-1}{\ell} = 2^{k-1}$.}.

We think of $\tilde{\Pi}_k$ as a projection of $\configS$ on $\tilde{\Pi}_k(\configS)$ a set of finite cardinality. We observe that an IBM $X$ with step distribution satisfying \eqref{eqn:kSupport} is compatible with this projection, in the sense that $\tilde{\Pi}_k(X)$ remains a Markov chain.

\begin{lemma}
\label{lem:finiteMarkovchain}
Let $k \geq 2$ be an integer, and assume that the law $\mu$ is supported in $\{0\} \cup\{1,\ldots,k\} \cup\{\infty\}$ with $\mu(\{k\})>0$.
Consider the IBM defined by $X_{n+1}=\Phi_{\xi_{n+1}}(X_n)$, where the $(\xi_n)$ are i.i.d.\ random variables with law $\mu$. Then $Z_n = \widetilde \Pi_k(X_n)$, $n \ge 0$, is an irreducible Markov chain with values in the set $\widetilde\Pi_k(\mathbb S)$.
In particular, the chain is positive recurrent and has a unique stationary probability measure.
\end{lemma}

\begin{proof}[Intuition of the proof]
The reason that the function $\widetilde \Pi_k$ preserves the Markov property is because each of the states $a\in \widetilde \Pi_k(\mathbb S)$ conveys just enough information about the state $X \in \mathbb S$ with $a= \widetilde \Pi_k(X)$ that is enough to decouple the past from the future.

For example, if $a=\varnothing$ then we know that any $X \in \mathbb S$ with $a= \widetilde \Pi_k(X)$ must satify $B(X,k)=F(X)$ (because $a$ has
length $0$) and so the front bin $F(X)$ contains at least $k$ balls. This means that regardless of the value of $\xi$, $1 \le \xi \le k$, we have $\Phi_\xi(X)$ is $X$ together with a single ball in the bin to the right of $F(X)$. So $\Phi_\xi(X)$ has a front at $F(X)+1$ with a single ball and so $\widetilde\Pi_k(\Phi_\xi(X))=[1]$. On the other hand, if $\xi=0$ or $\infty$ then the next state remains $a$. A complete proof of this fact being available in \cite[Lemma~3.1]{MR16}, it is perhaps best to work out two examples and leave the formal details to the reader.

\begin{figure}[ht]
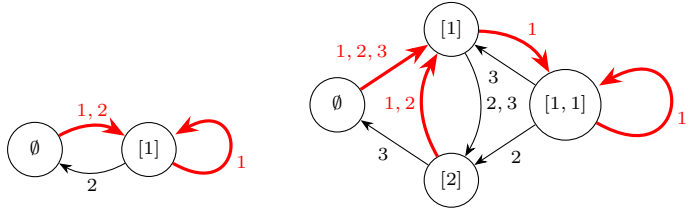

\centering

\begin{subfigure}{.3\linewidth}
\centering
\include{figTex/projectedMarkovA}
\end{subfigure}
\begin{subfigure}{{.3\linewidth}}
\centering
\include{figTex/projectedMarkovB}
\end{subfigure}

\caption{This figure describes the projected Markov chain, with $k=2$ on the
left picture, and $k=3$ on the right one. The integers over the directed edges
represent the values of $\xi$ corresponding to the transitions. For example, if
$\mu$ is the law of $\xi$, the transition probability for state $[1]$ to state
$[2]$ on right equals $\mu(2) + \mu(3)$. The thick arrows represent transitions
that move the front by one unit. Note that if $\xi \in \{0,\infty\}$, the
projected Markov chain remains at the same state, moving the front by $1$ if and
only if $\xi =0$.}
\label{wordchain}
\end{figure}

It is easy to see that the Markov chain $\widetilde \Pi_k(X_n)$ has an irreducible class connected to the state $\varnothing$ (by producing a sequence of moves that takes the chain from any state to $\varnothing$; the assumption that $\P(\xi=k)>0$ is required at this step). Since the state space is finite, the chain is positive recurrent and has a unique invariant probability measure.
\end{proof}

Given $a=[a_1, \ldots, a_\ell] \in \widetilde \Pi_k(\mathbb S)$
let $L(a) = a_\ell$ with the convention that $L(a)=k$ if $a=\varnothing$.
So if $Z_n = \widetilde\Pi_k(X_n)$ then $L(Z_n)$ corresponds to the number of balls in the front bin of $X_n$ (or $k$ if this number is larger than $k$).

\begin{proposition}
\label{prop:finitespeed}
Let $(X_n)$ be an IBM($\mu$) where $\mu$ is supported on
$\{0\}\cup\{1,\ldots,k\} \cup \{\infty\}$ where $k \ge 2$ and $\mu(k)>0$.
Then there exists a constant $v_\mu\in[0,1]$,
that does not depend on $X_0$, such that
\[
\lim_{n\rightarrow\infty} \frac{F(X_n)}{n}=v_\mu \quad \text{a.s.}
\]
\end{proposition}

\begin{proof}
Let $Z_n = \widetilde\Pi_k(X_n)$. Notice that
\begin{equation}
  \label{eqn:cnsFrontMoving}
F(X_j) > F(X_{j-1}) \iff F(X_j)-F(X_{j-1})=1 \iff L(X_{j-1}) \geq \xi_j,
\end{equation}
using that a ball is added in the leftmost empty bin if $\xi$ is either equal to $0$ or smaller than the number of balls in the rightmost bin, see Figure~\ref{wordchain}. Setting $b_j$ for the indicator of the event on the right
of the last equivalence, we have
\[
F(X_n)-F(X_0)=\sum_{j=1}^n b_j.
\]
By the ergodic theorem $\frac{1}{n} \sum_{j=1}^n b_j$ converges to a constant $v_\mu$ a.s. which completes the proof.
\end{proof}

\begin{remark}
Observe that $v_\mu$ can be computed from the knowledge of the invariant distribution $\pi$ of the Markov chain $\widetilde\Pi_k(X_n)$. More precisely, writing $Z$ for a random word distributed according to $\pi$ and $\xi$ an independent random variable with law $\mu$, we have
\[
  v_\mu = \P(L(Z) \geq \xi).
\]
For example, if $k=2$, we have $v_\mu = \pi(\varnothing) \mu(\{0,1,2\})
+ \pi([1]) \mu(\{0,1\}) = (\mu(1)+\mu(2))^2/(\mu(1)+2\mu(2))$; see Figure
\ref{wordchain}. However, since the size of $\Pi_k(\configS)$ grows exponentially
with $k$, this exact formula can be computed for small values of $k$ only.
\end{remark}

\paragraph{The general $\mu$ case.}
In order to extend the existence of the speed to the case when $\mu$ may have infinite support, we compare $\mu$ to three distributions with finite supports and use the increasing coupling of Proposition \ref{prop:coupling}. Let $\xi$ be a random element of $\overline\Z_+$ with distribution $\mu$. Fix $k\geq1$ and define
\begin{equation*}
  \label{xijk}
  \xi^j_k := \xi \1_{\xi\le k} + j \1_{\xi>k},
\end{equation*}
and let $\mu^j_k$ be the distribution of $\xi^j_k$.
We are only interested in the cases where $j=0, k$ or $\infty$
(noticing that $\xi^\infty_k = \xi$ when $\xi \le k$ or
$\infty$ when $\xi>k$).
Clearly,
\[
\xi^0_k \le \xi^k_k \le \xi \le \xi^\infty_k.
\]
Let $(\xi^0_{k,n}, \xi^k_{k,n}, \xi_n, \xi^\infty_{k,n})$, $n\in \N$,
be i.i.d.\ copies of $(\xi^0_k, \xi^k_k, \xi, \xi^\infty_k)$.
Define 4 coupled IBMs, $(X^0_{k,n})$,
$(X^k_{k,n})$, $(X_n)$ and $(X^\infty_{k,n})$,
as in Definition \ref{def31}, using the
4 coupled sequences
$(\xi^0_{k,n})$, $(\xi^k_{k,n})$, $(\xi_n)$ and $(\xi^\infty_{k,n})$
of selection numbers, respectively.
By \eqref{eq:monotonicity} and Proposition \eqref{prop:coupling} we have
\begin{equation}
\label{ALLIBMS}
\forall n \ge 0 ~~
X^\infty_{k,n}
\preceq
X_n
\preceq
X^k_{k,n}
\preceq
X^0_{k,n}
~\text{ a.s.}
\end{equation}
We next observe that
\begin{equation}
  \label{eqn:someFact}
\widetilde\Pi_k(X^\infty_{k,n})
=\widetilde\Pi_k(X^0_{k,n}) \quad \text{for all }  n\geq0,
\end{equation}
where $\widetilde\Pi_k$ was defined in \eqref{PPP}.
\begin{proof}[Proof of \eqref{eqn:someFact}]
It follows immediately from its definition that $\Phi_{0}$ commutes with $\Phi_k$ for all $k \in \bar{\Z}_+$. Using that $\xi^\infty_{k,n} \neq \xi^0_{k,n}$ if and only if $\xi_k \in (k,\infty)$, in which case the former takes value $\infty$ while the latter take value $0$, we have
\begin{align*}
  X^0_{k,n} = \Phi_{\xi^{0}_{k,n}} \circ \cdots \circ \Phi_{\xi^0_{1,n}}(X_0) = \Phi_0^{I_n} \circ  \Phi_{\xi^{\infty}_{k,n}} \circ \cdots \circ \Phi_{\xi^\infty_{1,n}}(X_0) =\Phi_0^{I_n}\left(X^\infty_{k,n}\right),
\end{align*}
where $I_n = \sum_{j=0}^{n-1} \ind{\xi_k \in (k,\infty)}$. Consequently, $X^0_{k,n}(\ell) = X^\infty_{k,n}(\ell - I_n)$ for all $\ell \in \Z$, from which we deduce that $\widetilde\Pi_k(X^\infty_{k,n}) =\widetilde\Pi_k(X^0_{k,n})$, completing the proof.
\end{proof}

We are now able to complete the proof of Theorem~\ref{thm:speed}.

\begin{proof}[\bf\em Proof of Theorem~\ref{thm:speed}]
Let $k\geq1$ be an integer.
By Proposition \ref{prop:finitespeed}, there exist
$v_{\mu^\infty_k}$ and $v_{\mu^0_k}$ in $[0,1]$ such that
\begin{align*}
\lim_{n\rightarrow\infty}
\frac{F(X^\infty_{k,n})}{n}&=v_{\mu^\infty_k}, \quad
\lim_{n\rightarrow\infty}
\frac{F(X^0_{k,n})}{n}=v_{\mu^0_k} \qquad \text{a.s.}
\end{align*}
Using that $F(X^0_{k,n}) = F(X^\infty_{k,n})+I_n$, as observed above, and that $I_n/n \to \mu([k+1,\infty))$ a.s. by law of large numbers, we conclude that
\begin{equation}
\label{eq:upperlowerbounds}
v_{\mu^{0,k}}=v_{\mu^{\infty,k}}+\mu([k+1,\infty)).
\end{equation}
Furthermore, we have that for every $n\geq1$,
\begin{equation}
\label{eq:frontbounds}
F(X^\infty_{k,n})\leq F(X_n) \leq F(X^0_{k,n}) \quad \text{a.s.}
\end{equation}
Combining \eqref{eq:upperlowerbounds} and~\eqref{eq:frontbounds} we deduce that
\begin{equation}
\label{eq:kbounds}
v_{\mu^\infty_k} \leq \liminf_{n\rightarrow\infty} \frac{F(X_n)}{n}
\leq \limsup_{n\rightarrow\infty} \frac{F(X_n)}{n}
\leq v_{\mu^\infty_k}+\mu([k+1,\infty)) \quad \text{a.s.}
\end{equation}
Since for every $i\geq0$ and $k\geq1$, we have
$\mu^{\infty,k}([0,i])\leq\mu^{\infty,k+1}([0,i])$, we have by
Proposition~\ref{prop:coupling} that $X^{\infty,k}_n\preceq X^{\infty,k+1}_n$
thus the sequence $(v_{\mu^{\infty,k}})_{k\geq1}$ is nondecreasing.
Since it is upper-bounded by $1$,
we have that
\[
\lim_{k \to \infty} v_{\mu^{\infty,k}} = v_\mu,
\]
for some $v_\mu\in[0,1]$.
Letting $k$ go to infinity in \eqref{eq:kbounds} we conclude that
\[
\lim_{n\rightarrow\infty} \frac{F(X_n)}{n}=v_\mu \quad \text{a.s.}
\]
\end{proof}

\begin{remark}
\label{rem:speedmonotonicity}
If $\nu$ is another probability distribution on $\bar\Z_+$ such that
$\nu([0,i])\leq \mu([0,i])$ for all $i\geq0$, then it follows from
Proposition~\ref{prop:coupling} that $v_\nu\leq v_\mu$.
\end{remark}

\subsection{Coupling words}
\label{subsec:couplingwords}
We assume, throughout this section, that $\mu$ is a probability measure on the
set $\N$ of positive integers, in other words that $\mu(\{0\}) =
\mu(\{\infty\}) = 0$. The IBM$(\mu)$ is the Markov process introduced in
Definition \ref{def31} and aims to describe a stationary version of this
process\footnote{Observe that the restriction made in this section remains
tame, as a selection number of $0$ moves the whole configuration to the right
by $1$ while a selection number of $\infty$ leaves the configuration unchanged. A
generic IBM($\nu$) $X$ can thus be coupled with an IBM($\mu$) $Y$,  where
$\mu(\cdot) = \nu(\cdot |\N)$, in such a way that the configuration $X_n$ is
obtained as a random shift of $Y_{Z_n}$, where $(Z_n)$ is a random walk
independent of $Y$. See Lemma \ref{lem:connectionIbm}.}.
Since $F(X_{n+1})-F(X_{n}) \ge 0$, and is positive with
positive probability, it is easy to see that $(X_n)$ has no stationary version.
But this is an illusion! To remedy it, we simply modify $X_n$ by shifting the
origin of space to the front bin $F(X_n)$. This will be done in Section
\ref{subsec:stationaryversion}.

Recall that $X_n$ is defined by the repeated application of maps of type
$\Phi_\xi$ several times. To simplify notation, we write $\xi X$ instead of
$\Phi_\xi(X)$, so that $\eta \xi X$ corresponds to $\Phi_\eta(\Phi_\xi(X))$.
Applying a finite number of these maps is thus identified by a ``selection
word''.

Let $\N^*$ be the set of words from $\N$. A {\em selection word} (or, simply,
word when no confusion arises) is a {nonempty} word in $\N$. Let
\[
\calU=\bigcup_{n\geq1} \N^n = \N^* \backslash \{\varnothing\}
\]
be the set of selection words. If $\alpha \in \cal U$ then we write it, rather uncoventionally, \emph{from
right to left}
\[
\alpha= \alpha_\ell \cdots \alpha_1
\]
(or $(\alpha_\ell, \ldots, \alpha_1)$ when confusion arises)
as we made the identification
\begin{equation}
\label{wordaction}
\Phi_{\alpha_\ell} \comp \cdots \comp \Phi_{\alpha_1} \equiv \alpha_\ell \cdots \alpha_1.
\end{equation}

\begin{example}
\label{weex}
Take $X = [ \ldots, 5, 3, 2, 1, 2, 0,0,\ldots]$, with $F(X)=0$,
and consider the selection word $\alpha = (2, 4, 5, 1)$.
To compute $\alpha(X)$ we first apply $1$ (that is, $\Phi_1$), then $5$, then $4$
and then $2$.
We have $1 (X)= \Phi_1 (X) = X + \delta_1 = [ \ldots, 5, 3, 2, 1, 2, 1,0,\ldots]$,
$5 (1X) = \Phi_5 \comp \Phi_1 (X) = X+\delta_1 + \delta_{-1}$,
$4(5 (1X)) = \Phi_4 \comp \Phi_5\comp\Phi_1 (X)= X+\delta_1 + \delta_{-1} + \delta_0$,
and, finally, $\alpha(X) = 2(4(5(1(X)))) =
X+ 2\delta_1 + \delta_{-1} + \delta_0$.
\end{example}

Recall that we denote by $|\alpha|$ the length of the word $\alpha$. An element of a selection word is a selection number and is sometimes referred to
as a {\em letter}. There is only one word of length $0$, the empty word $\varnothing$, which is not considered as a selection word. However, we can formally define $\Phi_\varnothing$ as the identity on $\mathbb S$.

If $\alpha = \alpha_\ell \cdots \alpha_1$ is a selection word and $k \le \ell$
then $\alpha_k \cdots \alpha_1$ is a {\em prefix} of $\alpha$, while
$\alpha_\ell \cdots \alpha_k$ is a {\em suffix}.
The number of occurrences of letter $n$ in a word $\alpha$ is written
\[
\vartheta_n(\alpha):=|\{1\leq k \leq |\alpha|, \alpha_k=n\}|.
\]
The concatenation of $\alpha=\alpha_\ell\cdots \alpha_1$ with
$\beta =\beta_k\cdots\beta_1$
is the word $\beta\alpha=\beta_k\cdots\beta_1 \alpha_\ell\cdots \alpha_1$,
again keeping in mind the right-to-left convention. Of course, $\varnothing \alpha
= \alpha \varnothing = \alpha$.
For $m, \ell \in \N$, the word $m^\ell$ is the length-$\ell$ word whose letters are all equal to $m$.

For $k\in\N$ define
\[
\Pi_k:
\begin{array}{rcl} \configS & \longrightarrow & \N^k\\ X & \longmapsto
&\left( X\left(F(X)-k+1\right), \ldots, X\left(F(X) \right) \right)
\end{array}
\]
which isolates the content of the rightmost $k$ non-empty bins. The definition
can be extended to $k=0$ by setting $\Pi_0(X)$ to be the empty vector.

\begin{definition}[$k$-coupling words]
\label{defcoupl}
We say that a selection word $\alpha$ is {\em $k$-coupling} if
\[
\forall X, Y \in \configS ~~ \Pi_k(\alpha X) = \Pi_k(\alpha Y).
\]
Let $\mathcal C_k$ be the set of all $k$-coupling words.
A word $\alpha$ is called {\em coupling} if it is $k$-coupling for some positive
integer $k$.
\end{definition}

In other terms, a word $\alpha$ is said to be $k$-coupling if the action of
$\alpha$ on any configuration always give the same values in the $k$ rightmost
non-empty bins. For example, $1$ is a $1$-coupling word, as the rightmost
non-empty bin in $1X$ always contain exactly one ball. More generally, for $k
\in \N$, the word $1^k$ is a $k$-coupling word.

Note that $\mathcal C_1 \supset \mathcal C_2 \supset \cdots$. Hence $\mathcal
C_1$ is the set of all coupling words. Note that $\mathcal C_1$ is the set of
all words $\alpha$ such that the number of balls in the bin at the front of
$\alpha X$ is the same as for $\alpha Y$ for any other configuration $Y$. The
words in $\mathcal U \setminus \mathcal C_1$ are non-coupling.

\begin{definition}[Coupling number]
The {\em coupling number} of $\alpha$ is defined by
\[
K(\alpha) = \sup\{k \ge 0:\, \forall X, Y \in \mathbb ~
\Pi_k(\alpha X) = \Pi_k(\alpha Y)\}.
\]
\end{definition}

In particular, for a word $\alpha$, $K(\alpha) = 0$ is equivalent to $\alpha \in
\mathcal{U} \setminus \mathcal{C}_1$, i.e.\ to the fact that $\alpha$ is non-coupling. More generally, for $k \geq 1$, we have
\[
  K(\alpha) = k \iff \alpha \in \mathcal{C}_k \setminus \mathcal{C}_{k+1}.
\]
Hence, $\alpha$ is $k$-coupling if and only if $K(\alpha) \geq k$. We naturally let $K(\varnothing)=0$.

\begin{example}
We have that $K(1)=1$, $K(1,1)=2$, and more generally, if $1^\ell=(1,\ldots,1)$ is the word consisting of $\ell$ letters all equal to $1$, then  $K(1^\ell)= \ell$. Indeed, as previously observed, $1^\ell$ is an $\ell$-coupling word. But $1^\ell \not\in \mathcal C_{\ell+1}$ because the $\ell+1$st rightmost nonempty bins contains the number of balls in the rightmost non-empty bin of $X$.

Consider next the word $(2,2,1)$. We see that, for any $X$, we have $\Pi_2((2,2,1)X) = (2,1)$; but the three rightmost bins of $(2,2,1)X$ depend on the content of the front of $X$. Hence $(2,2,1) \in \mathcal C_2 \setminus \mathcal C_3$. As a third example, one can check that $(2,2)$ is a non-coupling word.
\end{example}

We now define a class of words which will provide useful examples of coupling words.

\begin{definition}[Triangular words]
A word $\alpha=(\alpha_\ell,\ldots,\alpha_1)\in\calU$ is called
\emph{triangular} if for every $1\leq i \leq \ell$ we have $\alpha_i\leq i$. We
denote by $\calT$ the set of triangular words. An infinite sequence $(\ldots,
\alpha_3, \alpha_2, \alpha_1)$ of positive integers is called an {\em infinite
triangular word} if $\alpha_i \le i$ for all $i$. Let $\calT_\infty$ be the set
of infinite triangular words.
\end{definition}

We claim that every triangular word is a coupling word. More precisely, if $\alpha$ is a triangular word with $k$ occurrences of the letter $1$, we show that $\alpha$ is $k$-coupling.

\begin{lemma}
\label{lem:triangularcoupling}
Let $k\geq1$ and let $\alpha$ be a triangular word with $\vartheta_1(\alpha)=k$.
Then $\alpha$ is $k$-coupling.
\end{lemma}
\begin{proof}
Write $\alpha=(\alpha_\ell,\ldots,\alpha_1)\in\calT$ and let $X_0,Y_0\in \mathbb S$. Without loss of generality, we assume that $F(X_0) = F(Y_0) = 0$.
For every $1\leq i\leq n$, define $X_i=\alpha_i \cdots \alpha_1(X_0)$ and $Y_i=\alpha_i \cdots \alpha_1(Y_0)$.

We show by induction on $0\leq i\leq \ell$ that $X_i$ and $Y_i$ coincide for all bins of positive indices  and have $i$ balls in these bins. In other words, we show  that for every $k\geq1$, $X_i(k)=Y_i(k)$ and $\sum_{j\geq1} X_i(j)=\sum_{j\geq1} Y_i(j)=i$. This is trivially true when $i=0$.

If $1\leq i\leq \ell$, by induction hypothesis, we know that $X_{i-1}$ and $Y_{i-1}$ have $i-1$ balls in the same positions in bins of positive indices and at least one ball in the bin of index $0$, thus $B(X_i,\alpha_i)=B(Y_i,\alpha_i)\geq0$. Hence $\Phi_{\alpha_i}$ adds a ball to
$X_{i-1}$ and $Y_{i-1}$ in the same bin of positive index, completing the proof of that statement.

Next, setting $f:=F(X_\ell)=F(Y_\ell)$, we have $f\geq1$ (since $\alpha_1 = 1$). As $X_{|Z_+} = Y_{|Z_+}$, we have $\Pi_f(\alpha(X_0))=\Pi_f(\alpha(Y_0))$, so $K(\alpha)\geq f$. Since each transition $\Phi_1$ adds a balls to a previously empty bin, we have that $f\geq \vartheta_1(\alpha)=k$, which shows that $\alpha\in\mathcal C_k$.
\end{proof}

More generally, an \emph{infinite triangular word} allows the coupling of IBM with arbitrary initial conditions.

\begin{lemma}
\label{lem:infinitetriangular}
Let $\alpha \in \calT_\infty$ and let $X_0,Y_0\in \configS_0$. For every $n\geq1$, denote by $X_n={\alpha_n \cdots \alpha_1}(X_0)$ and by $Y_n={\alpha_n \cdots \alpha_1}(Y_0)$. Then for every $n\geq0$ we have ${X_n}_{|\N} = {Y_n}_{|\N}$ (both configurations have identical contents in all bins with positive indices).
\end{lemma}

The proof follows from an induction similar to the one used above in the proof of Lemma \ref{lem:triangularcoupling}.

\subsection{Construction of a stationary version and coupling}
\label{subsec:stationaryversion}
In order to construct a stationary version of the infinite bin model, we first define a variant of the IBM where the front is pinned at position $0$.
Denote by $\configS_0$ the set of all $X\in \configS$ such that $F(X)=0$. For
$X \in \Z^\Z$ and $u \in \Z$, we introduce the map
\[
  \Theta_u X(k):= X(k-u).
\]
In other words, $\Theta_u$ shifts $X$ to the right by $u$. We remark immediately that $\Theta_u = \left(\Phi_0\right)^u$.

Next, let
\[
\Theta X : \begin{array}{rcl}
  \configS &\to& \configS_0\\
  X &\mapsto & \Theta_{-F(X)} X
\end{array}
\]
be the projection $\configS \to \configS_0$ which ``pins the front'' at position $0$.

\begin{definition}[Pinned infinite bin model]
Given a probability measure $\mu$ on $\N$, and a sequence $(\xi_n)_{n \ge 1}$ of i.i.d.\ random variables with law $\mu$, we define the pinned IBM$(\mu)$ to be the Markov process $(\widehat X_n)$ with values in $\mathbb S_0$ given by
\[
  \widehat X_{n+1} = \Theta \circ \Phi_{\xi_{n+1}}(\widehat X_n), \quad n \ge 0.
\]
We leave the initial configuration $\widehat X_0 \in \configS_0$ unspecified.
\end{definition}

Let $(\widehat X_n)$ be the pinned IBM as above. Using the same sequence of selection numbers, and recalling the convention \eqref{wordaction}, we define an unpinned IBM by setting $X_0 = \hat{X}_0$ and $X_{n} = \xi_n \cdots \xi_1 {X}_0$ for $n \geq 1$.
We observe that for any selection word $\alpha$ and $X \in \configS$, we have $\Theta (\alpha X) = \Theta (\alpha (\Theta X))$. Therefore, for each $n \in \N$, we have
\[
  \widehat X_n = \Theta X_n.
\]
As a result, $\widehat{X}_n$ can be thought of as the pinned version of the IBM $X$.

\begin{definition}[Stationary pinned IBM]
Given a stationary sequence $(\xi_n)_{n \in \Z}$ we say that $(Y_n)_{n \in \Z}$ is a stationary pinned IBM if
\begin{equation}
  \label{eq:recursion}
  Y_{n+1}=\Theta({\xi_{n+1}}Y_n), \quad n \in \Z.
\end{equation}
If $(\xi_n)_{n \in \Z}$ is an i.i.d.\ sequence with common law $\mu$ then we refer to (the law) of $(Y_n)$ as a stationary version of IBM$(\mu)$.
\end{definition}

We show the existence and uniqueness of this stationary version under a first moment condition on $\mu$.
\begin{theorem}
\label{thm:stationary}
Let $\mu$ be a probability distribution on $\N$ such that
\begin{equation}
  \label{eq:munice}
  \sum_{k \in \N} k \mu(k)<\infty \text{ and } \sup_{k \in \N} \mu(k) < 1.
\end{equation}
Let $(\xi_n)_{n\in\Z}$ be a sequence of i.i.d.\ random variables with law $\mu$. Write $\calF_n:=\sigma(\xi_k,-\infty < k\leq n)$. Then a.s. there exists a unique process $(Y_n)_{n\in\Z}$ on $\configS_0$ satisfying the recursion \eqref{eq:recursion} and
such that $Y_n$ is $\calF_n$--measurable for all $n \in \Z$.
\end{theorem}

We observe that the condition $\sup_{k \in \N} \mu(k) < 1$ is only here to avoid considering the case when $\mu$ is a Dirac mass $\delta_k$ for some $k \geq 1$. In that case, there would be $k$ processes $(Y^j_n)_{n \in \Z, 1 \leq j \leq k}$ on $\configS_0$ satisfying \eqref{eq:recursion}, given by
\[
  Y^k_n(\ell) = \begin{cases}
    k & \text{ if } \ell \leq -1\\
    0 & \text{ if } \ell \geq 1\\
    \left((j+n) \text{ mod } k\right) + 1& \text{ if } \ell = 0
  \end{cases}.
\]

The proof of Theorem~\ref{thm:stationary} relies on the existence of infinitely many $\infty$-coupling words embedded in the bi-infinite sequence $(\xi_n)_{n \in \Z}$. These words induce \emph{renewal events} for the stationary IBM($\mu$).

In order to avoid unnecessary technicalities, we prove Theorem~\ref{thm:stationary} under the additional condition $\mu(1) > 0$. Under this condition, we prove the existence of infinitely many infinite triangular words in Lemma~\ref{lem:infiniterenovations}. These triangular words induce renovation events for the IBM($\mu$), in the sense that conditionally on their realisation at time $k$, the only algebraic dependence of the future (starting from time $k$) on the past (before time $k$) is given by the position of the front at time $k$. For the concept of renovation events see \cite{BOR78,BOR98,BORFOS92,FK03}.

\begin{remark}
To extend Theorem~\ref{thm:stationary} to measure $\mu$ such that $\mu(0) = 1$,
one can use \cite{CR17} in which for all $a\neq b \in \N$, a coupling word
$\alpha$ with letters in $\{a,b\}$ with arbitrary $K(\alpha)$ is constructed. An
analogue of Lemma~\ref{lem:infiniterenovations} can be stated replacing a
triangular word by $\alpha$ followed by an \emph{infinite coupling word}
$\beta$.
The rest of the
proof would follow straightforwardly;
see \cite{MR19}.
\end{remark}

With Lemma~\ref{lem:infinitetriangular} in mind, we show almost surely, there
exist infinitely many triangular words in the infinite sequence
$(\xi_n,n\in\Z)$. This result (and its proof) should be compared and contrasted
with Lemma~\ref{lemoll}.

\begin{lemma}
\label{lem:infiniterenovations}
Let $(\xi_n)_{n\in\Z}$ be an i.i.d.\ sequence with $\E \xi_0<\infty$.
Then the law of the random set
\begin{equation}
\label{RRR}
\mathscr R = \{k \in \Z:\, (\cdots \xi_{k+2}\, \xi_{k+1}\, \xi_k) \in \calT_\infty\}.
\end{equation}
is invariant under translations and
$\mathscr R \cap [0,\infty)$ and $\mathscr R \cap (-\infty,0]$ are
infinite sets a.s.
\end{lemma}

\begin{proof}
The invariance by translation is obvious. Observe that
\[
\P((\cdots \xi_{k+2}\, \xi_{k+1}\, \xi_k) \in \calT_\infty)
= \P(\xi_k \le 1, \xi_{k+1} \le 2, \xi_{k+2} \le 3, \ldots)
= \prod_{j=1}^\infty (1-\P(\xi_0>j)).
\]
According to the assumptions \eqref{eq:munice} we have $\P(\xi_0=1)>0$ and hence all terms in this product are positive. We also have $\sum_{j=0}^\infty \P(\xi_0>j) = \E \xi_0 <  \infty$ and so the whole product is positive. As a result, $\mathcal{R}$ possesses a positive density on $\Z$. Just as in Lemma~\ref{lemoll} we conclude that $\mathscr R$ contains infinitely many positive and infinitely many negative integers a.s.
\end{proof}

We write $\mathscr R = \{T_0, T_{\pm 1}, T_{\pm 2}, \ldots\}$ and arrange the indexing so that
\begin{equation*}
  \label{Tpoints}
  \cdots<T_{-1}<T_0 \le 0 <T_1<\cdots
\end{equation*}
For each $n$ in $\Z$, let $s(n)$ be the unique (random) integer such that $T_{s(n)} \le n < T_{s(n)+1}$.

\begin{proof}[\bf\em Proof of Theorem~\ref{thm:stationary}]
We shall construct $(Y_n)$ as a measurable function of $(\xi_n)$ by constructing $\Pi_k(Y_n)$ for all $k$. Fix $n\in\Z$ and $k\geq1$. Consider the selection word
\[
  \xi_n \cdots \xi_{T_{s(n)-k+1}}.
\]
Since $T_{s(n)-k+1} \in \mathscr R$, the word $\xi_n \cdots \xi_{T_{s(n)-k+1}}$
is triangular. Note that the number of elements of $\mathscr R$ in the interval
$[T_{s(n)-k+1}, T_{s(n)}]$ is $k$. For each $T_j$ we have $\xi_{T_j} =1$, by the
definition of $\mathscr R$. Hence the selection word $\xi_n \cdots
\xi_{T_{s(n)-k+1}}$ contains at least $k$ letters equal to $1$. Thus by Lemma
\ref{lem:triangularcoupling} it is $k$-coupling. Hence the vector
$\Pi_k({\xi_n\cdots \xi_{T_{s-k+1}}}(X))$ is the same for all $X\in \configS$.

Fix an $X$ and define the rightmost $k$ bins of $Y_n$ to be equal to this vector. This definition is consistent for different values of $k$, hence defines $Y_n$ up to a global shift. Requiring $F(Y_n)=0$ yields a unique definition of $Y_n$ a.s. By construction, for every $n\in\Z$, $Y_n\in\calF_n$ and the sequence $(Y_n)_{n\in\Z}$ satisfies \eqref{eq:recursion}. Conversely, any sequence of configurations in $\configS_0$ satisfying \eqref{eq:recursion} has to coincide with the sequence $(Y_n)$ constructed above.
\end{proof}

It follows from Theorem~\ref{thm:stationary} that every finite-dimensional marginal of the pinned IBM started at time $0$ converges and even gets coupled to the corresponding marginal of the stationary version.

\begin{corollary}[Coupling-convergence]
\label{cor:couplingconvergence}
Let $(\xi_n)_{n\in\Z}$ be an i.i.d.\ sequence of law $\mu$ satisfying conditions
\eqref{eq:munice}. Let $\widehat X_0\in \mathbb S_0$
and let $(\widehat X_n)_{n\geq0}$
be a pinned IBM($\mu$) constructed using the variables $\xi_n, n\geq1$. Let
$(Y_n)$ be the stationary version of the pinned IBM. Then for every $k\geq1$,
there exists $n_k\geq1$ such that for every $n\geq n_k$, we have
\begin{equation}
  \label{eq:couplingconvergence}
  \Pi_k(\widehat X_n)=\Pi_k(Y_n).
\end{equation}
\end{corollary}

\begin{proof}
Let $k\geq1$ and set $n_k=T_k$ which is finite a.s. If $n\geq n_k$, we have $\vartheta_1(\xi_n \cdots \xi_1)\ge k$ and \eqref{eq:couplingconvergence} follows from Lemma~\ref{lem:triangularcoupling}.
\end{proof}

A second consequence of Theorem~\ref{thm:stationary} is a simple expression for the speed $v_\mu$ in terms of the stationary version of the pinned IBM($\mu$).

\begin{corollary}
\label{cor:stationaryspeed}
Let $(Y_n)$ be the stationary version of the pinned IBM constructed using the i.i.d. variables $(\xi_n)_{n\in\Z}$ of law $\mu$ satisfying conditions~\eqref{eq:munice}. Then
\begin{equation}
  \label{eq:stationaryspeed}
  v_\mu=\P(\xi_1\leq Y_0(0)).
\end{equation}
\end{corollary}

\begin{proof}
Let $(X_n)$ be an IBM constructed using the random variables $(\xi_n, n \geq 0)$ such that $X_0 = Y_0(0)$ a.s. We recall that
\[
  v_\mu = \lim_{n \to \infty} \frac{F(X_n)}{n}  \quad \text{a.s.}
\]
Therefore, using the dominated convergence theorem, we have
\begin{equation*}
  v_\mu  = \lim_{n \to \infty} \frac{\E(F(X_n))}{n} = \lim_{n \to \infty} \frac{1}{n} \sum_{j=1}^n \E\left( F(X_j)-F(X_{j-1}) \right)
   = \lim_{n \to \infty} \frac{1}{n} \sum_{j=1}^n \P(\xi_j \leq L(X_{j-1})),
\end{equation*}
using \eqref{eqn:cnsFrontMoving} (recall that $L(X)$ is the number of balls in the rightmost non-empty bin of $X$). Then, as $L(X_n) = L(Y_n) = Y_n(0)$ a.s. for all $n\in\N$, we have
\[
  \P(\xi_j \leq L(X_{j-1})) = \P(\xi_j \leq Y_{j-1}(0)) = \P(\xi \leq Y_0(0)),
\]
using that $Y$ is a stationary adapted sequence. This result immediately implies \eqref{eq:stationaryspeed}.
\end{proof}

\subsection{New expressions for the speed}
\label{speedexpress}
Using Corollary~\ref{cor:stationaryspeed}, we deduce in this section new
formulas for the speed $v_\mu$ as infinite sums over some special classes of
selection words. In constructing the stationary regime, we were interested in
renovation events, that is events that resulted in decoupling the future from
the past.
In computing the speed, we are merely interested in whether the front advances
or not at time $1$.

Words that manage to advance the front of any configuration at their last selection index are called {\em good}. Words that never do this are called {\em bad}. Bad is not the opposite of good: there are selection words
that sometimes move the front {at the last step} and sometimes do not. In what follows recall that, for integers $\xi, \eta, \ldots$ and $X \in \mathbb
S$, the symbol $\eta \xi X$ stands for $\Phi_\eta(\Phi_\xi(X))$.

\begin{definition}[Good/bad words]
Define the following sets of selection words:
\begin{align*}
\mathcal G & := \{\alpha_\ell \cdots \alpha_1 \in \calU:\,
\forall X \in \configS ~
F({\alpha_\ell}{\alpha_{\ell-1}} \cdots {\alpha_1}  X)
= F( {\alpha_{\ell-1}} \cdots {\alpha_1}  X) +1\},
\\
\mathcal B & := \{\alpha_\ell \cdots \alpha_1 \in \calU:\,
\forall X \in \configS ~
F({\alpha_\ell}{\alpha_{\ell-1}} \cdots {\alpha_1}  X)
= F( {\alpha_{\ell-1}} \cdots {\alpha_1}  X)\},
\\
\text{and } \mathcal A &:= \calU \setminus (\mathcal G \cup \mathcal B).
\end{align*}
We call the words in $\mathcal{G}$ \emph{good}, those in $\mathcal{B}$ \emph{bad}, and those in $\mathcal{A}$ \emph{ambivalent}.
\end{definition}

\begin{example}
The word $1$ is good. More generally every word $\alpha_\ell \cdots\alpha_1$
ending with $\alpha_\ell = 1$ is good.

The word $(2,1)$ is bad. To see this, let $X$ have $F(X)=0$. We observe that $1X
= X + \delta_1$ and $2(1X) = X + 2 \delta_1$ as the second ball of $1X$ is at
bin $0$. Hence, $F(2(1X)) = F(1X) = 1$. As this relation holds regardless of
$X$, the word $(2,1)$ is bad.

The word $2$ is ambivalent. Indeed, if $X$ has $F(X)=0$ and $X(0)=1$ then
$2X=X+\delta_0$, so $F(2X) = 0$, but if $X(0)\ge 2$ then $2X=X+\delta_1$, so
$F(2X)=1$.
\end{example}

Coupling words can be used to generate a large class of good and bad words. More
precisely, the following result holds.

\begin{lemma}
\label{lem:goodbadexamples}
If $\alpha\in\mathcal C_1$ and $\xi\in\N$
then either $\xi\alpha \in \mathcal G$ or $\xi\alpha \in \mathcal B$.
In particular, every triangular word is either good or bad.
\end{lemma}
\begin{proof}
Since $\alpha$ is in $\mathcal C_1$, the front of $\alpha X$ contains exactly
the
same number of balls as the front of $\alpha Y$ for any configuration $Y$;
see forthcoming Definition~\ref{defcoupl}.
Let $b$ be this number of balls.
Then $F(\xi \alpha X) = F(\alpha X)+1$ if and only if $\xi \le b$.
Since neither $\xi$ nor $b$ depend on $X$, it follows that the word is good if
$\xi \le b$, or bad otherwise.

By Lemma \ref{lem:triangularcoupling} every triangular word is in $\mathcal
C_1$. Hence the previous argument applies.
\end{proof}

To apply Corollary~\ref{cor:stationaryspeed} to compute the speed $v_\mu$ of the IBM($\mu$), we have to detect, based on the sequence $(\xi_n, n \leq 1)$, if the front of $Y$ will increase at time $1$. To this end, we introduce the notion of minimal good and bad words.

\begin{definition}[$\cal V$-minimal words]
For every $\mathcal V \subset \mathcal U$ define the set of $\mathcal V$-minimal
words as the set of words in $\mathcal{V}$ with no strict suffix belonging to
$\mathcal{V}$, i.e.
\[
  \mathcal V_{\min} :=\{\alpha=\alpha_\ell\cdots\alpha_1 \in \mathcal V:\, \alpha_\ell\cdots\alpha_i  \not \in \mathcal V, i=2,\ldots,\ell\}.
\]
\end{definition}

Observe that if $\alpha \in \mathcal{V}$ satisfies $|\alpha| = \min \{|\beta|,\beta \in \mathcal{V}\}$, then $\alpha \in \mathcal{V}_{\min}$.

\begin{example}
First consider $\mathcal V = \mathcal T$ the set of triangular words. We see that $(2,2,1) \in \mathcal T_{\min}$ (because $(2,2,1)$ is in $\mathcal
T$ but $(2,2)$ and $(2)$ are not in $\mathcal T$). On the other hand, the
triangular word $(2,2,1,3,2,1)$ is not in $\mathcal T_{\min}$ because its suffix
$(2,2,1)$ is triangular.

Let us now consider minimal good and bad words. We have obviously $1 \in \mathcal{G}_{\min}$, but $(1,1) \not \in \mathcal{G}_{\min}$. Similarly, we observe that $(3,2,1) \in \mathcal B$ as $3(2(1X)) = X + 3 \delta_1$ and $2(1X) = X + 2\delta_1$, but both $(3)$ and $(3,2)$ are ambivalent words (hence not bad). Therefore $(3,2,1) \in \mathcal B_{\min}$.
\end{example}
We next observe that
\begin{equation}
\label{GB}
(\calG\cup\calB)_{\min}=\calG_{\min} \cup \calB_{\min}.
\end{equation}
This follows from the fact that a suffix of a good word cannot be bad and a suffix of a bad word cannot be good.

\begin{definition}[Weight of a word]
The {\em weight} of
$\alpha=(\alpha_\ell,\ldots,\alpha_1)\in\calU$ under the probability measure $\mu$ is defined to be
\[
  w_\mu(\alpha):=\prod\limits_{i=1}^\ell \mu(\alpha_i).
\]
\end{definition}

We are now ready to explain how to get new formulas for the speed $v_\mu$.

\begin{proposition}
\label{prop:Cminimal}
Let $\mu$ be any probability measure on $\N$ and $(\xi_n)_{n\in\Z}$ a sequence
of i.i.d.\ random variables with law $\mu$. Let $\mathcal V$ be a class of selection words (recalling that $\varnothing$ is not a selection word) and define
\[
T^*_\mathcal V :=\sup\{-\infty < \ell \le 1:\, \xi_1\,\xi_0\cdots \xi_\ell \in\mathcal V\}.
\]
Assume that
\\
(i) $\mathcal V \subset \calG \cup \calB$ ($\mathcal V$ contains no ambivalent words)
\\
(ii) $\P(T^*_\mathcal V  > -\infty) =1$.
\\
Then
\begin{equation}
\label{eq:Cminimalspeed}
v_\mu=\sum\limits_{\alpha \in \mathcal V_{\min} \cap \calG} w_\mu(\alpha)
=1-\sum\limits_{\alpha \in \mathcal V_{\min} \cap \calB} w_\mu(\alpha).
\end{equation}
\end{proposition}

\begin{proof}
Fix the set $\mathcal V \subset \mathcal G \cup \mathcal B$ and let $T = T^*_{\mathcal V}$ for brevity. Let $(Y_n)$ be the stationary version of the pinned IBM constructed from $(\xi_n)_{n\in\Z}$. From Corollary \ref{cor:stationaryspeed} we have $v_\mu = \P(\xi_1 \le Y_0(0))$. Observe that
\begin{equation}
\label{AJAX}
\{\xi_1 \le Y_0(0)\}
= \{F(\xi_1 \xi_0 \cdots \xi_T Y_{T-1})-F(\xi_0 \cdots \xi_T Y_{T-1})=1\}.
\end{equation}
Since $\xi_1 \xi_0 \cdots \xi_T \not \in \mathcal A$, it follows that
\[
F(\xi_1 \xi_0 \cdots \xi_T Y_{T-1})-F(\xi_0 \cdots \xi_T Y_{T-1})
= F(\xi_1 \xi_0 \cdots \xi_T X)-F(\xi_0 \cdots \xi_T X) \text{ for all
$X \in \mathbb S$.}
\]
And so the event of \eqref{AJAX} {is equal to}
\[
\{\xi_1 \xi_0 \cdots \xi_T \in \mathcal G\}.
\]
On the other hand, $\xi_1 \xi_0 \cdots \xi_T \in \mathcal V$
but $\xi_1 \xi_0 \cdots \xi_{T+j} \not \in \mathcal V$ if $j>0$.
Hence
\[
\xi_1 \xi_0 \cdots \xi_T \in \mathcal V_{\min} \text{ a.s.}
\]
We conclude that
\begin{equation}
\label{vmu}
  v_\mu = \P(\xi_1 \xi_0 \cdots \xi_T \in \mathcal V_{\min} \cap \mathcal G)
  = \sum_{\alpha \in \mathcal V_{\min} \cap \mathcal G} \P(\xi_1 \xi_0 \cdots \xi_T =\alpha).
\end{equation}
The length of the word $\xi_1 \xi_0 \cdots \xi_T$ is $2-T$, so
\begin{multline*}
  \P(\xi_1 \xi_0 \cdots \xi_T =\alpha) = \P(\xi_1 \xi_0 \cdots \xi_{2-|\alpha|}=\alpha, T=2-|\alpha|) \\
  = \P(\xi_1 \xi_0 \cdots \xi_{2-|\alpha|}=\alpha) = \P(\xi_1 \cdots \xi_{|\alpha|}=\alpha) =\mu(\alpha_1) \cdots \mu(\alpha_{|\alpha|}) = w_\mu(\alpha),
\end{multline*}
where the second equality follows from the fact that we calculate this probability for $\alpha \in \mathcal V_{\min}$. Thus the first identity in \eqref{eq:Cminimalspeed} is proved. The second identity follows from $\mathcal V_{\min}\setminus (\mathcal V_{\min} \cap \mathcal G) = \mathcal V_{\min} \cap \mathcal B$.
\end{proof}

By taking special choices for the class $\mathcal V $ we obtain the following useful formulas when $\mu$ satisfies assumptions \eqref{eq:munice}:

\begin{theorem}[speed as a sum over words, \cite{MR19}]
\label{thm:sumoverwords}
Let $\mu$ be a probability measure on $\N$ satisfying assumptions \eqref{eq:munice}. Then
\begin{align}
  \label{eq:minimaltriangularspeed}
  v_\mu&=\sum\limits_{\alpha\in \calT_{\min}\cap \calG} w_\mu(\alpha)\\
  \label{eq:minimalgoodspeed}
  v_\mu&=\sum\limits_{\alpha\in \calG_{\min}} w_\mu(\alpha).
\end{align}
\end{theorem}

\begin{remark}
As observed in Proposition~\ref{prop:Cminimal}, the number of possible formulas for the speed $v_\mu$ is equal to the number of subset of $\mathcal{V} \subset \mathcal{G} \cup \mathcal{B}$ such that $T^*_\mathcal{V} > -\infty$ a.s. However, the two formulas \eqref{eq:minimaltriangularspeed} and \eqref{eq:minimalgoodspeed} will be the more useful for our purpose. It should be apparent that setting $\mathcal{V} = \mathcal{G} \cup \mathcal{B}$ gives a formula such that $|T^*_\mathcal{V}|$ is minimal. However, for an algorithmic purpose, it is much easier to verify that a word is triangular than to verify that it is not ambivalent. Therefore, \eqref{eq:minimaltriangularspeed} is particularly efficient when estimating $v_\mu$ via Monte-Carlo methods, see forthcoming Section~\ref{sec:perfectSimu}.
\end{remark}

\begin{proof}
We first apply Proposition \ref{prop:Cminimal} with $\mathcal V = \mathcal T$
and make sure that (i) and (ii) in that proposition hold. By Lemma
\ref{lem:goodbadexamples} we have $\calT\subset\calG\cup\calB$, so (i) holds.
Recall that $T^*_{\mathcal T} = \sup\{\ell \le 1:\, (\xi_1\, \xi_0\,\xi_{-1}
\cdots \xi_{\ell+1} \xi_\ell) \in \mathcal T\}$. By Lemma
\ref{lem:infiniterenovations} all the points of $\mathscr R$ are finite in
absolute value a.s. Hence there are points $-\infty < \ell \le 0$ such that
$(\cdots \xi_{\ell+1} \xi_\ell) \in \mathcal T_\infty$. Any such point certainly
satisfies $(\xi_1\, \xi_0\, \cdots \xi_{\ell+1} \xi_\ell) \in \mathcal T$, so
$T^*_{\mathcal T} > -\infty$ a.s., therefore (ii) holds as well. As a
consequence \eqref{eq:minimaltriangularspeed} follows from
\eqref{eq:Cminimalspeed}.

We next apply Proposition \ref{prop:Cminimal} with $\mathcal V = \mathcal G \cup
\mathcal B$, so (i) holds immediately. Moreover, we have $T^*_{\calG\cup\calB}
=\sup\{\ell \leq 1:\, \xi_1 \cdots \xi_\ell \in\calG\cup\calB \}\geq T^*_\calT$
as $\xi_1\xi_0\cdots \xi_{T^*_{\calT}} \in \mathcal{G} \cup \mathcal{B}$, hence
$T^*_{\calG\cup\calB}>-\infty$ a.s. Additionally, thanks to \eqref{GB} we have
\[
  (\calG\cup\calB)_{\min}\cap\calG=(\calG_{\min}\cup\calB_{\min})\cap\calG =
(\calG_{\min} \cap \calG) \cup (\calB_{\min} \cap\calG) =\calG_{\min} \cup
\varnothing = \calG_{\min},
\]
which yields to \eqref{eq:minimalgoodspeed}, by using the first equality of \eqref{eq:Cminimalspeed} with $\mathcal V = \mathcal G \cup \mathcal B$.
\end{proof}

\begin{remark}
Note that there is no inclusion relation between $\calT_{\min}\cap\calG$ and
$\calG_{\min}$. For example
\[
  (2,4,2,1) \in\calG_{\min}\setminus(\calT_{\min}\cap\calG) \text{ and } (2,4,2,1,1) \in(\calT_{\min}\cap\calG)\setminus\calG_{\min}.
\]
\end{remark}

\begin{remark}
In \cite{MR19}, Theorem~\ref{thm:stationary} is showed to hold for any infinite bin model, under the condition that $\mu$ is not a Dirac mass. As a result, Corollary~\ref{cor:stationaryspeed} also holds without any assumption on the first of $\mu$. In fact, \cite[Theorem~1.2]{MR19} states that \eqref{eq:minimalgoodspeed} holds for any non-degenerated measure $\mu$. However, observe that \eqref{eq:minimaltriangularspeed} does not necessarily holds, as $\E(\xi) < \infty$ is a necessary condition for the existence of infinite triangular words.
\end{remark}

\section{Analytic properties of the asymptotic length of the longest path}
\label{sec:analytic}

In this section we use the results of the previous section on the infinite bin model to derive properties of the asymptotic length of the longest path in Barak-Erd\H{o}s graphs by coupling the latter to an infinite bin model in the case when $\mu$ is a geometric distribution. Let $\mu_p(k)=(1-p)^{k-1}p$, $k \in \N$. Let $(X_n)$ be the Markov process constructed in Section~\ref{BEIBM} from $\vec G(\Z, p)$. Then $(X_n)$ is an IBM$(\mu_p)$. Looking at \eqref{Xform} we see that
\begin{equation}
  \label{eq:lengthfront}
  F(X_n) = L_n,
\end{equation}
where $L_n$ is the maximum of all paths  in $\vec G(\Z, p)$ with endpoints in
$[1,n]$.
By Corollary \ref{aLL}, we have that $L_n/n \to C(p)$ a.s.
whereas Theorem \ref{thm:speed} shows that $F(X_n)/n \to v_{\mu_p}$ a.s. We
conclude that
\begin{equation}
\label{eq:Candv}
C(p)=v_{\mu_p}.
\end{equation}
We now make use of the results and formula obtained in Section \ref{speedexpress} to prove the analyticity of $C$ on $(0,1]$ in Section~\ref{subsec:analyticityOfC}, and to express its power series expansion in Section~\ref{subsec:powerseries}. But first, we rephrase the results of Section~\ref{speedexpress} for a geometrically distributed infinite bin model.

\subsection{First formulas for \texorpdfstring{$C$}{C}}
\label{subsec:firstFormula}

Using Corollary~\ref{cor:stationaryspeed} and the geometric distribution of $\xi$, we observe that $C(p)$ can be computed rather explicitly in terms of the stationary version $Y$ of the IBM($\mu_p$).

\begin{corollary}
\label{cor:Cexpectation}
Let $0<p\leq 1$, we denote by $R_p$ a random variable distributed as the content of the front bin in the stationary IBM($\mu_p$). Then
\begin{equation}
\label{eq:Cexpectation}
C(p)=1-\E\left((1-p)^{R_p}\right).
\end{equation}
\end{corollary}

\begin{proof}
Let $\xi$ be a random variable with geometric distribution $\mu_p$ which is independent of the stationary IBM($\mu_p$) $Y$. It follows from Corollary~\ref{cor:stationaryspeed} that
\[
  v_{\mu_p}=\P(\xi\leq Y_0(0))=1-\P(\xi> R_p).
\]
Conditioning on the value of $R_p$ and using formula \eqref{eq:Candv} yields formula \eqref{eq:Cexpectation}.
\end{proof}

We also apply Theorem~\ref{thm:sumoverwords} to obtain formulas for $C(p)$ as the sum of weights of well-chosen set of words. To this end we need the notion of height of a word.

\begin{definition}[Height]
The \emph{height} of $\alpha=\alpha_\ell\cdots\alpha_1\in\calU$ is defined as
\[
  H(\alpha):=\sum\limits_{i=1}^\ell (\alpha_i-1) =\sum\limits_{i=1}^\ell \alpha_i -|\alpha|.
\]
\end{definition}

With this notation, we observe that
\[
  w_{\mu_p}(\alpha) = \prod_{j=1}^\alpha \mu_p(\alpha_j) = p^{|\alpha|}(1 - p)^{H(\alpha)}.
\]
We immediately obtain the following restatement of Theorem~\ref{thm:sumoverwords} for the special case $\mu = \mu_p$.

\begin{theorem}
\label{thm:Csumoverwords}
Let $0<p\leq 1$. Then
\begin{align}
C(p)&=\sum\limits_{\alpha\in \calT_{\min}\cap \calG} p^{|\alpha|}(1-p)^{H(\alpha)}
\label{eq:Cminimaltriangular} \\
C(p)&=\sum\limits_{\alpha\in \calG_{\min}} p^{|\alpha|}(1-p)^{H(\alpha)}.
\label{eq:Cminimalgood}
\end{align}
\end{theorem}

\paragraph{Precise bounds.}
We use Theorem \ref{thm:sumoverwords} to obtain precise bounds on $C(p)$. Let $k\geq1$ and let $0<p\leq1$. Recalling the notation $\xi^j_k = \xi \1_{\xi \le k} + j \1_{\xi>k}$ from Section \ref{subsec:speed}, and assuming that $\xi$ has geometric law $\mu_p$, we let $\mu_{p,k}^j$ be the law of $\xi^j_k$.

Recall that IBM$(\mu^\infty_{p,k})$ bounds IBM$(\mu)$ from below, and that IBM$(\mu^k_{p,k})$ bounds IBM$(\mu)$ from above, in the sense of \eqref{ALLIBMS}.
Denote by $\underline C_k(p)$, $\overline C_k(p)$, the speeds of the IBM($\mu^\infty_{p,k}$), IBM($\mu^k_{p,k}$), respectively. We then have
\begin{equation}
\label{eq:Cgeneralbounds}
\forall p \in [0,1], \quad \underline C_k(p) \leq C(p) \leq \overline C_k(p).
\end{equation}

As the IBM($\mu^\infty_{p,k}$) and the IBM($\mu^k_{p,k}$) can be constructed using Markov chains on a finite state space, with transition probabilities that are polynomial functions of $p$, their speeds $\underline C_k(p)$ and $\overline C_k(p)$ are rational functions in $p$ that can be computed explicitly. For example, with $k = 3$, performing such computations yields
\begin{align*}
\underline C_3(p)&=\frac{p(p^2-3p+3)^2(p^4-6p^3+14p^2-16p+8)}{3p^6-26p^5+96p^4-196p^3+235p^2-158p+47} \\
\overline C_3(p)&=\frac{p^3-2p^2+p-1}{p^5-4p^4+8p^3-9p^2+6p-3}.
\end{align*}

It is worth noting that \eqref{eq:upperlowerbounds} implies that
\[
  0\leq\overline C_k(p)-\underline C_k(p)\leq (1-p)^k,
\]
therefore the sequences of functions $(\overline C_k(p))_{k\geq1}$ and $(\underline C_k(p))_{k\geq1}$ converge exponentially fast to $C(p)$, uniformly on every interval of the form $[\epsilon,1]$ with $\epsilon>0$. As we will see in Section~\ref{subsec:powerseries}, this convergence is so fast when $p$ is close to $1$ that they provide many coefficients of the power series expansion of $C(p)$ at $p=1$. In fact, comparing the asymptotic expansions of $\underline{C}_3$ and $\bar{C}_3$ around $p = 1$ already give
\[
  C(p) = 1 - (1-p) + (1-p)^2 - 3(1-p)^3 + 7 (1-p)^4 + O((1-p)^5) \text{ as $p \to 1$}.
\]

\begin{figure}[ht]
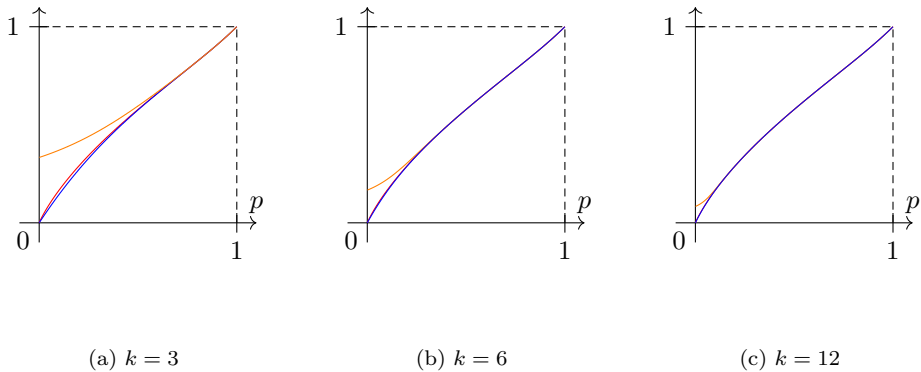

\centering
\begin{subfigure}{.28\linewidth}
\centering
\include{figTex/boundC3}
\caption{$k=3$}
\end{subfigure}
\begin{subfigure}{.28\linewidth}
\centering
\include{figTex/boundC6}
\caption{$k=6$}
\end{subfigure}
\begin{subfigure}{.28\linewidth}
\centering
\include{figTex/boundC12}
\caption{$k=12$}
\end{subfigure}
\caption{Successive bounds on the function $C$ (in red) by $\overline{C}_k$ (in
orange) and $\underline{C}_k$ (in blue). Observe that if these bounds appear
very sharp for $p$ close to $1$, the approximation (in particular the upper
bound) remains quite crude for $p$ close to $0$.}
\label{fig:approximations}
\end{figure}

\begin{remark}
Writing the balance equations for the stationary version $Y$ of the IBM($\mu_p$), Foss and Konstantopoulos \cite{FK03} obtained a different upper and lower bound for $C(p)$. This result on a more precise upper bound for $C(p)$ for $p$ close to $0$, however this bound still does not allow the capture of the asymptotic behavior of $C$ as $p \to 0$.
\end{remark}

\subsection{Analyticity of \texorpdfstring{$C$}{C}}
\label{subsec:analyticityOfC}

Using the formulas obtained above for the function $C$ resulting from the coupling with the infinite bin model, we can show that the function $C$ is analytic on $(0,1]$.

\begin{theorem}
\label{thm:Canalytic}
The function $p\mapsto C(p)$ is analytic on $\in(0,1]$.
\end{theorem}

\begin{proof}
For every $p,q\geq0$ define
\begin{equation}
\label{eq:Ddefinition}
D(p,q) = \sum\limits_{\alpha\in \calT_{\min}\cap \calG} p^{|\alpha|}q^{H(\alpha)}.
\end{equation}
By formula \eqref{eq:Cminimaltriangular}, we have that $C(p)=D(p,1-p)$ for every $0<p\leq1$. Let $0<p_0\leq1$. The rest of the proof consists in the construction of $(p',q')$ such that $p'> p_0$, $q'>1-p_0$ and the series \eqref{eq:Ddefinition} for $D(p',q')$ converges. This will imply the normal convergence of $D(p,q)$ on $[0,p']\times [0,q']$, therefore of the series of derivatives of~\eqref{eq:Cminimaltriangular} around $p_0$, which in turn will imply the analyticity of $C$ in a neighborhood of $p_0$.

Let $0<p\leq1$ and let $(\xi_n)_{n\in\Z}$ be i.i.d.\ of law $\mu_p$. We claim that, for all $r >1$,
\[
  D(rp,1-p)  \le \E(r^{2-T_0}),
\]
where $T_0$ is the largest nonpositive point of the set $\mathscr R$ defined in \eqref{RRR}. To see this, recall the time $T^*_{\calT} = \sup\{-\infty \le t \le 1:\, \xi_1 \xi_0 \cdots \xi_t \in \calT\}$ and recall that $\P(T^*_{\calT}>-\infty)=1$ (as in the proof of Theorem~\ref{thm:sumoverwords}) and that $\P(T^*_{\calT} \ge T_0)=1$. For brevity, set $T=T^*_{\calT}$.
Arguing as in \eqref{vmu} we have
\begin{multline*}
\E (r^{2-T_0}) \ge \E(r^{2-T}) \ge
\E\left( r^{2-T} ;\, {\xi_1 \xi_0 \cdots \xi_{T}
\in \calT_{\min} \cap \calG} \right)
= \sum_{\alpha \in \calT_{\min} \cap \calG}
\E\left(r^{2-T} ;\, {\xi_1 \xi_0 \cdots \xi_{T}=\alpha}\right)
\\
= \sum_{\alpha \in \calT_{\min} \cap \calG}  \E\left(r^{|\alpha|} ;\,
{\xi_1 \xi_0 \cdots \xi_{2-|\alpha|}=\alpha, 2-T=|\alpha|}\right)
= \sum_{\alpha \in \calT_{\min} \cap \calG}
r^{|\alpha|} \P\left({\xi_1 \xi_0 \cdots \xi_{2-|\alpha|}=\alpha}\right)
\\
= \sum_{\alpha \in \calT_{\min} \cap \calG} r^{|\alpha|}
p^{\alpha_1}(1-p)^{\alpha_1-1}\cdots p^{\alpha_\ell}(1-p)^{\alpha_\ell-1}
= D(rp,1-p).
\end{multline*}

We now observe that $|T_0|$ has some finite exponential moments. Indeed, $|T_0|$ can be seen as the first return time to $1$ of the Markov chain defined by $Z_0=1$ and $Z_{n+1} = \max(Z_n - 1,\xi_{1-n})$. This Markov chain can straightforwardly be dominated by a downward-skip free random walk with negative drift, yielding the existence of these exponential moments. We refer to \cite[p14, proof of Theorem~1.1]{MR19} for extra details on this proof.

As a consequence, there exists $r > 1$ such that for all $\frac{p_0}{2} < p < 1$, we have $\E(r^{|T_0}|)<\infty$. Choose $p$ such that $\max\left(\frac{p_0}{2},\frac{p_0}{r}\right)<p<p_0$. Then setting $p'=rp$ and $q'=1-p$ we have that $p'>p_0$, $q'>1-p_0$ and the series \eqref{eq:Ddefinition} for $D(p',q')$ converges.
\end{proof}

\begin{remark}
We observe that $C'(0) = e$ and $C''(0) = \infty$ (see forthcoming Section~\ref{sec:sparse}). Therefore the analyticity of $C$ cannot be extended up to $p=0$. \end{remark}

\subsection{Power series expansion of \texorpdfstring{$C(p)$}{C(p)} around \texorpdfstring{$p=1$}{p=1}}
\label{subsec:powerseries}

We use in this section Formula \eqref{eq:Cminimalgood} to prove that the power series expansion of $p \mapsto C(p)$ around $p=1$ only consists of integer coefficients. We write $q = 1-p$, and expand $C(1-q)$ as a power series in the variable $q$. Recall that $H(\alpha) = \sum_{i=1}^{|\alpha|} (\alpha_i-1)$ is the ``height'' of the word $\alpha$. We begin with the following observation.
\begin{lemma}
\label{lem:lengthbound}
For every $\alpha\in\calG_{\min}$, we have $|\alpha|\leq H(\alpha)+1$.
\end{lemma}

\begin{proof}
Let us assume that there exists $\alpha \in \calG_{\min}$ such that $|\alpha| > H(\alpha)+1$. We obtain a contradiction by showing that $\alpha$ will possess a strict suffix in $\mathcal{T}$ (hence in $\mathcal{G}\cup \mathcal{B}$), which violates the assumption that $\alpha$ is $\mathcal{G}$-minimal.

Let $S(k):=\sum_{i=1}^k (\alpha_i-2)$, $1\le k \le |\alpha|$, we remark that $S(|\alpha|) = H(\alpha)-|\alpha|<-1$. We denote by
$n=\max\{1 \le k \le |\alpha|: S(k)<-1\}$, and remark that $S(k) \leq -2$ for all $n \leq k \leq |\alpha|$. In particular, as $S(1) = \alpha_1- 2 \geq -1$, we have $n \geq 2$.

As $\alpha_k \geq 1$, we have $S(k) - S(k-1) \geq -1$ for all $2 \leq k \leq |\alpha|$. Therefore, $S(n) = -2$ and for all $k \leq n - |\alpha|$, we have $S(n+k) \geq -k - 2$. As a result, $\alpha_{n+k} = 2 + S_{n+k} - S_{n+k-1} \leq 2 - 2 - (-(k-1)-2) \leq k+1$. As a result, we now have proved that $\alpha_{|\alpha|}\cdots\alpha_n$ is a triangular word which is a strict suffix of $\alpha$, completing the proof by contradiction as mentioned above.
\end{proof}

We adopt the convention that for $\ell\in\N$ and $k\in\Z$, the binomial
coefficient $\binom{\ell}{k}$ vanishes whenever $k<0$ or $k>\ell$. We now use \eqref{eq:Cminimalgood}, and show that the power series expansion obtained around $p=1$ by rearranging its terms has positive radius of convergence, which completes the proof of the main result of the section.

\begin{theorem}
[\cite{MR16}]
\label{thm:powerseries}
For every $n\geq0$, define

\begin{equation}
\label{eq:ancleandef}
a_n:= \sum\limits_{\alpha\in\calG_{\min}}
\binom{|\alpha|}{n-H(\alpha)}(-1)^{H(\alpha)}.
\end{equation}
Then for every $0\leq q<\tfrac{\sqrt{2}-1}{2}$ we have
\begin{equation}
\label{eq:powerseries}
C(1-q)=\sum\limits_{n\geq0} (-1)^n a_n q^n.
\end{equation}
\end{theorem}
The fact that $a_n$ introduced in \eqref{eq:ancleandef} is
well-defined for all $n \in \Z_+$ is a consequence of Lemma~\ref{lem:lengthbound}. Indeed, for every $h\geq0$
and $\ell\geq1$ define
\[
\calU_\ell^h:=\{\alpha\in\calU, |\alpha|=\ell \text{ and } H(\alpha)=h\}.
\]
We can consider this as the number of arrangements of $h$
unlabelled balls into $\ell$ labelled boxes. Therefore, its cardinal is given by
\begin{equation}
\label{eq:compositions}
|\calU_\ell^h|=\binom{h+\ell-1}{\ell-1}.
\end{equation}
Any word $\alpha$ having a non-zero contribution in the sum defining $a_n$ must satisfy $H(\alpha)\leq n$ and $|\alpha|\leq
H(\alpha)+1\leq n+1$, where the latter condition follows from
Lemma~\ref{lem:lengthbound}. An equivalent formulation of~\eqref{eq:ancleandef}
is
\begin{equation}
\label{eq:anmessydef}
a_n=\sum\limits_{h=0}^n\sum\limits_{\ell=1}^{n+1}(-1)^{h}\binom{\ell}{n-h}|\calG_{\min}\cap\calU_\ell^h|
\end{equation}
which is in particular clearly finite.

\begin{proof}[Proof of Theorem~\ref{thm:powerseries}]
Let $n\geq0$ and let $0\leq q<1$. Rewrite formula~\eqref{eq:Cminimalgood} as
\begin{align}
C(1-q)=
\sum \limits_{h\geq0} \sum\limits_{\ell=1}^{h+1}
\sum\limits_{\alpha\in\calG_{\min}\cap\calU_\ell^h}	q^h(1-q)^\ell
&= \sum\limits_{h\geq0} \sum\limits_{\ell=1}^{h+1}
|\calG_{\min}\cap\calU_\ell^h| \sum\limits_{i=0}^{\ell} \binom{\ell}{i}(-1)^{i}q^{i+h}
\nonumber\\
&= \sum\limits_{h\geq0} \sum\limits_{n\geq0} \sum\limits_{\ell=1}^{h+1}
|\calG_{\min}\cap\calU_\ell^h| \binom{\ell}{n-h}(-1)^{n-h}q^n,
\label{eq:tripleseries}
\end{align}
For all $q$ such that \eqref{eq:tripleseries} absolutely converges, we can apply Fubini's
theorem to obtain that \eqref{eq:powerseries} holds, where $a_n$ given by \eqref{eq:anmessydef} or equivalently by \eqref{eq:ancleandef}.
Taking absolute values inside the sums of~\eqref{eq:tripleseries} we obtain
\begin{equation}
\label{eq:absolutesum}
I_q := \sum\limits_{h\geq0} \sum\limits_{n\geq0} \sum\limits_{\ell=1}^{h+1}
|\calG_{\min}\cap\calU_\ell^h| \binom{\ell}{n-h} q^n
= \sum\limits_{h\geq0} \sum\limits_{\ell=1}^{h+1}
|\calG_{\min}\cap\calU_\ell^h| q^h(1+q)^\ell.
\end{equation}
By \eqref{eq:compositions}, we have
\[
  I_q \leq (1+q) \sum\limits_{h\geq0} \sum\limits_{\ell=0}^{h} \binom{h+\ell}{\ell} q^h(1+q)^\ell.
\]
We use a random walk representation to find some values of $q$ for which this is
finite. Let $(S_N)_{N\geq0}$ be a random walk on $\Z$ starting at $0$ and taking
a step $+1$ (resp. $-1$) with probability $\tfrac{q}{1+2q}$ (resp.
$\tfrac{1+q}{1+2q}$). Performing the change of variables $N=h+\ell$, we have
\begin{equation}
\label{eq:rwrepresentation}
\sum\limits_{h\geq0}\sum \limits_{\ell=0}^{h}
\binom{h+\ell}{\ell} q^h(1+q)^\ell
=\sum\limits_{N\geq0}(1+2q)^N \P(S_N\leq0).
\end{equation}
Applying Chernoff's bound, we get
\[
\P(S_N\leq0) \leq \inf\limits_{t\geq0}\left(\E(e^{-tS_1})\right)^N
\leq \inf\limits_{t\geq0}\left(\frac{qe^t+(1+q)e^{-t}}{1+2q}\right)^N
\leq\left(\frac{2\sqrt{q(1+q)}}{1+2q}\right)^N.
\]
Thus, when $0\leq q<\tfrac{\sqrt2-1}{2}$, the series
in~\eqref{eq:rwrepresentation} converges and the series
in~\eqref{eq:tripleseries} converges absolutely.
\end{proof}

\begin{remark}
Lemma~\ref{lem:lengthbound} and Theorem~\ref{thm:powerseries} hold true with the
same proofs if one replaces $\calG_{\min}$ by $\calT_{\min}\cap\calG$, thus for every
$n\geq0$ we also have the formula
\begin{equation}
\label{eq:antriangular}
a_n=\sum\limits_{\alpha\in\calT_{\min}\cap\calG}\binom{|\alpha|}{n-H(\alpha)}(-1)^{H(\alpha)}.
\end{equation}
The power series expansion being unique, we observe that \eqref{eq:ancleandef} and \eqref{eq:antriangular} give the same values.
\end{remark}

\begin{remark}
Note that the radius of convergence of $\frac{\sqrt{2}-1}{2} \approx 0.207$ obtained in Theorem~\ref{thm:powerseries} is far from optimal, being obtained by the crude bound $|\mathcal{G}_{\min} \cap \mathcal{U}^h_\ell| \leq |\mathcal{U}^h_\ell|$. Based on the numerical computation of the first few terms of $\tfrac{a_{n+1}}{a_n}$ and $a_n^{1/n}$, it is reasonable to expect that the radius of convergence of this power series is  greater than $0.5$, but strictly smaller than $1$.
\end{remark}

In order to compute the first terms of the sequence $(a_n)$, three methods have
mainly been used so far. The first method is to use the
bounds~\eqref{eq:Cgeneralbounds} for some small values of $k$. Both
$\underline{C}_k(p)$ and $\overline{C}_k(p)$ arise as speeds of infinite bin
models associated to probability measures supported on $[0,k]\cup\{\infty\}$,
which are Markov chains on finite state spaces with a stationary distribution
that can be computed explicitly. A finite number of coefficients of the Taylor
expansion of $\underline{C}_k(p)$ and $\overline{C}_k(p)$ at $p=1$ coincide,
hence are coefficients of the Taylor expansion of $C(p)$ at $p=1$.
The first $17$ values of $a_n$ are computed in \cite{MR16}.
However, the size
of the state space of the Markov chain grows exponentially fast with $k$, and
the computations have to be made analytically, this method quickly becomes
computationally challenging.

The second method to compute $a_n$ consists in constructing the sets $\calU_\ell^h$ for small values of $h$
then using formula~\eqref{eq:anmessydef}. One may combine both methods, using a
formula analogous to~\eqref{eq:anmessydef} to obtain the beginning of the power
series expansions $\underline{C}_k(p)$ and $\overline{C}_k(p)$. They are
expressed as sums over words constrained to have letters at most equal to $k$.
Retaining the terms that coincide for the lower and upper bounds give terms for
$C(p)$. This last method was used in \cite{terlat} to obtain the first $24$
values of $a_n$. See Table~\ref{tab:anvalues} for the first few values of $a_n$. This sequence is referenced as A321309 in the On-Line Encyclopedia of Integer
Sequences~\cite{OEIS}.
\begin{table}[h]
\centering

\caption{The values of $a_n$ for $0\leq n\leq 12$.}
\label{tab:anvalues}
\begin{tabular}{|c|c|c|c|c|c|c|c|c|c|c|c|c|c|}
\hline
$n$ & 0 & 1 & 2 & 3 & 4 & 5 & 6 & 7 & 8 & 9 & 10 & 11 & 12 \\
\hline
$a_n$ & 1 & 1 & 1 & 3 & 7 & 15 & 29 & 54 & 102 & 197 & 375 & 687 & 1226 \\
\hline
\end{tabular}
\end{table}

From the observation of the first terms of the sequence, it was conjectured in \cite{MR16} that the sequence $(a_n)$ was non-decreasing and non-negative. Very recently, during the revision stage of this survey, Terlat \cite{Terlatthesis} found a third method to compute this sequence. This new method yields a dozen new terms and disproves the above conjectures.

This third method is based on an efficient algorithm to compute the Taylor expansion around $p=1$ of the maximal path growth rate $C(p,x)$ for the two-weights model presented in Subsection~\ref{subsec:twoweight}, in the case when $x$ is a finite negative integer. This two-weight model corresponds to last passage percolation on the complete directed graph on $\mathbb Z$ where each edge has weight $1$ with probability $p$ or $x$ with probability $1-p$. Terlat \cite{Terlatthesis} shows that the functions $C(p)$ and $C(p,x)$ have the same Taylor expansions around $p=1$ up to order $k$, as soon as $0\leq k < \binom{3-x}{2}$. Applying this to $x=-6$, one computes the values of $a_n$ up to $n=35$. One observes in particular that $a_{26}=683794$, $a_{27}=487644$ and $a_{28}=-425932$. This disproves both conjectures about non-decreasing absolute values and non-negativity.

\section{Longest path of the Barak-Erd\texorpdfstring{\H{o}}{o}s graph in the sparse regime}
\label{sec:sparse}
In this section we explore the asymptotic properties of the length of long paths
in a Barak-Erd\H{o}s graph $\vec G(\N,p)$
in the sparse graph limit, that is, when $p \to 0$.
It can be seen that, in this limit, $\vec G(\N,p)$ is well-approximated by a
{\em branching random walk}, a discrete-time particle system on the positive
half-line $\R_+$.
Throught this section, we will let $G$ denote the $\vec G(\N,p)$.

Let $L_n(p)$ be the maximum length of all paths in $\vec G(\N,p)$ from $1$ to $n$.
Using branching random walk approximation,
Newman \cite{NEWM92} obtained the lead order of the
asymptotic behavior of $L_n(p_n)$ when $p_n \to 0$ as $n \to \infty$.
He showed in particular that
\begin{equation} \label{eqn:firstOrder}
\lim_{n \to \infty} \frac{L_n(p_n)}{n p_n}
= \e \quad \text{in probability,}
\end{equation}
as long as $p_n \to 0$ and $n p_n\to \infty$. He also obtained the asymptotic
behavior of the overall longest path in that graph ($L_{0,n}$ with the notation of \eqref{LR}).

Recalling that $C(p)$ is the limit of $L_n(p)/n$ as $n \to \infty$,
Mallein and Ramassamy \cite{MR19} obtained the precise asymptotic
behavior of $C(p)$ as $p \to 0$ by comparing IBM$(\nu_p)$
(the infinite bin model with geometric distribution $\nu_p$ with small $p$)
to a continuous-time branching random walk with selection.
Precisely, \cite{MR19} states that
\begin{equation} \label{eqn:smallC}
\frac{C(p)}{p} = \e - \frac{\pi^2\e}{2} (\log p)^{{\red - }2}(1 + o(1)),
\quad\text{as $p \to 0$},
\end{equation}
using the so-called Brunet-Derrida behavior \cite{BD09,BeG10} of the speed of
branching random walks with selection that we now describe.

A {\em branching-selection process}
is a particle system in which each particle moves
and reproduces independently, but an exterior selection mechanism keeps the size
of the total population close to $N$ by killing particles.
\footnote{The most classical model is the $N$-branching Brownian motion (N-BBM)
defined as follows.
At each time $t>0$, there are $N$ particles on the real line.
The particles move according to i.i.d.\ Brownian motions.
At independent exponential times of parameter $N$, the leftmost particle
is killed and one of the $N-1$ other particles gives birth to a new particle
at at its currently occupied position. This model
was notably studied in \cite{Mai16} in which the speed and fluctuation of
the cloud if particles as $N \to \infty$ is obtained.
}

 Brunet and Derrida
\cite{BD09} conjectured, through numerical simulations and the study of exactly
solvable models, that for a large class of branching-selection processes, the
speed of the cloud of particles $v_N$ converges to its limit $v_\infty$ at a
slow rate, such that
\begin{equation}
  \label{eqn:bdbehavior}
  v_\infty - v_N = \frac{C(1+o(1))}{(\log N)^2}.
\end{equation}

\begin{remark}
Belief in this conjecture was increased by the study of an exactly solvable
model \cite{BDMM07,CorMal17}.
This type of behavior was  observed by Berestycki, Berestycki and Schweinsberg
\cite{BBS} for branching Brownian motions with absorption, Bérard and Gouéré
\cite{BeG10} for branching random walks, and for noisy F-KPP
\footnote{Equations of this type are partial differential equations of
the form $v_t=v_{xx} + F(v)$
were introduced by Kolmogorov, Petrovsky and Piscounov
\cite{KPP37} as models for a reaction-diffusion systems.
The name Fisher was added to these three names,
whence the acronym F-KPP, owing to Fisher's infamous work \cite{FIS37},
a paper cited in \cite{KPP37} also.
A duality relationship between the F-KPP equation, in the $F(v)=v(1-v)$ case,
and the branching Brownian
motion was established by McKean \cite{McK75}.
Connections between the noisy F-KPP equation and the branching Brownian
motion with selection were obtained in \cite{DurRem11,DeMFer19}.
}
equations modeling e.g.\ directed polymers \cite{MMLQ11}  among many other examples.
\end{remark}

We present in the current paper an alternative, possibly simpler, construction
of the coupling used by Mallein and Ramassamy \cite{MR19} between the
IBM($\nu_p$) and an $N$-branching random walk, a discrete analog of the $N$-BBM.
We give in Section~\ref{subsec:branching} some heuristics motivating
the kind of limit that sparse Barak-Erd\H{o}s graph has.
This limit, being interpreted as a particular branching random walk
sometimes called PWIT (Poisson-weighted infinite tree) is discussed in Section \ref{PWIT}.
In Section~\ref{THEcoupling} we
introduce a coupling between the Barak-Erd\H{o}s graph, and the PWIT.
This enables us to explain and describe the results of \cite{NEWM92} and \cite{MR19}.
We also extend these results to some other
stochastic ordered graphs in Section~\ref{subsec:extensions}. We then turn to
computations of the length of the longest path of the Barak-Erd\H{o}s graph in
Section~\ref{subsec:longestPath} and the shortest path in
Section~\ref{subsec:shortestPath}.

\subsection{Heuristics on the sparse limit}
\label{subsec:branching}

It is now commonly known that the neighborhoods in many sparse random graphs,
among which Erd\H{o}s-Rényi graphs and configuration models, are
well-approximated by branching processes. See \cite{vdH17} and references therein.
For example, let ${G}(n,\lambda/n)$ be an Erd\H{o}s-Rényi random graph.
That is, on the set $\{1,\ldots,n\}$ a pair of points forms an
indirected edge with probability $\lambda/n$, independently from pair to pair.
Using the graph distance of the Erd\H{o}s-Rényi random graph, we can observe
that the set of points within finite distance from any fixed vertex converges
weakly, as $n \to \infty$, to a Galton-Watson tree with Poisson($\lambda$)
offspring distribution; see e.g.\ \cite{Cur22}.

Indeed, the number of neighbors of a given vertex $v^*$ is given by a binomial
distribution with parameters $n$ and $\lambda/n$ that converges to a
Poisson($\lambda$) distribution as $n \to \infty$. In turn, the number of
neighbors of a given neighbor, excluding the vertex $v^*$, is given by an
independent binomial distribution with parameters $n-1$ and $\lambda/n$, which
also converges to a Poisson($\lambda$) distribution. Moreover, as there is with
large probability a bounded number of vertices in the ball of radius $k$ of the
vertex $v^*$, the probability of observing a non-trivial cycle of bounded size goes to $0$ as $n
\to \infty$. This proves that any finite neighborhood of the vertex $v^*$
converges in distribution to a Galton-Watson tree.

Consider now a Barak-Erd\H{o}s graph on $\N$ rather than on $\Z$,
as we are interested in paths from fixed root, the vertex 1 in this case.
Denote this by
\[
G := \vec G (\N, p),
\]
letting $E(G)$
be the random set of its edges.  Assume that it is sparse;
that is, we are interested in the limit as $p \to 0$.
We are able to obtain a similar description of the neighborhoods of
the vertex $1$ in terms of a branching process. However, to take into account
the directed structure of the graph, we have to record in the limiting
branching process the label of the vertices we consider. To achieve this, consider instead the graph $\vec G (p\N, p)$,
where $p \N := \{pk:\, k \in \N\}$.
Think of the immediate neighbors of the root as a point process on $(0,\infty)$
and let $N$ be a standard Poisson point process on $(0,\infty)$. We then
have
\[
\sum_{k\in \N} \1_{(1,k) \in E(G)}\, \delta_{pk}
\xrightarrow{\rm d} N, \quad \text{ as } p \to 0.
\]
where $\xrightarrow{\rm d}$ denotes convergence in distribution.
Similarly, for any $x>0$,
the set of immediate neighbors of vertex $\lfloor x/p \rfloor p$,
considered as a point process, also converges to $N$ in
distribution.
It was shown in \cite{GAB,FK18} that the connected component of the root
converges weakly, as $p \to 0$, to the {\em Poisson-weighted infinite tree}
(PWIT). This process is a branching random walk in which at each generation, all
particles in the system give birth to children independently, such that the
children of a particle at position $x$ are positioned according to a Poisson
point process with unit intensity on $[x,\infty)$.

\subsection{The PWIT and some of its properties}
\label{PWIT}

The terminology PWIT was introduced by Aldous and Steele \cite{AS}.
We describe it as a Markovian particle system that we call {\em immortal
particles process}.
At time $0$ an immortal particle is born. The particle produces a child at each epoch of a standard Poisson process;
and, recursively, each of the offspring has the same reproduction law, independently.
A convenient way to capture the system, together with all connection
information, is by letting
\[
\N^*:= \bigcup_{n=0}^\infty \N^n,
\]
where $\N^0=\{\varnothing\}$ be the vertex set of a tree with
edges $(u,v) \in \N^* \times \N^*$ only when $v = u k$, the concatenation
of $u$ with a single integer $k$.
Recall that $\N^*$ is the set of words (=finite sequences) of positive
integers equipped with the concatenation operation. The trivial
word $\varnothing$ is the identity of the concatenation operation. We do
not give a special symbol to the edges of $\N^*$ as they are immediately
fixed through $\N^*$. The resulting object is a tree that is now
known as the Ulam-Harris tree.
To encode the PWIT, simply add weights to the edges by letting,
\[
X(\varnothing)=0,
\]
and then, for each $u \in \N^*$,
\[
0< X(u1)-X(u) < X(u2)-X(u) < X(u3)-X(u) < \cdots
\]
be the epochs of an independent copy of a  Poisson$(1)$ point process on $(0,\infty)$.
In our immortal particles interpretation, $\N^*$ is the set of (names of)
all particles that are born to the end of time
and $X(u)$ is simply the time at which particle $u$ is born.
Thus, for example, $X(2,5,3)$ is the time at which the $3$rd offspring of
the $5$th offspring of the $2$nd offspring of $\varnothing$ is born, and has
the distribution of the sum of $2+5+3=10$ i.i.d.\ exponential random variables.
If we let $|u|$ be the length of the word $u$ and $\|u\|$ the sum
of its elements as integers then particle $u$ is born at generation $|u|$
and has Gamma($\|u\|$) distribution.
\footnote{
\label{interpr}
There are other ways to visualize the PWIT. First, recall that
a  branching random walk BRW$(N)$
in discrete time with parameter (the distribution of)
a (finite or infinite) point process $N$ on the real line,
is created by letting a single particle
at stage $0$, located at point $x$,
die at stage $1$ and immediately be replaced by children located at the points
of $N+x$ (that is, the set of points of $N$ all translated by $x$).
All children behave exactly in the same manner, independently.
If $N$ is a standard Poisson process (whose points are interpreted as
spatial points here) is the parameter of a branching random walk,
then this branching random walk is the PWIT;
this is the {\em first interpretation}.
In this interpretation, $X(u)$ is the spatial location of particle $u$.
The {\em second interpretation} of a PWIT is as a so-called ``Poisson cascade''
in the physics literature \cite{IK12}: Let $N_t$, $t \ge 0$,
be a collection of i.i.d.\ standard Poisson processes. Interpreting
$N$ as a set of points, we let $V = \bigcup_{t \ge 0} N_t$
be a set of vertices, letting $(s,t)$ be an edge if $t \in N_s+s$.
The corresponding graph is a random forest and the connected component of
$0$ is distributed like the PWIT.
A {\em third interpretation} \cite{AS} is as a random metric space $(\N^*, d)$
where the metric $d$ is as follows. First let $d(u,uk) = X(uk)-X(u)$ and then,
for each $u, v \in \N^*$, let $u=w_0, w_1, \ldots, w_\ell=v$
be the necessarily unique path between $u$ and $v$,
and let $d(u,v) = \sum_{i=1}^\ell d(w_{i-1}, w_i)$.
}
\footnote{The reason that we discuss different interpretations of
the PWIT is because there exist results in the literature referring
to seemingly different , but in essence identical stochastic models
around the PWIT. For instance, if we consider the continuous-time Markovian
branching process with offspring distribution $\delta_2$ (the Yule process)
then we can construct the PWIT as a deterministic function of it.
We shall not explain this here.}

Let
$
\Pi=(\N^*, X)
$
denote the standard (unit rate) PWIT.
Note that for any particle $u \in \N^*$ the subtree rooted at
$u$ is also a PWIT after relabeling and time-shifting.
We let
\[
V_t := \{u \in \N^*:\, X(u) \le t\},
\quad
X_t:= \{X(u):\, X(u) \le t\},
\quad
\Pi_t = (V_t, X_t),
\]
be the induced subgraph of $\Pi$ on the set of vertices $u$
with $X(u) \le t$. Hence $\Pi_t$ describes the immortal particles
process up to time $t$.
Note that $\Pi_t$, $t \ge 0$, is Markovian.
If we forget the connections between particles
and only keep the information of the lengths of their labels, then
we obtain an IBM-type of model.
Indeed, letting
\begin{equation}
\label{Zbinl}
Z_t(\ell) := |\{u \in \N^*:\, X(u) \le t,\, |u|=\ell\}|,
\end{equation}
then
\[
Z_t=(Z_t(0), Z_t(1), \ldots ),\quad t \ge 0,
\]
is a continuous-time IBM model whose evolution is as follows.
Let, for each $t$, each of the particles $u$ in $\Pi(t)$
possess an independent exponential$(1)$ clock.
One of the clocks expires first; say that this clock
is possessed by a particle in bin $\ell$; then we
add a new particle in bin $\ell+1$.
Note that $Z_t$, $t \ge 0$, is also Markovian.
This process is an Uchiyama-type continuous-time branching
random walk \cite{Uch82} on $\Z_+$, initiated from a single particle at
position $0$ at time $0$.
Also note that
\[
|V_t| := |\{u \in \N^*:\, X(u) \le t\}| = \sum_{\ell \ge 0} Z_t(\ell),
\quad t \ge 0,
\]
is also Markovian with state space $\{1,2,\ldots\}$ and
transition rate $y$ from state $y$ to $y+1$ (the Yule-Furry pure birth
process).
Finally note that there is a front bin, namely
\begin{equation} \label{FFF}
F_t = \max\{|u|:\, X(u) \le t\} = \max\{\ell \in \Z_+:\, Z_t(\ell)>0\},
\end{equation}
which is the largest generation particle present in $\Pi_t$.
\footnote{The word ``generation'' may be confusing,
especially in the immortal particles process interpretation of
the PWIT. To avoid confusion, simply interpret the
phrase ``generation of particle $u$'' as $|u|$.}
So $Z_t(\ell)=0$ for all $\ell > F_t$.
We use the abbreviation PWIT-IBM for $Z_t$, $t \ge 0$
and note that it differs from $\Pi_t$, $t \ge 0$,
only by the absence of the connections information.

The total number of particles in every bin $\ell$, with $\ell \ge 1$,
of the PWIT-IBM at time $t$ grows exponentially fast.
\begin{lemma}
\label{lem:growthRate}
Let $f$ be a nonnegative measurable function on $\R$.
Then, for all $\ell \in \N$,
\[
\E \sum_{|u|=\ell} f(X(u)) = \E \sum_{|v|=\ell-1} \int_{0}^\infty f(X(v)+t) dt
= \int_0^\infty f(x) \frac{x^{\ell-1}}{(\ell-1)!} dx.
\]
\end{lemma}

Let us stress at this point that, with our encoding, particles live forever. In particular, $\{|u|=1\}$ is the set of children of the ancestor $\varnothing$ of the population, which arise at the epoch of a Poisson process of intensity $1$. This set is a.s. infinite, and $\E \sum_{|u|=1} f(X(u)) = \int f(x) \dd x$ by Campbell's formula, as stated above.

\begin{proof}
Let $u \in \N^\ell$. Then $u=vk$ for some $v \in \N^{\ell-1}$ and some
$k \in \N$. So
\begin{align*}
\E \sum_{|u|=\ell} f(X(u))
&= \E \sum_{|v|=\ell-1} \E[f(X(v) + X(vk)-X(v))|X(v)]
\\
&= \E \sum_{|v|=\ell-1} \E\left[\int_0^\infty f(X(v) + t) N(dt)\bigg|X(v)\right]
= \E \sum_{|v|=\ell-1} \int_{0}^\infty f(X(v)+t) dt,
\end{align*}
where the $N$ above is a Poisson$(1)$ point process on $(0,\infty)$,
independent of $X(v)$.
Iterating this we obtain
\[
\E \sum_{|u|=\ell} f(X(u))
= \int_{[0,\infty)^\ell} f(t_1+\cdots+t_\ell) \, dt_1 \cdots dt_\ell,
\]
and the last expression follows by a change of variables.
\end{proof}

\begin{corollary}
We have the following formulas for the expected number $\E Z_t(\ell)$ of particles
in bin $\ell$ as well as the total number $\E |V_t|$ of the PWIT-IBM at time $t$:
\[
\E Z_t(\ell) =\frac{t^\ell}{\ell!}, \quad \E |V_t| = e^t.
\]
\end{corollary}

\begin{proof}
From \eqref{Zbinl} and Lemma \ref{lem:growthRate} with $f(x)=\1_{x \le t}$
we have
\[
\E Z_t(\ell) = \int_0^t \frac{x^{\ell-1}}{(\ell-1)!} dx = \frac{t^\ell}{\ell!}.
\]
The second claim follows by summation or by remembering that $|V_t|$, $t \ge 0$,
is the Yule-Furry process.
\end{proof}

We are interested in the PWIT since, as motivated in Section
\ref{subsec:branching}, the PWIT appears as the limit of a sparse
Barak-Erd\H{o}s graph.
We thus proceed in outlining some results concerning the PWIT.
In Section~\ref{THEcoupling} we will couple the PWIT together
with the Barak-Erd\H{o}s graph for all $p$ (or more specifically with an appropriate spanning tree of the connected component
of the root of $G$).
This coupling will be such that
\begin{equation}
  \label{eqn:minBRW}
  M_\ell = \inf_{|u|=\ell} X(u),
\end{equation}
the first time that bin $\ell$ of the PWIT-IBM becomes nonempty,
will give a lower bound for the index of any vertex of
$G$ linked to the vertex $1$ by
a path of length $\ell$. It will be enough to obtain the upper bound for
\eqref{eqn:firstOrder}. This lower bound will be sharp enough in very sparse
graphs, but some additional approximations will be needed when the density of
edges becomes too large, yielding the estimate \eqref{eqn:smallC}.
In terms of the BRW$(N)$ interpretation--see footnote \ref{interpr}--the
quantity $M_\ell$ is the minimal displacement (position of the leftmost particle)
at stage $\ell$ and this has been the subject of a large body of work.
These asymptotic properties of the minimal displacement for a BRW$(N)$ depend
on the quantity
\[
\kappa(\theta) := \log \int_{\R} e^{-\theta x} \E N(dx)
\]
the logarithm of the Laplace transform of the mean measure of $N$.
The speed $v$ of BRW$(N)$ is then expressed as
\[
v = \sup_{\theta > 0} \frac{-\kappa(\theta)}{\theta}.
\]
Indeed, Hammersley \cite{Ham74}, Kingman \cite{Kin75}
and Biggins \cite{Big76} proved under increasing generality that
\begin{equation} \label{eqn:speedBRW}
\lim_{n \to \infty} \frac{M_n}{n} = v \quad
\text{a.s. and in $\mathrm{L}^1$}.
\end{equation}

In our case, $N$ being standard Poisson process, we have
\[
\kappa(\theta) =\log \int_0^\infty \e^{-\theta x} \dd x = - \log \theta,
\quad
v := \sup_{\theta > 0} \frac{\log \theta}{\theta} = \frac{1}{\e},
\]
This is the same $1/\e$ that appears in the limit \eqref{eqn:firstOrder}.

Addario-Berry and Reed \cite{BERREE} and Hu and Shi \cite{HuS09} independently
proved that $M_n-n/\e$ increases at logarithmic rate; more precisely in our case that
\begin{equation} \label{eqn:logCorrection}
\lim_{n \to \infty} \frac{M_n - n/\e}{\log n} = \frac{3}{2\e}, \quad
\text{ in probability,}
\end{equation}
with almost sure fluctuation occurring on that logarithmic scale. The
convergence in distribution of the
minimal displacement of a branching random walk, when centered around
its median, was then obtained by A\"id\'ekon \cite{Aid13}. Using that
\[
 D_n = \sum_{|u|=n} (n/\e - X(u)) e^{-\e X(u)}
\]
is a non-uniformly integrable signed martingale that converges almost surely
to a positive limit $D_\infty$, he proved that there exists $c_\star > 0$ such that for
all $x \geq 0$,
\begin{equation}
  \label{eqn:cvInLawBRW}
  \P(M_n \geq n/\e + \tfrac{3}{2\e} \log n + x) = \E\left( \exp\left( - c_\star D_\infty \e^{\e x}\right) \right)
\end{equation}
The result was independently obtained in \cite{BERFOR} in the specific case
of the PWIT.

Recalling the PWIT-IBM interpretation of $Z_t$, $t \ge 0$,
we note that the functions $t \mapsto F_t$ and $\ell \mapsto M_\ell$,
defined in \eqref{FFF} and \eqref{eqn:minBRW}, respectively, are
generalized inverses of one another.
Indeed, it is clear that, for all $t \ge 0$ and all $\ell \in \Z_+$,
\[
Z_t(\ell)> 0 \iff \exists u \in \N^\ell \, X(u) \le t
\iff M_\ell \le t,
\]
which implies that the front bin $F_t$ in the PWIT-IBM, as defined by
\eqref{FFF}, satisfies
\begin{equation}
\label{Ft}
F_t =  \max\{ \ell \in \Z_+:\, M_\ell \le t\}.
\end{equation}
Hence, from the asymptotic behavior of $M_\ell$ as $\ell \to \infty$ we are able to obtain the asymptotic
behavior of $F_t$ as $t \to \infty$. This method as already been used by
Corre in \cite{Cor17} in this purpose. However, note that in his description of
the Yule process, at each birthing event particles were dying giving birth to
two new children. Therefore our result does not align exactly with  the one of Corre.

\begin{lemma}[Corre \cite{Cor17}]
\label{lem:asympK}
For all $t \geq 1$, we set
\[ n_t = \floor{\e t - \frac{3}{2}\log t} \in \N \quad \text{and} \quad y_t = \e t - \frac{3}{2}\log t - n_t \in (0,1).\]
For all $k \in \Z$, we have
\[
  \lim_{t \to \infty} \P(F_t \leq n_t + k) - \E\left( \exp\left( - c_\star D_\infty \e^{y_t - k - 1}\right) \right) = 0.
\]
\end{lemma}
\begin{remark}
Lemma \ref{lem:asympK} shows that as $t \to \infty$, $F_t$ remains tight
around its median but does not converge in distribution, due to the
fluctuations of $\log t - \floor{\log t}$.
\end{remark}
\begin{proof}[Proof of Lemma \ref{lem:asympK}]
By definition of $F_t$, and the fact that $M_\ell$ is atomless and increases
as $\ell$ increases,
we have
\begin{equation*}
\label{eqn:temp}
\P(F_t \leq n_t + k) = \P\left( M_{n_t + k+1} > t \right)
= \P\left( M_{n_t + k+1} \geq t \right), \quad k \in \N.
\end{equation*}
In addition, since $x \mapsto \P(M_n \geq n/\e + \tfrac{3}{2\e} \log n + x)$
converges pointwise to a monotone decreasing continuous function
from $\R$ into $[0,1]$,
we deduce from Dini's theorem that this convergence is uniform.
In particular, for any bounded sequence $(x_n)$, we have
\[
\lim_{n \to \infty} \P(M_n \geq n/\e + \tfrac{3}{2\e} \log n + x_n) - \E\left( \exp\left( - c_\star D_\infty \e^{\e x_n}\right) \right) = 0.
\]
Fix $k \in \Z$ and observe that
\[
\frac{1}{\e} (n_t+k+1) + \frac{3}{2\e} \log (n_t+k+1)
= t + \frac{1}{\e}(k+1-y_t) + o(1), \quad \text{as } t \to \infty.
\]
As a result we obtain that, for all $k \in \Z$,
\[
\lim_{t \to \infty} \P(M_{n_t + k+1} \geq t)
- \E\left( \exp\left( - c_\star D_\infty \e^{y_t - k - 1}\right) \right) = 0,
\]
which completes the proof.
\end{proof}

\subsection{Coupling of the PWIT and the Barak-Erd\texorpdfstring{\H{o}}{o}s graph}
\label{THEcoupling}

We construct a coupling between the PWIT and the Barak-Erd\H{o}s graph such that the heuristic convergence of neighborhoods of vertex $1$ described in Section~\ref{subsec:branching} is more explicit as $p \to 0$.
Recall that $G$ denotes the $\vec G(\N, p)$ Barak-Erd\H{o}s graph.
We will describe the laws of two random subraphs, the connected component
$C$ of vertex $1$ (the root of $G$) and a special spanning tree
$B$ of $C$ that has the property that the path from every of its
vertices to the root has maximal length.
We will then change the vertex set of $C$ and, in a sense, put
$C$ in continuous time, explaining the ``correct'' time scale.
Having done this, we will explain how to construct $B$ and the PWIT
on the same probability space. We will embed $B$ into the PWIT
making sure that the edge-relationships are preserved and that
the locations (in continuous time) of vertices are also correct.

\subsubsection*{The law of the set of vertices of \texorpdfstring{$C$}{C}}
Let $q=1-p$.
Define $C$ as the subgraph of $G$ containing the root, i,e.\ vertex $1$,
and all vertices $k \ge 1$ such that $1 \leadsto k$ (i.e.\ there is a path in $G$ between $1$ and $k$). We define these vertices recursively, letting
\[
\kappa_0=1,\qquad
\kappa_i = \min\{k > \kappa_{i-1}:\, 1 \leadsto k\}, \quad i \in \N.
\]
Clearly,
\begin{equation}\label{k1}
\kappa_1-\kappa_0, \, \kappa_2-\kappa_1, \ldots \text{ are independent},
\end{equation}
and, for all $i \in \N$
\begin{equation}\label{k2}
\kappa_i-\kappa_{i-1} \text{ is geometric$(1-q^i)$};
\qquad \P(\kappa_i- \kappa_{i-1} > m) = (q^i)^m, \quad m \ge 0.
\end{equation}
This describes the law of the random set
\[
\K =\{\kappa_0, \kappa_1, \kappa_2, \ldots\} = V(C)
\]
of vertices of $C$ (and hence of any spanning tree of $C$).
We can think of $\K$ as an inhomogeneous renewal process on $\N$
that quickly converges to a homogeneous one. Indeed, $q^i \to 0$
as $i\to \infty$, so fast that the Borel-Cantelli lemma ensures that,
eventually, $\K$ contains all positive integers. This is another manifestation
of the existence of skeleton points, as described in Section \ref{ergosec}.
Note that the rates of the above geometric random variables are increasing with $i$,
indicating that, initially, $\K$ is a sparse set (the smaller the $p$ the
sparser the  $\K$ is initially).

\subsubsection*{The spanning tree $B$ and its law}
Let $L_k=L^{\gl,\gr}_{1,k}$, the maximum length of all paths from $1$ to $k$
in $G$. A spanning tree of $C$ is a tree whose set of vertices is $\K$.
We say that a spanning tree of $C$ is a {\em maximal length spanning tree}
if for all $\kappa_i \in \K$, the length of the (necessarily unique) path
from $1$ to $\kappa_i$ equals $L_{\kappa_i}$.
To define the special maximal length spanning tree $B$
we need a definition of ordering on the set
\[
\K_{i-1} :=\{\kappa_0, \ldots, \kappa_{i-1}\}
\]
of the first $i$ vertices of $C$.
Roughly speaking, we order the elements of $\K_{i-1}$ in decreasing length,
breaking ties in some way.

\begin{definition}[Vertex ranking]
\label{def:vr}
Fix $i\in \N$ and let $a,b$ be distinct elements of $\{0,1,\ldots,i-1\}$.
We say that
$\kappa_a \lhd \kappa_b$ (or $\kappa_b \rhd \kappa_a$),
if either $L_{\kappa_a} > L_{\kappa_b}$
or $L_{\kappa_a} = L_{\kappa_b}$ and $a > b$.
Note that $(\K_{i-1}, \lhd)$ is a totally ordered set and that $\lhd$ depends on
$i$
(that is, the order on $\K_{i}$ is different from the order on $\K_{i-1}$).
Let now
\[
\sigma: \{1,\ldots,i\} \to \{0, \ldots, i-1\}
\]
be the unique bijection so that
\[
\kappa_{\sigma(1)} \lhd \kappa_{\sigma(2)} \lhd \cdots \lhd \kappa_{\sigma(i)}.
\]
The domain of $\sigma$ is the set of ranks and $\kappa_{\sigma(1)}$ is
the vertex of rank $1$, $\kappa_{\sigma(2)}$ the vertex of rank 2, etc.
We always have $\kappa_{\sigma(i)}=\kappa_0$ and $\kappa_{\sigma(i-1)}=\kappa_1$.
We also let
\[
\rho := \sigma^{-1}
\]
denote the inverse function.
\end{definition}

\begin{definition}[Special spanning tree $B$]
The special spanning tree $B$ of $C$ is defined as being the maximal-length
spanning tree with the property that, for all $i$, the unique parent of
$\kappa_i$ in $B$ is the minimal element $\kappa_j$ for the order $\lhd$ among the elements such that $(\kappa_j, \kappa_i)$ is an edge in $G$.
\end{definition}
Figure \ref{oraio} gives an example of the definitions.

\begin{figure}[ht]
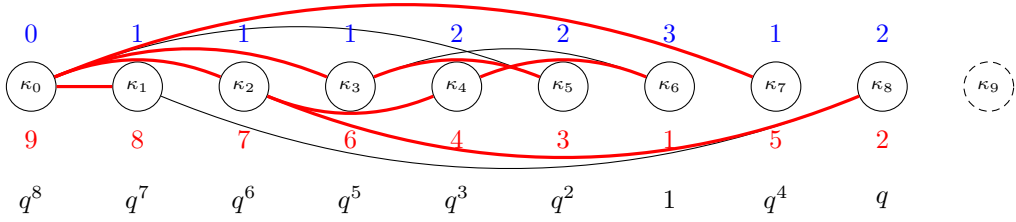

\centering
\include{figTex/berank}
\caption{The first $9$ vertices of $C$ are shown, together with the edges
between them. Vertices of $G$ not connected to $1$ via a path are not shown. The
number above each vertex $\kappa_j$ is the maximal path length $L_{\kappa_i}$.
The number below each vertex is its rank. The thick edges are the edges of the
special spanning tree $B$. To explain this, consider, e.g., vertex $\kappa_6$
and observe that is connected to $\kappa_4$ with $L_{\kappa_4}=2$ and to
$\kappa_3$ with $L_{\kappa_3}=1$; so we choose $\kappa_4$ as the parent of
$\kappa_6$ in $B_5$. Consider vertex $\kappa_8$. It is connected to vertices
$\kappa_2$ and $\kappa_1$ with $L_{\kappa_2}=L_{\kappa_1}=2$. Since $\kappa_2$
has lower rank than $\kappa_1$ we declare that $\kappa_2$ is the parent of
$\kappa_8$ in $B$. Finally, the quantities on the lower line are proportional to
the probability that $\kappa_9$ chooses its parent among the existing vertices, i.e it will be connected to e.g. $\kappa_8$ with probability $\frac{q}{1+q+q^2+q^3+q^4 +q^5+q^6+q^7+q^8}$.}
\label{oraio}
\end{figure}

To describe the law of the edges of $B$, we let, for each $i \in \N$,
the index $\pi(i)$ be such that $\kappa_{\pi(i)}$ is the parent of
$\kappa_i$ in $B$.
For any $A \subset \N$ let $G(A)$ denote the induced subgraph of $G$
on the set $A$ (i.e.\ the graph that contains as vertices the elements of $A$
and edges only those edges with endpoints in $A$).

\begin{lemma}  \label{Bedges}
With $\kappa_{\pi(i)}$ denoting the parent of $\kappa_i$ in $B$ we have
\begin{equation} \label{k3}
\P(\pi(i) = \sigma(r) | G(\K_{i-1})) = \frac{q^{r-1}(1-q)}{1-q^i}, \quad
r=1,\ldots, i.
\end{equation}
\end{lemma}
\begin{proof}[Sketch of proof]
To understand this, first note that $\sigma$ depends only on $G(\K_{i-1})$,
where $\K_{i-1}=\{\kappa_0, \ldots, \kappa_{i-1}\}$.
Second, this formula says that the parent of $\kappa_i$ is
the least ranked vertex $\kappa_j$, say, in $\K_{i-1}$ such that
$(\kappa_j, \kappa_i)$ is an edge in $G$.
The denominator $1-q^i$ expresses the probability that $\kappa_i$
connects to one of the vertices in $\K_{i-1}$.
\end{proof}

Figure \ref{oraio} shows these probabilities in an example
with $i=9$.

\begin{remark}
Note that \eqref{k1}, \eqref{k2} and \eqref{k3} provide a complete characterization of the law
of the special spanning tree $B$.
\end{remark}

\subsubsection*{The special tree $B$ in continuous time}
Putting $B$ in continuous time means replacing its vertex set $\K$
by a possibly random subset of $(0,\infty)$ in a way allowing its coupling with the PWIT. Let us start by recalling the clockwork lemma:
consider a finite set of positive numbers, say $a_1, \ldots, a_i$.
To simulate a random variable that takes value $j\in \{1,\ldots, i\}$
with probability proportional to $a_i$, we can use $\tau_1, \ldots, \tau_i$ independent exponential random variables with
parameters $a_1, \ldots, a_i$. Then the minimum of these random variables
equals $\tau_j$ with probability $a_j/\sum_{k=1}^i a_k$.

Fix $i$ and the set $\K_{i-1}=\{0,1,\ldots,i-1\}$.
We imagine that this set, along with the edges from $B$, has been constructed at
time $t$.
Equip each vertex $\kappa_j \in \K_{i-1}$
with an independent exponential random variable $\tau_j$, corresponding to the elapsed time
after time $t$ at which $\kappa_j$ wishes to connect to a new vertex $\kappa_i$.
We declare that this takes place at time $t+\tau$, where
$\tau=\min(\tau_1, \ldots, \tau_i)$  and if $\tau=\tau_j$ then it is
vertex $\kappa_j$ that will connect to $\kappa_i$.
Let $\rho(j)$ be the rank of $\kappa_j$, i.e.\ the inverse of $\sigma$ defined in Definition~\ref{def:vr}.
We should assign rate $q^{\rho(j)-1}$ to $\tau_j$, as then
\[
\P(\tau_j = \tau) = \frac{q^{\rho(j)-1}}{1+q+\cdots+q^{i-1}}
=  \frac{q^{\rho(j)-1}(1-q)}{1-q^i}, \quad 0 \le j \le i-1,
\]
which precisely equal to the probability \eqref{k3}.
If then we let $T(\kappa_i)$ be the time of appearance of $\kappa_i$
then we must have
\begin{equation}
\label{TTT}
T(\kappa_i)-T(\kappa_{i-1})
\eqdist \tau \sim \operatorname{exponential}(1+q+\cdots+q^{i-1}),
\end{equation}
that is,
\[
T(\kappa_i)-T(\kappa_{i-1})
\sim \operatorname{exponential}\left(\frac{1-q^i}{1-q}\right).
\]
Hence we can obtain $T(\kappa_i)$ by an appropriate thinning of
a Poisson process.
\begin{lemma}	\label{pxi}
Let $\xi_1 = 0$ and $0<\xi_2 < \xi_3 < \cdots$ be the points of a Poisson$(1)$ point process,
independent of the vertex set $\K=\{\kappa_0, \kappa_1, \ldots\}$ of $C$.
Then
\[
p \xi_{\kappa_i} \eqdist T(\kappa_i), \quad i=0,1,2,\ldots
\]
\end{lemma}
\begin{proof}
If $\tau_1, \tau_2, \ldots$ are i.i.d.\ exponential random variables with
rate $\lambda$ and if $\nu$ is an independent geometric random variable
with parameter $\alpha$ (i.e., $\P(\nu=n) = (1-\alpha)^{n-1}\alpha$, $n \in \N$),
then $\tau_1+ \tau_2 + \cdots + \tau_\nu$ is exponential with rate $\alpha \lambda$,
by the standard Poisson thinning theorem.
Recall, from eq.\ \eqref{k2},
that $\kappa_i-\kappa_{i-1}$ is geometric$(1-q^i)$.
Applying the above thinning theorem we have that
$\xi_{\kappa_i}-\xi_{\kappa_{i-1}}$ is exponential with rate
$(1-q^i)$ and so $p\xi_{\kappa_i}-p\xi_{\kappa_{i-1}}$ is exponential with rate
$(1-q^i)/p$.
Hence, from \eqref{TTT},
\[
p\xi_{\kappa_i}-p\xi_{\kappa_{i-1}} \eqdist T(\kappa_i)-T(\kappa_{i-1})
\]
and the claim follows by independent increments of both sides.
\end{proof}
\begin{remark}
This shows that the correct set of vertices of $B$, if we are to put
it in continuous time and create a Markov process, is the set
\[
0= p\xi_{\kappa_0} < p\xi_{\kappa_1} < p \xi_{\kappa_2} < \cdots
\]
To spell out the construction explicitly, we identify every integer vertex
$\kappa_i$ with the real number $p \xi_{\kappa_i}$
and keep the edges intact.
Let now $B_t$ be the induced subgraph of $B$
on the set of new vertices $p\xi_{\kappa_i}$ that do not exceed $t$.
Then $B_t$, $t \ge 0$, is Markovian.
\end{remark}

\subsubsection*{The coupling}
The quantities defined above allow us to give a constructive proof of the following theorem.
\begin{theorem}
\label{thm:coupling}
Let $B$ be the special spanning tree of $G$
and $\Pi$ a standard PWIT.
Then there is a probability space on which
$(\Pi, B)$ is defined in such a way that
there is an injection
\[
\Phi: \K \to \N^*,
\]
preserving edges and having the property that
\begin{equation} \label{PHIcorrect}
|\Phi(\kappa_i)| = L_{\kappa_i},
\end{equation}
for all $\kappa_i \in \K$.
\end{theorem}

\begin{proof}[Informal construction of the coupling.]
We simultaneously construct $B_t$ together with $\Pi_t$
at times $t$.
Observe that every particle in $\Pi_t$ gives birth to offspring at unit rate, while a particle $p \xi_{j}$ in $B_t$ are connected to newly discovered vertex at slower rate $1 - q^k$, where $k$ is the rank of $p \xi_j$ for the order $\lhd$. We construct the coupling by strategically selecting the particles in $\Pi$ corresponding to the vertices in $B$. For sake of illustration, we think of particles in $\Pi$ as being blue, unless they correspond to vertices of $B$ via $\Phi$, in which case they are red. We informally refer to particle $\Phi(\kappa_i)$ as vertex $\kappa_i$.

At time $t=0$ the situation is trivial, there is a single red particle in $\Pi_0$ corresponding to vertex $1$ in $B_0$. Fix $t > 0$, let $i$ be the number of red particles in $\Pi_t$, corresponding to vertices $\kappa_0, \ldots, \kappa_{i-1}$ in $B_t$. The remaining particles in $\Pi_t$ are blue. We associate to each blue particle an independent blue exponential clock with parameter $1$, and to each red particle a pair of independent exponential clocks, such that the rate of the red clock associated to the red particle ranked $r$ is taken to be $q^r$, and the blue clock is $1 - q^r$. These clocks correspond to the times at which each particle creates a child, with the color of the clock corresponding to the color of the newborn particle.

We consider the time $\tau$ corresponding to the smallest of these exponential clocks. At time $t + \tau$, if the clock is blue we add a new blue particle to $\Pi_t$ with position $t + \tau$, which is born from the particle associated to this exponential clock. If the clock at time $t + \tau$ is red, we add a new red particle to $\Pi_t$ with position $t+\tau$, connected to the (red) parent particle $\Phi^{-1}(\kappa_j)$. Simultaneously, we add to $B_{t+\tau}$ the vertex $\kappa_i$, which is connected to the newborn particle in $\Pi_{t+\tau}$ via $\Phi$.

Since red clocks have rate $1,q,\ldots, q^{i-1}$, it follows that the particle corresponding to the new vertex will be added after an exponential time of parameter $1 + \cdots + q^{i-1}$. Therefore, this time corresponds to $p\xi_{\kappa_i}$ in distribution, so we set $\xi_{\kappa_i} = X(\Phi^{-1}(\kappa_i))/p$. It then should be clear, by construction that $(\Pi_t)$ is the PWIT, $B$ has the law of the special spanning tree, and \eqref{PHIcorrect} holds, since every edge in $B$ correspond to a parent-child relationship in the PWIT.
\end{proof}

As Barak-Erd\H{o}s graphs with any connecting constant can be associated to the
same branching random walk, this coupling also yields a coupling of
Barak-Erd\H{o}s graphs with different constants. It is worth noting that with
this coupling, the set of red ``coding'' particles is decreasing with $p$. It is
also worth observing that typically, the position of the leftmost blue
descendant of a red particle tend to be of order $-\log p$, as newborn red
particle occur at rate $1$, and the rate at which new blue particle occur can be
bounded from above by $1 - (1-p)^n$ with $n$ the number of red particles born
after that particle.

\begin{figure}[ht]
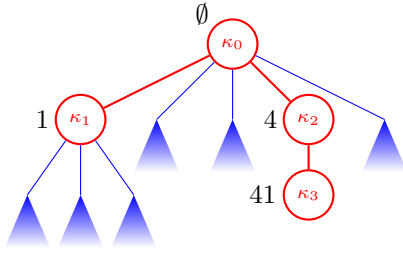

\centering
\include{figTex/scenaria3}
\caption{A possible evolution of the PWIT with the vertices of $\K$ embedded
correctly. The little triangle at the end of blue particles symbolize the blue
descendants produced by each blue particle, which play no role in the coupling,
hence do not interact with $\Phi$ or the times at which the red particles
appear. However, the blue particles represented here play a role, as they
influence the label given to each of the red particle.}
\label{stages}
\end{figure}

\begin{example}
We refer to the example drawn in Figure \ref{stages}.
We start with $\kappa_0$ as a root particle of the PWIT.
Since $\kappa_0$ is red and has rank $1$ its red clock has rate $q^{1-1}=1$,
so its blue clock has rate $1-1=0$. That is, $\kappa_0$, at this stage,
produces red offspring. The first offspring, $\kappa_1$, appears at an exponential
time with parameter $1+q+\cdots+q^{i-1}$ with $i=1$, hence at an exponential$(1)$
time, that can be taken to be $p\xi_{\kappa_1}$.
We set $\Phi(\kappa_1)=1$.
At this time, $\kappa_1$ has rank $1$ and $\kappa_0$ has rank $2$.
Therefore the red and blue clocks of $\kappa_1$ have rates $1, 0$, respectively,
while the red and blue clocks of $\kappa_0$ have rates $q, 1-q$ respectively.
Hence $\kappa_2$ appears at time differing from $p\xi_{\kappa_1}$
by an exponential$(1+q)$ random variable, that is, at time $p\xi_{\kappa_2}$.
Moreover, $\kappa_2$ has parent $\kappa_1$ with probability $1/(1+q)=
(1-q)/(1-q^2)$
or $\kappa_0$ with probability $q/(1+q) = q(1-q)/(1-q^2)$.
Suppose the latter happens. If $\kappa_0$ had produced $2$ blue offspring
then $\Phi(\kappa_2)=4$, as it is its third offspring.
Once $\kappa_2$ is added we have that $\kappa_2$ has rank $1$ so
has no blue clock, $\kappa_1$ has rank $2$ so it has a red clock at
rate $q$ and a blue at rate $1-q$, while $\kappa_0$ has the least rank, $3$,
and has a red clock at rare $q^2$ and a blue at rate $1-q^2$.
At time $p\xi_{\kappa_3}$ the new particle $\kappa_3$ is added and if it
chooses $\kappa_2$ as parent then its label is $\Phi(\kappa_2)1=41$, since
$\kappa_2$ produced no offspring, so $\kappa_3$ is its first one.
\end{example}

We may call the times $p \xi_{\kappa_i}$ at which a new vertex of $B$ appears
in $\Pi$ as embedded times. We define for $\ell \in \Z_+$
\[
  \bar{M}_\ell = \inf\{X(u) : u \text{ red particle with } |u|=\ell\}
\]
and for $t \geq 0$
\[
  \bar{F}_t = \max\{\ell \in \Z_+: \bar{M}_\ell \leq t\},
\]
which are the analogue of $(M_\ell)$ and $(F_t)$ defined in \eqref{Ft}, restricted to the set of red particles from the coupling.

\begin{corollary}
\label{cor:couplingResult}
The process $(\bar{F}_t, t \ge 0)$ and $(L_k, k \ge 0)$ where $L_k=L^{\gl,\gr}_k$, are related by
\[
  F_{p \xi_{\kappa_i}} = \max_{0 \le j \le i} L_{\kappa_j}.
\]
\end{corollary}

\subsection{Longest chain issued from the initial vertex}
\label{subsec:issuedFrom1}

The length of the longest path issued from the vertex $1$ is studied using the
branching random walk coupling described above. This coupling was introduced by Newman
\cite{NEWM92} for this purpose. In particular, using this coupling, he obtained
the following result for the asymptotic behavior of $L_n$ as a function of $p$.

\begin{theorem}[Newman \cite{NEWM92}]
\label{thm:start1}
Let $(p_n)$ be a null sequence.
\begin{enumerate}
\item If $\lim_{n \to \infty} n p_n = t \in (0,\infty)$, then $\lim_{n \to \infty} L^{(1)}_n(p_n) = F_t$ in distribution.
\item If $\lim_{n \to \infty} n p_n = \infty$, then $\lim_{n \to \infty} L^{(1)}_n(p_n)/np_n = \e$ in probability.
\end{enumerate}
\end{theorem}

\begin{proof}
We first assume that $n p_n \to t \in (0,\infty)$. In this situation, the
connected component of the vertex $1$ in the Barak-Erd\H{o}s graph
$\overline{\mathcal{G}}(n,p_n)$ converges, as $n \to \infty$ to $(X(u), u \in
\mathcal{U} : X(u) \leq t)$, using the branching random walk coupling.

Indeed,
the number of vertices in this connected component remains tight as $n \to \infty$,
the probability of observing one extra edge between two vertices in the
component goes to $0$. Therefore, $C^{(p_n)} \cap[1,\ldots n]$ is a tree with high probability for $n$ large enough, hence $B^{(p_n)}\cap [0,n] = C^{(p_n)} \cap [0,n]$ with high probability, which completes the proof using Theorem~\ref{thm:coupling}.

 We deduce that the length of the longest path in this connected component is the same as the length of the longest path in the tree. This longest path is given by the largest generation $F_t$ at which there is at least one individual to the left of $t$, completing the proof.

We now assume that $n p_n \to \infty$, and we use once again the coupling
described in Theorem~\ref{thm:coupling}. By law of large numbers, the
position $\xi_{\kappa_n}$ of the $n$th vertex in the coupling satisfies $p_n \xi_{\kappa_n}/ n p_n \to
1$ a.s. as $n \to \infty$. Then, using \eqref{eqn:cvInLawBRW}, for all $\epsilon
> 0$, for $n$ large enough we have $M_{\floor{\e np_n(1 + \epsilon)}} \geq
np_n(1+\epsilon/2)$. Therefore, almost surely for $n$ large enough, the longest
path issued from $1$ will be shorter than $\e n(1+\epsilon)$, i.e.\
\[
  \limsup_{n \to \infty} \frac{L_n}{n p_n} \leq \e \quad \text{a.s.}
\]

For the lower bound, we fix $t > 0$, and we consider in a first time the length
of the longest path in the restriction of the Barak-Erd\H{o}s graph to the set
$\{1,\ldots, \floor{t/p_n}\}$. Thanks to the previous case, we know that the
length of this path converges to $F_t$. Moreover, the smallest vertex with
index larger than $t/p_n$ connected to this path is positioned at geometric
distance with parameter $p_n$ from $\floor{t/p_n}$. Therefore, with high
probability, there exists a path of length at least $F_t(1-\epsilon)$ starting
from vertex $1$ and ending at a vertex of index between $(t+1/2)/p_n$,
$(t+2)/p_n$.

Chaining this argument, we obtain that with high probability,
$L^{(n)}_1(p_n)$ is larger than the sum of $n p_n / (t+2)$ independent copies of
$F_t$. Letting $n \to \infty$, then $\epsilon \to 0$, we obtain
\[
  \liminf_{n \to \infty} \frac{L_n}{np_n} \geq \frac{\E(F_t)}{t}.
\]
Then using \eqref{eqn:speedBRW}, we have $\lim_{t \to \infty} \frac{\E(F_t)}{t} = \e$, which completes the proof.
\end{proof}

Theorem \ref{thm:start1} gives the precise asymptotic behavior of
$L_n(p_n)$ as long as $\lim_{n \to \infty} n p_n < \infty$. However, when
$np_n \to \infty$, it only gives the first order of its asymptotic expansion.
Using the coupling described in Theorem~\ref{thm:coupling}, Itoh
\cite{Itoh} showed that the asymptotic behavior of the branching random walk
applies to the sparse Barak-Erd\H{o}s settings, and that as long as $np_n \to
\infty$ ``slowly enough'', we have
\begin{equation}
  \label{eqn:brwLike}
  \lim_{n \to \infty} \sup_{k \in \N} \left| \P( L_n(p_n) =k) - \P(F_{np_n} = k) \right| = 0,
\end{equation}
showing that the law of $L_n$ has the same asymptotic behavior as $F_{np_n}$ in this regime, described in Lemma \ref{lem:asympK}. In particular, this results states that $(L_n(p_n) - np_n \e - \frac{3}{2\e} \log np_n)$ remains tight.

However, this asymptotic behavior fails to hold when $p_n$ decays too slow. The branching random walk associated to the Barak-Erd\H{o}s graph gives a good description of the connected component of $1$ as long as the number of blue balls below position $n p_n$ remains small. This is for example true assuming $np_n$ converges to a constant, under which conditions the number of blue balls before $n p_n$ go to $0$ as $n \to \infty$. On the opposite end of the spectrum, if $p$ remains constant, the
asymptotic behavior of $L_n$ differs sharply from the one of $F_{np}$. Indeed the leftmost particle will tend to be blue with high probability.

The behavior of red particles in the branching random walk can be compared to a \emph{branching random walk with selection}. Indeed, when considering dynamically their behavior, we observe that the $k$th largest red ball produces a red offspring at rate $q^k$. In particular, if $k \ll -\log p$, particles branch at a rate close to $1$, similarly to the usual branching random walk, while if $k \gg -\log p$, the branching rate of that particle is close to $0$. The set of red balls can therefore be closely estimated by the set of surviving particles in the branching random walk with selection of the $N \approx \floor{-\log p}$ rightmost particles.

The branching random walk with selection of the $N$ rightmost particles evolves according to the following procedure. At each generation, all particles reproduce independently, a particle at position $x$ giving birth to a Poisson point process with unit intensity on $[x,\infty)$. Among all children of these particles, only the $N$ leftmost are selected to constitute the new generation of the process.

It is worth observing that this process can be straightforwardly embedded in the
branching random walk. We write $\mathcal{D}_N$ for the set of particles in the
branching random walk that form the branching random walk with selection of the
$N$ rightmost generation, i.e.\ such that for all $n \in \N$, $\mathcal{D}_N \cap\{|u|=n\}$ is the set of the $N$ leftmost children of the particles in $\mathcal{D}_N \cap\{|u|=n-1\}$.

Bérard and Gouéré \cite{BeG10} first proved the Brunet-Derrida behavior for a branching random walk with binary branching and selection of the $N$ rightmost individuals. This result was extended by Mallein \cite{Mal17a} to more general reproduction laws covering the present case. From Kingman's subadditive theorem, we know the existence of a sequence $(v_N)$ such that for each $N \in \N$,
\[
  \lim_{n \to \infty} \frac{1}{n}\max_{u \in \mathcal{D}_N, |u|=n} = v_N \quad \text{a.s.}
\]
The Brunet-Derrida behavior of this branching random walk can be stated as follows
\begin{equation}
  \label{eqn:brunetDerrida}
  v_N = \e - \frac{\pi^2 \e(1 + o(1))}{2 (\log N)^2} \quad \text{as $N\to \infty$.}
\end{equation}
Using the coupling between the branching random walk and the Barak-Erd\H{o}s graph, and the fact that the set of red particles is well-approximated by the set $\mathcal{D}_{\floor{(-\log p)^2}}$, we obtain the following result.
\begin{theorem}[Mallein and Ramassamy \cite{MR19}]
\label{thm:cforSmallp}
As $p \to 0$, we have $$C(p) = p \e  - \frac{\pi^2 \e}{2}p(-\log p)^{-2} + o\left( p (-\log p)^{-2}\right).$$
\end{theorem}

This asymptotic behavior extends to sparse Barak-Erd\H{o}s graphs, as long as
$p_n$ decays to $0$ slowly enough, we have
\begin{equation}
  \label{eqn:bdLike}
  \frac{L_n}{n p_n} = \e - \frac{\pi^2 \e}{2(-\log p_n)^2}(1+o(1)).
\end{equation}
The question of distinguishing sequences $p_n$ for which $L_n$ exhibit a behavior similar to \eqref{eqn:brwLike} or \eqref{eqn:bdLike}, or do distinguish an intermediate behavior remains open. In order to
obtain a proper insight on this question, we begin by giving a proof scheme for
Theorem \ref{thm:cforSmallp}.

\begin{proof}[Scheme of proof of Theorem~\ref{thm:cforSmallp}]
We first remark that typically, the leftmost blue particle born from a given red
particle is produced at a time when the parent is at a distance of order $-\log p$ of the position of the leftmost current red particle in the system. Therefore, for any $\delta > 0$, the set of red particles alive at a large generation $n$ is included in $\mathcal{D}_{(1+\delta)\floor{-\log p}}$, and contains all particles alive at generation $n$ in $\mathcal{D}_{(1-\delta) \floor{-\log p}}$ for all $p$ large enough.

Using \eqref{eqn:brunetDerrida}, we deduce that the leftmost red particle at a large generation $n$ is situated in $[n v_{(1-\delta)\floor{-\log p}}, n v_{(1+\delta)\floor{-\log p}}]$. Then, using Corollary \ref{cor:couplingResult}, we deduce that with high probability,
\[
  \frac{L_n}{p} \in [n v_{(1-\delta)\floor{-\log p}}, n v_{(1+\delta)\floor{-\log p}}]
\]
Letting $n \to \infty$, then $\delta \to 0$ and finally $p \to 0$, Theorem~\ref{thm:cforSmallp} will hold.
\end{proof}

We observe that in the proof of Theorem \ref{thm:cforSmallp}, we used the
coupling with the branching random walk to compare the connected component of
the vertex $1$ with the set of particles in the branching random walk staying at
all time within distance $-\log p$ from the position of the leftmost particle.
Note also that we are interested in the position of the minimal displacement at
time $n p_n$.

Chen \cite{Che15} proved that in a branching random walk, the trajectory
yielding to the minimal position at time $n$ can be scaled in time by a factor
$n$ and in space by a factor $\sqrt{n}$ to converge to a standard Brownian
excursion. In particular, we know that typical trajectory yielding to the
minimal position at time $n$ remains at all time within distance $\sqrt{t}$ from
the leftmost position. As a result, it is expected that as long as $-\log p_n
\gg \sqrt{n p_n}$, the asymptotic behavior of the position of the minimal
displacement at time $n$ is similar in the branching random walk and in the
$-\log p_n$-branching random walk. On the opposite, if $-\log p_n \ll (-\log
p_n)$, similar techniques as the ones used in \cite{Mal17a} show that the
minimal displacement becomes more similar to the branching process with
selection.

As a result, we formulate the following conjecture for the asymptotic behavior
of the length of the longest path starting from $1$ in sparse Barak-Erd\H{o}s
graphs.

\begin{conjecture}
Let $(p_n)$ be a null sequence with $n p_n = \infty$.
\begin{enumerate}
    \item If $\lim_{n \to \infty} n p_n/\log n^2 = 0$, then \eqref{eqn:brwLike} holds.
    \item If there exists $\delta > 0$ such that $\lim_{n \to \infty} np_n/(\log n)^{2+\delta} = \infty$, then \eqref{eqn:bdLike} holds.
\end{enumerate}
\end{conjecture}

\subsection{Extension to stochastic ordered graphs with geometry}
\label{subsec:extensions}

In this section, we consider the two following planar extensions of the
Barak-Erd\H{o}s graph. In both cases, the set of vertices of the graph is
$\N^2$. In the first model, for all pair of distinct vertices $(i,j)$ and
$(k,l)$ a directed edge between $(i,j)$ and $(k,l)$ is present with probability
$p$ independently on any other edges if and only if $i \geq k$ and $j \geq l$
(and $(i,j)\neq (k,l)$). This random graph is written $\mathcal{G}^{(1)}(p)$. In
the second model, a directed edge is present with probability $p$ between
$(i,j)$ and $(k,l)$ if and only if $i=k$ and $j > l$, or $j=l$ and $i > k$. We
denote by $\mathcal{G}^{(2)}(p)$ this random graph. We also write
$\mathcal{G}^{(1)}(n,m,p)$ (respectively $\mathcal{G}^{(2)}(n,m,p)$) the
restriction of $\mathcal{G}^{(1)}(p)$ (resp. $\mathcal{G}^{(2)}(p)$) to the
vertex set $\{1,\ldots n\} \times\{1,\ldots,m\}$.

Similarly to the Barak-Erd\H{o}s graph, the connected component of the vertex
$(1,1)$ can be described, in the sparse limit $p \to 0$ as a branching random
walk. More precisely, in the graph $\mathcal{G}^{(1)}(p)$, rescaling the vertex
indices by $p^{1/2}$, the connected component of $(1,1)$ converges to a
branching random walk $X^{(1)}$ on $\R_+^2$, in which a particle at position
$(x,y)$ gives birth to a Poisson process of offspring on $[x,\infty)\times
[y,\infty)$ with unit intensity. This branching random walk provides an example
of \emph{stable} branching process as defined in \cite{BCM18}.

When scaling the vertex indices by $p$, the connected component of the vertex
$(1,1)$ converges, as $p \to \infty$ to a branching random walk $X^{(2)}$ on
$\R_+^2$ in which a particle at position $(x,y)$ gives birth to a Poisson
process of offspring on $\{x\} \times [y,\infty)$ and an independent Poisson
process of offspring on $[x,\infty)\times \{y\}$ with unit intensity.

Using similar methods as the ones used in the Barak-Erd\H{o}s graph, we can
obtain the following results. First, observing that the law of $X^{(1)}$ is
stable by the transformation $(x,y) \mapsto(x/a,ay)$, and that
\[
  \lim_{n \to \infty} \frac{1}{n} \min \Vert X^{(1)}(u) \Vert_\infty = \frac{2}{\e} \quad \text{a.s.}
\]
and applying the same proof as in \cite{NEWM92}, we obtain the following estimate
for the length of the longest path starting from vertex $1$.
\begin{theorem}
We denote by $L_{n,m}(p)$ the length of the longest path issued from $(1,1)$ in $\mathcal{G}^{(1)}(n,m,p)$.
\begin{enumerate}
  \item If $p \to 0$ with $mnp \to t$, then $L_{n,m}(p)$ converges in law to $K(t) = \max\{ |u| : X^{(1)}(u) \in [0,\sqrt{t}]^2 \}$.
  \item If $p \to 0$ with $mnp \to t$, then $L_{n,m}(p)/\sqrt{mnp} \to \e/2$ in probability.
\end{enumerate}
\end{theorem}

Similarly, we can obtain the following estimate for the asymptotic behavior of
the length of the longest path issued from $1$ in the graph
$\mathcal{G}^{(2)}(n,n,p)$.
\begin{theorem}
We denote by $L_{n}(p)$ the length of the longest path issued from $(1,1)$ in $\mathcal{G}^{(2)}(n,n,p)$.
\begin{enumerate}
  \item If $p \to 0$ with $np \to t$, then $L_{n}(p)$ converges in law to $K(t) = \max\{ |u| : X^{(1)}(u) \in [0,t]^2\}$.
  \item If $p \to 0$ with $np \to t$, then $L_{n}(p)/np \to \e/(1 - \log 2)$ in probability.
\end{enumerate}
\end{theorem}

\subsection{Longest path in the Barak-Erd\H{o}s graph}
\label{subsec:longestPath}

We now turn to the length of the longest path in the Barak-Erd\H{o}s graph in
the sparse regime. Again, Newman \cite[Theorem 2]{NEWM92} obtained a completed
description of the asymptotic behavior of $L_n(p_n)$ as long as $n p_n = O(\log
n)$. Moreover, he proved that as long as $np_n \gg \log n$, then the asymptotic
behavior of $L_n(p_n)$ and $L^{(1)}_n(p_n)$ are similar.

We decompose the asymptotic behavior of the length of the longest path in the
Barak-Erd\H{o}s graph along the following lines. We first assume that $p_n =
o(n^{1 - \epsilon})$ for some $\epsilon > 0$. In this situation, the support of
the length of the longest path in the Barak-Erd\H{o}s graph converges to one or
two points, depending on the exact asymptotic behavior of $p_n$. Remark that
with high probability $1$ will not be connected to any other vertex in the
graph, so $L_n^{(1)}(p_n)$ converges to $0$ in probability.
\begin{theorem}[Newman \cite{NEWM92}, Theorem 2]
\label{thm:longest1}
Let $(p_n)$ be a null sequence such that $\lim_{n \to \infty} n^{1+\epsilon}p_n =0$ for some $\epsilon > 0$.
\begin{enumerate}
  \item If $\lim_{n \to \infty} n^2p_n = 0$, then $\lim_{n \to \infty} L_n(p_n) = 0$ in probability.
  \item If there exists $m \in \N$ such that
  \[\lim_{n \to \infty} p_n n^{1+1/m} = \infty\quad\text{and}\quad \lim_{n\to \infty} p_n n^{1+1/(m+1)} = 0,\]
   then $L_n(p_n) \to m$ in probability.
   \item If there exists $\theta \in (0,\infty)$ and $m \in \N$ such that
   \[
     \lim_{n \to \infty} p_n n^{1+1/m} = \theta,
   \]
   then $\displaystyle \lim_{n \to \infty} \P(L_n(p_n) = m-1) = 1 - \lim_{n \to\infty} \P(L_n(p_n) = m) = e^{-\frac{\theta^m}{(m+1)!}}$.
\end{enumerate}
\end{theorem}

\begin{proof}
This result is obtained using first and second moment methods, computing the
first and second moment of the number $Z_k(n,p)$ of paths of length $k$. Then,
using the Markov inequality from one part and the Cauchy-Schwarz inequality for
the second, we have
\[
  \E(Z_k(n,p)) \geq \P(Z_k(n,p) \geq 1) \geq \E(Z_k(n,p))^2/\E(Z_k(n,p)^2).
\]

We first note that in the complete graph on the vertex set $\{1,\ldots, n\}$,
there is $\binom{n}{k+1}$ increasing paths, each of which having probability
$p^k$ to be open in the Barak-Erd\H{o}s graph. Therefore,
\[
  \E(Z_k(n,p)) = \binom{n}{k+1} p^k \approx \frac{n^{k+1} p^k}{k!}.
\]
With similar computations, we have that
\[
  \E(Z_k(n,p)^2) = \sum_{\substack{1 \leq i_1 \leq \cdots \leq i_{k+1} \leq n\\1 \leq j_1 \leq \cdots \leq j_{k+1} \leq n}} \P((i_1,i_2),\ldots(j_k,j_{k+1}) \in E(G)).
\]
Observing that if two paths have $r$ vertices in common,
they have at most $r-1$ edges in common, we obtain
\begin{align*}
  \E(Z_k(n,p)^2) &\leq \binom{n}{k+1} \binom{n-k-1}{k+1} p^{2k} + \sum_{r = 1}^{k+1} \binom{n}{k+1} \binom{n-k-1}{k-r+1} p^{2k-r-1}\\
  &\approx (n^{k+1} p^k)^2 + \sum_{r =1}^{k+1} n^{2k+2-r} p^{2k-r-1}.
\end{align*}

If $n^2 p_n \to 0$ as $n \to \infty$, we have
\[
  \P(L_n(p_n) \geq 1) = \P(Z_1(n,p_n) \geq 1) \leq \frac{n(n-1)}{2} p_n.
\]
Therefore $\lim_{n \to \infty} \P(L_n(p_n) = 0) = 1$, showing that $L_n(p_n) \to
0$ in probability.

We now assume that there exists $m \in \N$ such that
\[\lim_{n \to \infty} p_n n^{1+1/m} = \infty\quad\text{and}\quad \lim_{n\to \infty} p_n n^{1+1/(m+1)} = 0.\]
We observe that
\[
  \P(L_n(p_n) \geq m+1)  = \P(Z_{m+1}(n,p_n) \geq 1) \leq \E(Z_{m+1}(n,p_n)) \leq \binom{n}{m+1} p_n^{m}.
\]
Then, as $\lim_{n \to \infty} n^{m+1} p_n^m = 0$, we deduce that $\lim_{n \to \infty} \P(L_n(p_n) \geq m+1) = 0$. Similarly, we have
\begin{align*}
  \P(L_n(p_n) \geq m) &= \P(Z_{m}(n,p_n) \geq 1) \geq  \E(Z_m(n,p_n))^2/\E(Z_m(n,p_n)^2)\\
  &\geq \frac{1}{1 + \sum_{r=1}^{m+1} \frac{\binom{n-k-1}{k-r+1}}{\binom{n-k-1}{k+1}} p_n^{-r-2}}.
\end{align*}
But as $\lim_{n \to \infty} n^{-r} p_n^{-r-2} = 0$ for all $r \leq m+1$ by assumption, we obtain that $\lim_{n \to \infty} \P(L_n(p_n) \geq m) = 1$, completing the proof that $L_n(p_n) \to m$ in probability in that case.

Next, assuming that $p_n^m n^{m+1} \to \theta^m$, we show that $Z_m(n,p_n)$
converges to a Poisson random variable with parameter $\theta^m/(m+1)!$. Indeed,
$Z_m(n,p_n)$ is the sum of $\binom{n}{m+1}$ Bernoulli random variables with
parameter $\theta^m/n^{m+1}$. We then observe that the sum of the covariances of
these Bernoulli random variables converges to $0$ as $n \to \infty$, proving
this convergence. As a result, we obtain
\[
  \lim_{n \to \infty}\P(L_n(p_n) \geq m)  = \lim_{n \to \infty} \P(Z_m(n,p_n) \geq 1) = 1 - e^{-\theta^m/m!},
\]
and as $\P(L_n(p_n) \geq m-1) \to 1$ using the same computations as above, we have obtained the convergence in law of $L_n(p_n)$.
\end{proof}

We then consider the asymptotic behavior of $L_n(p_n)$ assuming that $n p_n =
o(\log n)$. In this regime, recall that the length of the longest path issued
from vertex $1$ is of order $\e n p_n = o(\log n)$. The following result holds.

\begin{theorem}
\label{thm:longest2}
Let $(p_n)$ be a null sequence such that $\lim_{n \to \infty} n^{1+\epsilon}p_n
= \infty$ for all $\epsilon > 0$ with $\lim_{n \to \infty} np_n/\log n = 0$. We
set
\[
     \ell_n = \sup\left\{ k \in \N : \binom{n}{k}p_n^k \leq 1 \right\},
   \]
then $\lim_{n \to \infty} L_n(p_n)/\ell_n = 1$ in probability.
\end{theorem}

This result is obtained in the same way as Theorem~\ref{thm:longest1}, using
first and second moment methods to bound the probability of existence of a path
of length $\ell_n (1 \pm \epsilon)$. Using similar methods again, one can turn
to the case $p_n \sim \gamma \log n/n$. In this situation, $L^{(1)}_n(p_n) \sim
\gamma \e \log n$, and the longest path also is of order $\log n$, but with a
larger constant.

\begin{theorem}
\label{thm:longest3}
Let $(p_n)$ be a sequence such that $p_n \sim \gamma \log n/n$ as $n \to \infty$
for some $\gamma \in (0,1)$. Then
\[
  \lim_{n \to \infty} L_n(p_n)/n p_n = \e A(\gamma) \quad \text{in probability},
\]
where $A(\gamma)$ is the only solution larger than $1$ of
$x \log x = (\e\gamma)^{-1}$.
\end{theorem}

Finally, when $n p_n \gg \log n$, the length of the longest path and the longest
path starting from $1$ have similar orders of magnitude.
\begin{theorem}
\label{thm:longest4}
Let $(p_n)$ be a null sequence such that $\lim_{n \to \infty} n p_n/\log n = \infty$, then $$\lim_{n \to \infty} L_n(p_n)/np_n = \e.$$
\end{theorem}

To prove that  $\lim_{n \to \infty} \P(L_n(p_n) > n p_n \e) = 0$, we use again first moment estimates. Using that $L_n(p_n)
\geq L_n^{(1)}(p_n)$ on the other hand and Theorem \ref{thm:start1}, we conclude
to Theorem \ref{thm:longest4}.

\subsection{Shortest path}
\label{subsec:shortestPath}

Based on similar computations with the previous section, we are able to compute
the asymptotic behavior of the length of the \emph{shortest} path between $1$
and $n$ in the Barak-Erd\H{o}s graph $\overrightarrow{\mathcal{G}}(n,p)$.
The results in this section are based on \cite{MT22}.
In
this section, we write
\[
  S_n(p) = \min\left\{ k \in \N : \exists i_1 < \cdots < i_{k-1} : (1,i_1),\ldots, (i_{k-1},n) \in E(\overrightarrow{\mathcal{G}}(n,p)) \right\}.
\]
the length of the shortest path linking $1$ and $n$. By convention, $S_n =
\infty$ if there is no path between $1$ and $n$.

Observe that for a fixed value of $p$, the edge $(1,n)$ is present with
probability $p$, and if this edge is not present, the events
\[(\{(1,k),(k,n)
\in E(\overrightarrow{\mathcal{G}}(n,p))\}, 2 \leq k \leq n)
\]
are independent of one another, and with same probability $p^2$ of occurrence.
As a result, we have
\[
  \P(S_n(p)= 1) = p, \quad \P(S_n(p)= 2) = (1 - p) (1 - (1 - p^2)^{n-2}).
\]
Therefore, $S_n(p)$ converges in distribution, as $n \to \infty$, to a random
variable $S_\infty(p)$ with $\P(S_\infty(p) = 1) = 1 - \P(S_\infty(p) = 2) = p$.
With similar computations, it appears that if $p_n \to 0$ with $n p_n^2 \to
\infty$, then $S_n(p_n) \to 2$ in probability as $n \to \infty$.

The above computation extends to larger paths, and we obtain the following
extension of the results of Section~\ref{subsec:longestPath}.
\begin{theorem}
\label{thm:sparseGraphShortest}
Let $(p_n)$ be a null sequence.
\begin{enumerate}
  \item If there exists an integer $m \in \N$ such that
  \[\lim_{n \to \infty} p_n n^{1 - 1/(m+1)} = \infty \quad \text{with} \quad \lim_{n \to \infty} p_n n^{1 - 1/m} = 0,\]
  then $S_n(p_n) \to m+1$ in probability.
  \item If there exists an integer $m \in \N$ such that
  \[\lim_{n \to \infty} p_n n^{1 - 1/m} = \theta \in \R_+,\]
  then $\displaystyle \lim_{n \to \infty} \P(S_n(p_n) = m) = 1 - \lim_{n \to \infty} \P(S_n(p_n) = m+1) = e^{-\theta^m/(m-1)!}$
  \item If $\lim_{n \to \infty} n^{1-\epsilon} p_n = \infty$ for all $\epsilon>0$, then $\lim_{n\to\infty} S_n(p_n)/\ell_n = 1$ in probability, where $\ell_n$ is the same quantity as in Theorem~\ref{thm:longest2}.
\end{enumerate}
\end{theorem}

For a graph with too small value of $p$, the vertices $1$ and $n$ will no longer
be connected with high probability. It is well-known for Erd\H{o}s-Rényi graphs
that as long as $p_n > (1+\epsilon)/n$, vertices $1$ and $n$ will be in the same
connected component with positive probability. However, this does not implies
the existence of a directed path from $1$ to $n$ in the graph, and in fact the
coupling with the branching random walk shows that as long as $\limsup_{n \to
\infty} np_n < \infty$, the probability that $1$ and $n$ are in the same
connected component goes to $0$ as $n \to \infty$.

The critical decay rate at which the probability that $1$ and $n$ are connected
in $\overrightarrow{\mathcal{G}}(n,p_n)$ with positive probability is currently
unknown, as well as the asymptotic behavior of $S_n(p_n)$ for $p_n \sim (\log
n)^\alpha/n$ for some $\alpha > 0$.

\section{Regenerative properties of directed random graphs}
\label{secreg}
We now return to the setup of Section \ref{ergosec}
and consider the case where we are given a sequence $p_1, p_2, \ldots$
of numbers in $[0,1]$ such that
\[
0<p_1<1.
\]
The case $p_1=1$ is excluded because it uninteresting.
We assume that $p_1>0$ so that the regenerative methods work:
as will be shown in Section \ref{break}, we can then break the graph into i.i.d.\
pieces.
For the $p_1=0$ case see Remark \ref{jdepend} below.

We then consider the random graph on $\Z$ such
that the pair $(i,j) \in \Z \times \Z$ is a vertex if $i<j$ and with probability
$p_{j-i}$, independently from pair to pair.
Let us denote this graph by $\vec G(\Z, (p_j))$.
The reason for introducing probabilities that depend on the physical
distance between vertices is twofold: from the point of view of applications,
it is more natural; mathematically, the case of constant $p_j$, insofar as
regenerative properties are concerned, is not much different from the more
general case.
Let
\[
q_j := 1-p_j, \quad Q_j := q_1 \cdots q_j, \quad Q_0:=1.
\]

Our goal is to prove a functional central limit theorem for the
quantity $L_{0,n}$, the maximum length of all pathr in $\vec G(\Z, (p_j))$
with endpoints between $0$ and $n$.
This is Theorem \ref{momthm}.
To do this, we first exhibit some regenerative properties of the graph
and then estimate moments of the distances between two typical skeleton points.

\subsection{Breaking the graph into independent cycles}
\label{break}
We showed in Lemma \ref{lela},
that if the sequence $p_j$ does not convergence to $0$ too fast,
in the  sense that condition \eqref{summa} of Lemma \ref{lela},
\[
\sum_{n=1}^\infty Q_n < \infty,
\]
holds, then the skeleton set $\SSS$
of all vertices $v$ such that $u \leadsto v \leadsto w$
for all $u < v < w$ (where $i \leadsto j$ stands for ``there is
a path from $i$ to $j$'')
is infinite in both directions a.s.\
and has strictly positive rate
\begin{equation} \label{lambdatoo}
\lambda  = \prod_{j=1}^\infty (1-Q_j)^2 > 0,
\end{equation}
as in \eqref{lprod}.
We label the elements of $\SSS$ by random integers $\Gamma_k$, $k \in \Z$,
so that
\begin{equation}
\label{gammapoints}
\cdots < \Gamma_{-1} < \Gamma_0 \le 0 < \Gamma_1 < \Gamma_2 < \cdots
\end{equation}
We therefore know that $\P(0 \in \SSS)=\lambda>0$ and
\[
\E(\Gamma_{k+1}-\Gamma_k|0\in \SSS) =
\E(\Gamma_{k+1}-\Gamma_k|\Gamma_0=0) = 1/\lambda < \infty,
\quad k \in \Z.
\]
Consider now the random sequence
\[
\alpha^{(0)} := (\alpha^{(0)}_1, \alpha^{(0)}_2, \ldots),
\]
consisting of independent entries with
\[
\P(\alpha^{(0)}_j=1)=1-\P(\alpha^{(0)}_j=-\infty) =p_j, \quad j \in \N,
\]
and let $\alpha^{(n)}$, $n \in \Z$, be i.i.d.\ copies of $\alpha^{(0)}$.
Then we can construct $\vec G(\Z, (p_j))$ as the graph on $\Z$
with edge set
\[
\{(i,j) \in \Z\times \Z:\, i < j,\, \alpha^{(i)}_{j-i}=1\}.
\]
(Note that the random variable $\alpha^{(i)}_{j-i}$ was denoted by
$\alpha_{i,j}$ in \eqref{GZxi}.)
We refer to the random object
\[
\CC_k :=(\alpha^{(n)},\, \Gamma_k \le n < \Gamma_{k+1})
\]
as the $k$th ``cycle'', $k \in \Z$.
Let $\vec G_k$ be the induced subgraph of $\vec G(\Z,(p_j))$
on the set of vertices $\{v \in \Z:\, \Gamma_k \le v \le \Gamma_{k+1}\}$,
that is, keep only those edges with endpoints in this set.
The following was proved in \cite{DFK12}.
\begin{lemma}
\label{indg}
Assume that \eqref{summa} holds. Then,
conditional on $\{0 \in \SSS\}$, the cycles $\CC_k$, $k \in \Z$, are i.i.d.\
and the random graphs $\vec G_k$, $k \in \Z$, are i.i.d.
\end{lemma}
\begin{proof}[Sketch of proof]
It suffices to show that $\CC_0$ is independent of
$(\CC_{-1}, \CC_{-2}, \ldots)$ conditional on $0 \in \SSS$.
Let $\FF^+:= \sigma(\alpha^{(n)}, n > 0)$,
$\FF^-:=\sigma(\alpha^{(n)}, n < 0)$.
As in the proof of Lemma \ref{lela},
for each $j \in \Z$, consider the largest of the vertices $u<j$ such that
$(u,j)$ is an edge, letting $\ell(j) = j-u$ be its distance from $j$.
Similarly, consider the smallest of the vertices $v>j$ such that $(j,v)$
is an edge, setting $r(j)=v-j$.
Note that
\begin{equation}
\label{0isskeleton}
\{0\in \SSS\}= \{\Gamma_0=0\}
=
\{r(-1) \le 1, r(-2) \le 2, \ldots; \ell(1) \le 1, \ell(2) \le 2,\ldots\}
\end{equation}
Define the random variables
\begin{align*}
\widehat \Gamma_{-1} &:= \max\{n<0: r(n-1) \le 1, r(n-2) \le 2,\ldots;
\ell(n+1) \le 1, \ldots, \ell(0)\le |n|\}
\\
\widehat\Gamma_1&:=\min\{n > 0:~ r(0)\le n, \ldots, r(n-1)\le 1;
\ell(n+1) \le 1, \ell(n+2) \le 2, \ldots\}
\end{align*}
and observe that
\[
\text{ if } \Gamma_0=0 \text{ then }
\Gamma_{-1} = \widehat \Gamma_{-1} , \, \Gamma_1=\widehat\Gamma_1.
\]
Whereas $\Gamma_{-1}$ is not $\FF^-$--measurable,
the random variable $\widehat \Gamma_{-1}$ is.
Also, $\widehat \Gamma_{1}$ is $\FF^+$ measurable.
Similarly, each cycle $\CC_k$ with negative index $k$
becomes equal to an $\FF^-$--measurable random object.
This argument shows that $\CC_0$ is independent of
$(\CC_{-1}, \CC_{-2}, \ldots)$, conditional on $\{\Gamma_0=0\}$.
That all the $\CC_k$, $k \in \Z$,
are identically distributed, under the same conditioning, follows from stationarity.
For the last claim
observe that $\vec G_k$ is a function of $\CC_k$ for all $k \in \Z$.
\end{proof}

\begin{remark}
\label{palmrem}
Let $\alpha$ denote the map $n \mapsto \alpha^{(n)}$.
For all $k \in \Z$, denote by $\theta^k \alpha$
the map $n \mapsto \alpha^{(n+k)}$.
Note that the probability measure $\P$ relates to $\P(\cdot|0 \in \SSS)$
as follows. Let $\Phi(\alpha)$ be a measurable bounded function of $\alpha$.
Denote by $\theta^k \Phi$ the function $\alpha \mapsto \Phi(\theta^k\alpha)$.
Then
\[
\E(\Phi) = \lambda
\E\bigg\{\sum_{\Gamma_0 \le k < \Gamma_1} \theta^k \Phi
\bigg| 0 \in \SSS\bigg\}.
\]
For a simple proof of this, see \cite{KZ95}.
This implies that (without conditioning on $0 \in \SSS$), the cycles
$\CC_k$, $k \in \Z \setminus \{0\}$, are i.i.d.\ each
with the same distribution as the conditional distribution of $\CC_0$
given $0 \in \SSS$. Moreover, $\CC_0$ is independent of the rest of the
cycles, and its law can be found by the last formula.
In particular, $\E(\Gamma_{k+1}-\Gamma_k) =
\E(\Gamma_{k+1}-\Gamma_k|\Gamma_0=0)= 1/\lambda$ for all $k\neq 0$.
For $k=0$, letting $\Phi=\Gamma_1-\Gamma_0$ in the above formula,
we have $\E(\Gamma_1-\Gamma_0)=\lambda \E(\Gamma_1^2|\Gamma_0=0)$
may be equal to $\infty$
unless a condition stronger that \eqref{summa} is assumed.
\end{remark}

\subsection{Moments of auxiliary  stopping times}
We aim at studying the moments of the auxiliary random vertices
$\mu$ and $\nu$, defined in \eqref{mupoint} and \eqref{nupoint},
respectively. These random vertices are positive integers
and stopping times with respect to the filtration, in the index $n$, generated
by the ($\alpha^{(k)}$, $k \le n$).
The stopping time property is crucial in constructing,
in Section \ref{moske}, certain iterates of $\mu$ and $\nu$,
used for identifying the least positive skeleton point of the graph.

We need some notation, partially introduced in the proof of Lemma \ref{lela}.
Recalling that $i \leadsto j$ means that there is a path
in $\vec G(\Z,\alpha)$ from $i$ to $j$, we let $J \subset \Z$
and write $i \leadsto J$ if $i \leadsto j$ for all $j \in J$.
Similarly, $J \leadsto i$ means $j \leadsto i$ for all $j \in J$.
We will need the events
\begin{align}
\label{A+-}
\begin{split}
A^+_{u,v} &= \{u \leadsto [u+1, v]\}, \\
A^-_{u,v} &= \{[u,v-1] \leadsto v\},
\end{split}
\end{align}
where $u <v$ are integers.
By $[a,b]$ when $a, b$ are integers, $a<b$, we mean
the set $\{a,a+1, \ldots, b\}$.
We have
$\ell(j) = \max\{k>0:\, \alpha_{j-k,j}=1\}$
and $r(j) = \min\{k>0:\, \alpha_{j,j+k}=1\}$ for all $j \in \Z$.
Hence $j-\ell(j)$ is the first predecessor of $j$ in $\vec G(\Z,(p_j))$
and $j+r(j)$ its first successor.
The $\ell(j)$, $j \in \Z$ are i.i.d.\ with common distribution
determined by
\[
\P(\ell(0)>k) = Q_k = \P(r(0)>k).
\]
In particular, $\E \ell(0)=\E r(0) = \sum_{k=0}^\infty Q_k$.
A few moments of reflection show that, for $d \in \N$,
\begin{align*}
A^+_{u,u+d} &= \{\ell(u+1)\le 1, \ldots, \ell(u+d)\le d\}
\\
A^-_{u-d,u} & = \{r(u-1)\le 1, \ldots, r(u-d)\le d\}.
\end{align*}
Both events decrease as $d$ increases.
Define the random variable
\begin{equation} \label{mupoint}
\mu:=\inf\{k \in \N:\, A_{0,k}^+ \text{ fails}\}
\end{equation}
(where ``$A$ fails'' stands for ``$A^c$ occurs'', that is, $\1_{A^c}=1$)
Since $A^+_{0,k}$ decreases as $k$ increases, we have
\[
\P(\mu > k) = \P(A_{0,k}^+) = \P(\ell(1) \le 1, \ldots, \ell(k) \le k)
= (1-Q_1) \cdots (1-Q_k),
\]
and this implies that $\mu$ is defective:
\[
\P(\mu=\infty) = \lim_{k\to \infty} (1-Q_1) \cdots (1-Q_k) = \sqrt{\lambda}>0.
\]
We also have
\begin{align}
\P(n<\mu < \infty)
&= \P\left(A_{0,n}^+ \cap \bigcup_{k=1}^\infty (A_{0,k}^+)^c\right)
= \P\left(A_{0,n}^+ \cap \bigcup_{k=n+1}^\infty (A_{0,k}^+)^c\right)
= \P\left(A_{0,n}^+ \cap
\bigg(\bigcap_{k=n+1}^\infty A_{0,k}^+\bigg)^c\right)
\nonumber
\\
&=\P(\ell(1) \le 1,\ldots, \ell(n)\le n;
\{\ell(1) \le 1, \ldots, \ell(n) \le n, \ell(n+1) \le n+1, \ldots\}^c)
\nonumber
\\
&=\P(\ell(1) \le 1,\ldots, \ell(n)\le n;
\{\ell(n+1) \le n+1, \ell(n+2)\le n+2, \ldots\}^c)
\nonumber
\\
&=\P(\ell(1) \le 1,\ldots, \ell(n)\le n)\,
\P(\exists k > n~ \ell(k)>k)
\label{theone}
\\
&= \P(\ell(1) \le 1,\ldots, \ell(n)\le n)\,
\left[1-\P(\ell(n+1) \le n+1, \ell(n+2) \le n+2, \ldots)\right]
\nonumber
\\
&=
(1-Q_1) \cdots (1-Q_n)\,
\left[1-(1-Q_{n+1})(1-Q_{n+2}) \cdots\right]
\label{thetwo}
\\
&= (1-Q_1) \cdots (1-Q_n)-\sqrt{\lambda}.
\nonumber
\end{align}
We obtain an upper bound of this easily from \eqref{theone}:
\[
\P(n<\mu < \infty)
\le \P(\exists k > n~ \ell(k) > k)
\le \sum_{k=n+1}^\infty \P(\ell(k)>k)
= \sum_{k=n+1}^\infty \P(\ell(0)>k)
= \sum_{k=n+1}^\infty Q_k,
\]
and this upper bound converges to $0$ as $n \to \infty$ by
the assumption that the sequence $(Q_n)$ is summable.
We also have an asymptotic lower bound by using \eqref{thetwo},
the inequality $(1-Q_1)\cdots(1-Q_n) \ge \sqrt{\lambda}$, for all $n$, and the
fact that $(1-e^{-x})/x$ decreases as $x$ increases:
\begin{gather*}
\P(n<\mu < \infty) \ge
\sqrt{\lambda} \left[1-e^{-(Q_{n+1}+Q_{n+2}+\cdots)}\right]
\ge\sqrt{\lambda}\,\frac{1-e^{-\E \ell(0)}}{\E \ell(0)}\, \sum_{k=n+1}^\infty Q_k.
\end{gather*}
We have thus proved
\begin{lemma}
\[
\frac{\sqrt{\lambda}}{1-\sqrt{\lambda}}\,
\frac{1-e^{-\E \ell(0)}}{\E \ell(0)}\,
\sum_{k=n+1}^\infty Q_k
\le \P(\mu>n|\mu<\infty)
\le \frac{1}{1-\sqrt{\lambda}}\,
\sum_{k=n+1}^\infty Q_k.
\]
\end{lemma}
Since $Q_k=\P(\ell(0)>k)$, the above inequality says that
the tail of the distribution of $\mu$ conditional on it being finite
is comparable to the ``integrated'' tail of the distribution of
$\ell(0)$.
We say that a positive random variable $Z$ has $p$-moment, for $p\ge1$,
if $\E Z^p < \infty$; we say that it has an exponential moment
of $\E e^{\theta Z}<\infty$ for some $\theta>0$.
From the above lemma we conclude the following.

\begin{corollary}	\label{mumom}
Conditional on $\{\mu < \infty\}$, $\mu$ has a $p$-moment
(respectively,  exponential moment) if and only if $\ell(0)$
has a $(p+1)$-moment (respectively, exponential moment).
\end{corollary}

We similarly wish to examine the moments of the random variable
\begin{equation} \label{nupoint}
\nu := \inf\{k \in \N:\, A_{0,k}^- \text{ occurs}\}.
\end{equation}
Since $A_{0,k}^-$ is not monotonic in $k$, we need to use an
argument different than before.
First observe that
\[
\nu = \inf\{k \in \N\, [0,k-1] \leadsto k\}
= \inf\{k \in \N:\, r(k-1) \le 1, \ldots, r(0) \le k\}.
\]
Define a sequence of nonnegative integer-valued random variables
$x_n$, $n=0,1,\ldots$, by $x_0=0$ and
\[
x_n= \max\{r(0)-n, r(1)-(n-1), \dots, r(n-1)-1\}.
\]
Then
\[
\nu=\inf\{n \in \N:\, x_n=0\}.
\]
But the $x_n$ satisfy
\[
x_{n+1} = \max(x_n,r(n))-1,\quad n \ge 0,
\]
and since the $r(n)$ are i.i.d., the sequence $(x_n)$ is Markovian.
For any integer $K>0$,
if $x_n\ge K$, then
\begin{align*}
x_{n+1}-x_n = (r(n)-x_n)^+-1\le (r(n)-K)^+-1.
\end{align*}
Since $r(0)$ has finite expectation,
we have that $\E(r(0) -K)^+ <1$, for $K$ large enough.
Therefore, after the Markov chain
leaves the interval $[0,K]$, it is upper-bounded by a random walk
with increments distributed like $(r(0) -K)^+-1$ whose mean is negative.
By standard properties of random walks this implies that the
finiteness of the $p$-th moment of the return time $T_K$
to the set $[0,K]$ is eqiuvalent to the finiteness of the $p$-th moment
of the positive part of the increments and, in turn, to
the finiteness of the $p$-th moment of $r(0)$ which is
the same as the $p$-th moment of $\ell(0)$. Similar conclusions
are made for exponential moments.
We have thus proved:

\begin{lemma}	\label{numom}
$\nu$ has a $p$-moment (respectively, exponential moment) if and only if
$\ell(0)$ has a $p$-moment (respectively, expomential moment).
\end{lemma}

\begin{remark}	\label{qmom}
Since $\P(\ell(0)>k) = Q_k$, $k \ge 0$, we see that $\ell(0)$ has
$p$-th moment if and only iff $\sum_{k=1}^\infty k^{p-1} Q_k < \infty$.
and $\ell(0)$ has an exponential moment if and only if
$\sum_{k=1}^\infty z^k Q_k < \infty$ for some $z>1$.
In particular, these conditions hold  if $p_j=p \in (0,1)$ for all $j$,
\end{remark}

\subsection{Moments of skeleton points and a central limit theorem}
\label{moske}
We show that not only skeleton points exist but that they can also
be constructed recursively and causally.
This will be done by means of iterates of the
the stopping times $\mu$ and $\nu$.
We define two interlaced sequences of stopping times
\[
\nu[1] < \mu[1] < \nu[2] < \mu[2] < \nu[3] < \mu[3] < \cdots,
\]
as follows.
\begin{align}
\nu[1] &:=\nu
\nonumber\\
\mu[1] &:= \nu +  \theta^{\nu}\mu
= \inf\{j > \nu:~ A^+_{\nu,j}  \text{ fails} \},
\label{r1}
\end{align}
and, recursively, for $k \ge 2$,
\begin{align}
\nu[k] &:=
\inf\{j > \mu[k-1]: A_{\nu[k-1],j }^-  \text{ occurs} \}
\nonumber\\
\mu[k] &:= \nu[k] + \theta^{\nu[k]}\mu
= \inf\{j > \nu[k]: A_{\nu[k],j}^+  \text{ fails} \}.
\label{r2}
\end{align}
It is understood that if for some $k$
we have $\mu[k]=\infty$ then $\nu[j]=\mu[j]=\infty$ for all $j \ge k+1$.
In fact, since $\mu$ is a defective random variable, with $\P(\mu=\infty)
=\sqrt{\lambda}$, it follows that the recursion terminates in finitely
many steps.
Let
\begin{equation}
\label{K}
K:= \inf\{k\ge 1:~ \mu[k]=\infty\}.
\end{equation}
It is easy to see that $K$ is a geometric$(\sqrt{\lambda})$
random variable, that is,
$\P(K>k) = (1-\sqrt{\lambda})^k$, $k \ge 0$.
By construction, we have $\nu[K]<\infty$ a.s.

\begin{remark}
It may be instructive to provide intuition regarding the construction above.
We wish to find a vertex $K$ that has the property $K \leadsto j$ for all $j>K$
{\em and} $i \leadsto K$ for all $0 \le i < K$.
(This is not quite saying that $K$ is a skeleton point because we don't
care whether negative vertices lead to $K$ via a path.)
Keep in mind that $\nu$ is a.s.\ finite but $\mu$ has positive probability
that it be infinity.
We first have that , $A^-_{0,\nu}$ holds. that is, $i \leadsto \nu$ for all
$0 \le i < \nu$.
Now define $\mu[1]$ such that $\mu[1]-\nu[1]$ has law of $\mu$.
Suppose it so happens that $\mu[1]=\infty$.
This means, by definition, that $A^+_{\nu, j}$ holds for all $j > \nu$,
that is, $\nu \leadsto j$ for all $j > \nu$.
And so, the vertex $K=\nu$ has the required properties.
If $\mu[1]$ is not $\infty$, then we repeat the procedure
as many times as is required until we obtain one that is infinity.
\end{remark}

Using the observation that $\nu[k]$ and $\mu[k]$ are stopping times
with respect to the filtration $(\FF^-_n)_{n \ge 0}$,
where $\FF^-_n = \sigma(\alpha^{(k)}, k \le n)$,
we easily obtain that
\[
\nu[K] \eqdist \sum_{i=1}^\kappa \nu_i + \sum_{i=1}^{\kappa-1} \mu_i,
\]
where $\kappa, \nu_1, \nu_2, \ldots, \mu_1, \mu_2,\ldots$
are independent, with $\kappa \eqdist K$, hence geometric$(\sqrt{\lambda})$,
$\nu_i \eqdist \nu$, and $\P(\mu_i \in \cdot) = \P(\mu \in \cdot |\mu<\infty)$.
\begin{lemma}
\label{l67}
$\nu[K]$ has a $p$-moment (respectively, exponential moment) if and only if
$\ell(0)$ has a $(p+1)$-moment (respectively, exponential moment).
\end{lemma}
\begin{proof}
From Lemma \ref{numom}, we have that $\sum_{i=1}^\kappa \nu_i$
has an exponential (respectively, $p$-th) moment if and only if
$\ell(0)$ has an exponential (respectively, $p$-th) moment.
From Corollary \ref{numom}, we have that $\sum_{i=1}^{\kappa-1} \mu_i$
has an exponential (respectively, $p$-th) moment if and only if
$\ell(0)$ has an exponential (respectively, $(p+1)$-th) moment.
\end{proof}
\begin{theorem} \label{momthm}
Let $p\ge 1$.
\\
(i) $\E((\Gamma_1-\Gamma_0)^p|\Gamma_0=0) <\infty$
if and only if $\sum_{k=1}^\infty k^p Q_k < \infty$
\\
(ii) The distribution of $\Gamma_1-\Gamma_0$ conditional on $\Gamma_0=0$
has an exponential moment if and only if $\sum_{k=1}^\infty z^k Q_k< \infty$
for some $z > 1$.
\end{theorem}
\begin{proof}
By the definition of $\nu$ and $\mu$ we easily see that
\[
\big[0,\nu[K]-1\big] \leadsto \nu[K] \leadsto
\big[\nu[K]+1,\infty \big) \text{ a.s.}
\]
In fact, $\nu[K]$ is the least $j\ge 1$ such that
$[0,j-1] \leadsto j \leadsto [j+1,\infty)$.
Note that
\[
(-\infty, -1] \leadsto 0 \text{ and } \big[0,\nu[K]-1\big] \leadsto \nu[K]
\text{ implies } (-\infty, \nu[K]-1] \leadsto \nu[K].
\]
Hence, conditional on $0$ being a skeleton point ($\Gamma_0=0$),
we have that $\nu[K]$ is the next skeleton point:
\[
\P(\nu[K]=\Gamma_1-\Gamma_0|\Gamma_0=0) =1.
\]
We now use Lemma \ref{l67} and Remark \ref{qmom} to conclude.
\end{proof}

The results above are used in the proof of the following
functional central limit theorem.

\begin{theorem}
\label{barbarossa}
Consider the random graph $\vec G(\Z, (p_j))$, assuming $0<p_1<1$,
and
\[
\sum_{k=1}^\infty k^2 (1-p_1) \cdots (1-p_k)<\infty.
\]
For integers $i<j$, let $L_{i,j}$ be the maximum length of all paths with endpoints
between $i$ and $j$. Let $\Gamma_k$, $k \in \Z$, be the skeleton points,
assuming $\Gamma_0\le 0 < \Gamma_1$, and $C$ the constant appearing
in the law of large numbers for $L_{0,n}$ as in Lemma \ref{lemmix}. Then
\begin{equation}
\label{varla}
0< \sigma^2 := \var(L_{\Gamma_1,\Gamma_2} - C (\Gamma_2-\Gamma_1)) < \infty
\end{equation}
Define
\[
\ell_n(t) := \frac{L_{0,[nt]}-C n t}{\sigma \sqrt{\lambda n}}, \quad
t \ge 0, \quad n \in \N.
\]
Considering the sequence $\ell_n$, $n \in \N$, as a sequence of
random elements of the Skorokhod space $D[0,\infty)$ (see \cite{BILL68}),
equipped with the topology of local uniform convergence,
converges in distribution to a standard Brownian motion.
\end{theorem}
\begin{proof}[Sketch of proof]
We have $\lambda>0$ because $p_1>0$.
It is clear that $\sigma^2>0$. The reason that $\sigma^2$ is finite
follows $L_{\Gamma_1,\Gamma_2} \le \Gamma_2-\Gamma_1$ and Theorem \ref{momthm}
which guarantees that $\E[(\Gamma_1-\Gamma_2)^2]
= \E[(\Gamma_1-\Gamma_0)^2 |\Gamma_0=0] <\infty$.
From Remark \ref{palmrem} we also have that $\E\Gamma_1<\infty$.
Let $N_n = \max\{j \ge 1:\, \Gamma_j \le n\}$.
By the definition of $\SSS$, if $v \in \SSS$ then the maximal path from
some $u<v$ to some $w>v$ must necessarily include $v$.
Hence
\[
L_{0,n} = L_{0,\Gamma_1}+L_{\Gamma_1,\Gamma_2} + \cdots + L_{\Gamma_{N_n,n}}.
\]
We can then write
\begin{equation}	\label{ldec}
\ell_n(t) = \frac{L_{0,\Gamma_1}-C\Gamma_1}{\sigma\sqrt{\lambda n}}
+\frac{1}{\sigma\sqrt{\lambda n}}
\sum_{i=2}^{N_{nt}} \left[L_{\Gamma_{i-1},\Gamma_i}-C(\Gamma_i-\Gamma_{i-1})\right]
+\frac{L_{\Gamma_{nt},nt}-C(nt-\Gamma_{N_{nt}})}{\sigma\sqrt{\lambda n}}
\end{equation}
Define also
\[
\widehat \ell_n(t) := \frac{1}{\sigma\sqrt{\lambda n}}
\sum_{i=2}^{nt} \left[L_{\Gamma_{i-1},\Gamma_i}-C(\Gamma_i-\Gamma_{i-1})\right]
,\quad
\phi_n(t) := \frac{N_{nt}}{n}.
\]
The first and third terms of \eqref{ldec} converge to $0$ in distribution,
as random elements of $D[0,\infty)$.
The middle term is simply equal to
$\widehat \ell_n \comp \phi_n$.
Hence the limit of $\ell_n$ exists if and only if the limit of
$\widehat \ell_n \comp \phi_n$ exists, in which case the limits are equal.
But $\widehat \ell_n$ converges in distribution to $\frac{1}{\sqrt{\lambda}} B$,
where $B$ is a standard Brownian motion, and $\phi_n$ converges in distribution
to the function $\phi$, where $\phi(t)=\lambda t$.
Therefore, by the continuity in $D[0,\infty)$, of the composition
operator, we have $\widehat \ell_n \comp \phi_n$ converges in distribution
to the process $\frac{1}{\sqrt{\lambda}} B(\lambda t)$, $t \ge 0$,
which is equal in distribution to $B$.
Hence $\ell_n$ converges in distribution to $B$.
\end{proof}

\begin{remark}
\label{jdepend}
We made the assumption that $p_1>0$ in order
that the set $\mathscr S$ serve as a skeleton set over which
the graph regenerates, that is, we have independent pieces.
If $p_1=0$ then $\mathscr S$ is useless; for example, $\lambda=0$ in this case.
However, if $p_1=\cdots=p_j=0$ but $p_{j+1} >0$
then we can establish a 1-dependent structure.\footnote{We say that
a random sequence $X_i$, $i \in \Z$, is 1-dependent if, for each $i \in \Z$,
the sequences $(X_j, j < i)$ and $(X_j, j > i)$ are independent.}
One can still deal with this case, but one needs a somewhat different technique
based on weakly regenerative processes.
\end{remark}

\begin{remark}
At this stage, we have no estimate for the variance \eqref{varla}.
Estimating this is probably a complex problem even for the
case when all the $p_j$ are equal.
\end{remark}

\begin{remark}
Observe that in $\vec G(\Z,p)$, with positive probability, there are many maximal length
paths between successive skeleton points.
Since $\lambda > 0$, this
implies that the number of linear extensions of the random partial order
induced by the edges  on $n$ vertices
increases exponentially fast,  and this allows one to obtain directly the CLT
for the logarithm number of linear extensions \cite{ABBJ94}.
\end{remark}

\begin{remark}
Note that the events appearing in the definition of $\mathscr S$
depend on the entire random sequence of edges but, nevertheless,
produce a regenerative structure.
This phenomenon, in a more general context, has been studied in \cite{FZ}.
\end{remark}

\section{Directed random graphs on partially ordered sets}
\label{sec:posets}
A direction towards generalizations of the Barak-Erd\H{o}s graph
on the set of integers, is to consider a set of vertices $V$
equipped with a partial order. The directed random graph
must respect the partial order of $V$.
We will consider two cases below, that of $V=\Z \times I$ where
$I$ is a finite set and that of $V=\Z\times \Z$.
We shall deal with law of large numbers and functional central
limit theorems and see that asymptotic normality fails.
There is an interesting connection with Brownian last passage
percolation \cite{GTW2001,BAR2001}. The Brownian last passage
percolation process is the one defined by \eqref{BLPPP} below.
The results of this section are taken from \cite{DFK12} and \cite{KT13}.

\subsection{Brownian last passage percolation}
Consider $V=\Z \times I$ where, for simplicity,
let $I=\{0,1,\ldots,M\}$ for a positive integer $M$.
Elements of $V$ are denoted by $(u,i)$, $(v, j)$, etc.
Define the standard partial order on $V$, denoted by $\ll$ by
\[
(u,i) \ll (v,j) \text { if } (u,i) \neq (v,j) \text{ and }
u \le v, i \le j.
\]
Consider the random graph
$\vec G(V, p)$, with $0<p<1$ with edges defined as follows.
A pair of vertices $(u,i)$, $(v,j)$, with $(u,i) \ll (v,j)$,
form an edge directed from $(u,i)$ to $(v,j)$
with probability $p$, independently from pair to pair.
We consider all directed paths from $(u,i)$ to $(v,j)$ and denote
by $L_{(u,i),(v,j)}$ the maximum length of all such paths.
We also let
\[
\text{
$L^*_{u,v} = $   maximum length of all paths with endpoints in the
$[u,v] \times I$}, \quad L^*_n := L_{n}^*.
\]
We are interested in the LLN and CLT for $L^*_{0,n}$ as $n \to \infty$.

Regarding the LLN the ergodic arguments of Section \ref{ergosec} go through,
provided we consider the shift $\theta$ that acts in the
horizontal direction only.
We can then easily obtain that, for the same constant $C=C(p)$ corresponding to
the Barak-Erd\H{o}s graph $\vec G(\Z, p)$, we have
\[
L^*_{n}/n \to C,  \text{ as $n \to \infty$, a.s.\ and in $L^1$},
\]
where $C$ is the constant for the Barak-Erd\H{o}s graph
$\vec G(\Z,p)$.

The CLT is more interesting.
It involves a process  that appears e.g.\ in \cite{OCONYOR2002}.
\begin{theorem}
\label{thmslab}
Consider the random graph $\vec G(\Z \times I, p)$, $I=\{1,\ldots,M\}$
and let $L^*_{n}$
be the maximum length of all paths in $\{0,1,\ldots, n\} \times I$.
Then, with $C, \sigma$ as in Theorem \ref{barbarossa},
\[
\ell_n(t) := \frac{L^*_{[nt]}-Cnt}{\sigma \sqrt{\lambda n}},
\quad t \ge 0, \quad n \in \N,
\]
as a sequence of random elements of $D[0,\infty)$ with the local
uniform topology, converges in distribution to the process
\begin{equation}
\label{BLPPP}
Z_M(t) := \max_{0=t_0\le t_1 \le \ldots \le t_M=t}\,
\sum_{i=1}^M \big(B^{(i)}(t_i) - B^{(i)}(t_{i-1})\big),
\end{equation}
where $B^{(1)}, \ldots, B^{(M)}$ are independent standard Brownian motions.
\end{theorem}
\begin{proof}[Sketch of proof]
If $\SSS^{(i)}$ denotes the set of skeleton points of the restriction of
$\vec G(\Z\times I,p)$ on the line $\Z \times\{i\}$,
we have that the $\SSS^{(i)}$ are independent and all equal in distribution to
the set of skeleton points of the Barak-Erd\H{o}s graph $\vec G(\Z,p)$.
Hence all the $\SSS^{(i)}$ are stationary renewal processes on $\Z$
and also aperiodic (in the sense that the greatest common divisor
of the positive integers in the support of the distance between
successive skeleton points is $1$).
Therefore the set $\cap_{i \in I} \SSS^{(i)}$ is also a renewal process
with positive rate \cite{LIN92} and the even smaller set
\[
\SSS^I =\big\{x \in \bigcap_{i \in I} \SSS^{(i)}:\, \text{ for all $i, j \in I$,
$i<j$, there is an edge between $(x,i)$ and $(x,j)$}\big\}
\]
is still a renewal process as it is obtained by thinning.
Moreover, $\SSS^I$ is stationary and ergodic (with respect to $\theta)$
and has positive rate. Hence $\SSS^I$ is an random subset of $\Z$ that is
infinite in both directions.
Let $\Gamma_k^I$, $k \in \Z$, be an enumeration of the elements of $\SSS^I$
with $\Gamma_k^I < \Gamma_{k+1}^I$ for  all $k \in \Z$, and $\Gamma^I_0\le 0
< \Gamma_1^I$.
Then, as in the Barak-Erd\H{o}s case (see Lemma \ref{indg}) is we let
$\vec G_k$ be the induced subgraph of $\vec G(\Z\times I, p)$
on $\{(u,i) \in \Z \times I:\, \Gamma_k^I \le u \le \Gamma_{k+1}\}$,
we have that, conditional on $\{0 \in \SSS^I\}$,
the $\vec G_k$, $k \in \Z$, are i.i.d.
Therefore, a maximal-length path on $\{0,\ldots,n\} \times I$
is necessarily a path from $(0,1)$ to $(n,M)$ such that if
$0 \le \Gamma_k^I \le n$ then such a path passes via a vertex whose
horizontal coordinate is $\Gamma_k^I$; see Figure \ref{maxpass}.
\begin{figure}[ht]
\centering
\includegraphics[height=4cm]{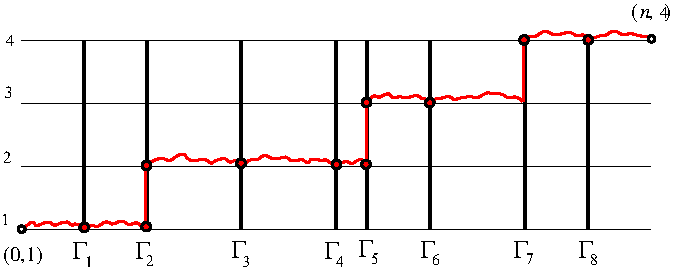}
\caption{A maximal length path on $[0,n]\times I$ passes through all
intermediate skeleton points and switches line at some of them.}
\label{maxpass}
\end{figure}
Using this and breaking $\ell_n(t)$ as in \eqref{ldec},
we arrive at the result by using Donsker's theorem.
Details can be found in \cite{DFK12}.
\end{proof}

\subsection{Generalization to partially ordered vertex sets and distance-dependent probabilities}
In the above, we can take $I$ to be a finite partially ordered set $(I, \preceq)$
with a bottom element called $1$ and a top element called $M$.
We then equip $\Z\times I$ with a strict partial order $\ll$
defined by
\[
(u,i) \ll (v,j) \text { if } (u,i) \neq (v,j) \text{ and }
u \le v, i \preceq j.
\]
A pair of vertices $(u,i)$, $(v,j)$, with $(u,i) \ll (v,j)$,
form an edge directed from $(u,i)$ to $(v,j)$
with probability $r_{v-u,i,j}$, independently from pair to pair.
As before, we let $L_{(u,i),(v,j)}$ the maximum length of paths
from $(u,i)$ to $(v,j)$.
A convenient set of assumptions for the probabilities $r_{n,i,j}$ is:
\begin{equation}   \label{rass}
\forall i \in I~ r_{n,i,i} =: p_n,
\quad 0<p_1<1, \quad \sum_{n=1}^\infty n \prod_{m=1}^n (1-p_m) < \infty.
\end{equation}
Let as denote by $\vec G(\Z\times I, (r_{n,i,j}))$ the resulting
random graph on $\Z \times I$ and let $L^*_{0,n}$ be the
maximum length of all paths with endpoints in $[0,n] \times I$.

Then Theorem \ref{thmslab} remains the same in form,
that is, we normalize $L^*_{0,[nt]}$ in the same way
and obtain that the normalized sequence of processes
converge in distribution to a process akin to \eqref{BLPPP}
but with an additional maximization since the presence of partial
order allows for more flexibility.

To express the limiting process we need the Hasse diagram $\mathbb H(I)$ of
the partially ordered set $(I,\preceq)$, which is the directed graph
on $I$ obtained by declaring there is an edge from $i$ to $j$
if
\[
\text{$i \preceq j$ and there
is no $k$, distinct from $i$ and $j$, such that $i \preceq k \preceq j$}.
\]
See \cite{DP90}.
Let $B^{(i)}$, $i \in I$, be independent standard Brownian motions.
For any path $\iota=(\iota_0,\ldots,\iota_r)$ in $\mathbb H(I)$ we write
\[
Z^{(\iota)}(t) := \max_{0 = t_0 \le t_1 \le \cdots \le t_r=t}
\big\{
B^{(\iota_0)}(t_1) + [B^{(\iota_1)}(t_2) - B^{(\iota_1)}(t_1)]
+ \cdots + [B^{(\iota_r)}(t_r) - B^{(\iota_r)}(t_{r-1})]\big\},
\]
and then let
\begin{equation}
\label{BLPPP2}
Z_I(t) := \max_\iota Z^{(\iota)}(t),
\end{equation}
where the maximum is taken over all paths $\iota$ in $\mathbb H (I)$
from $0$ to $M$.
Then $Z_I$ is the limit of the sequence of normalized processes.

\subsection{Convergence to the Tracy-Widom distribution}
\paragraph{Self-similarity.}
It is clear that the process $Z_I$ of \eqref{BLPPP2} is not Gaussian,
but is is continuous and self-similar:
\[
(Z_I(ct))_{t \ge 0} \eqdist c^{1/2} (Z_I(t))_{t \ge 0},
\]
for any $c>0$.
In particular, the process $Z_M$ of \eqref{BLPPP} is
a special case of \eqref{BLPPP2},  when $I=\{1,\ldots,M\}$, has well-known
connection with random matrix theory; see below.

\paragraph{Queueing theory.}
Glynn and Whitt \cite{GW} considered an infinite
number single-server FCFS infinite-buffer queues connected in series.
At time $0$ there are $M$ customers in the first one and none in the others.
Denote by $\sigma_{m,n}$ the service time of customer $m$ in queue $n$.
It is assumed that the $\sigma_{m,n}$, are i.i.d.\ random
variables with finite variance.
One quantity of interest in queueing theory is the time $L_{m,n}$
that customer $m$ departs from queue $n$. For this, we have an obvious recursion,
namely,
\[
L_{m,n} = \max(L_{m-1,n}, L_{m,n-1}) + \sigma_{m,n},
\]
because if customer $m$ finds, upon arrival, queue $n$ occupied,
it has to wait until the previous customer departs
from queue $n$ at time $L_{m,n-1}$, so $L_{m,n}= L_{m,n-1}+\sigma_{m,n}$;
and if it finds the queue empty then
$L_{m,n} = L_{m,n-1}+\sigma_{m,n}$.
We can easily solve the recursion and express it as follows.
Consider $\N\times \N$ as a directed graph where each vertex $(i,j)$
has two outgoing edges, one to $(i,j+1)$ and one to $(i+1,j)$.
Equip each $(i,j) \in \N \times \N$ with weight $\sigma_{i,j}$.
Let $\pi$ be a path in this graph. Then $w(\pi)$ is the sum
of the weights of the vertices of this path.
We obtain
\[
L_{m,n} = \max\{w(\pi):\, \pi \text{ is path from $(1,1)$ to $(m,n)$}\},
\quad (m,n) \in \N \times \N.
\]
Consider now the normalized departure times,
\[
\ell_m^{(n)} := \frac{L_{m,n} - n \E \sigma}{\sqrt{n \var\sigma}},
\]
where $\sigma$ has the same law as (any of) the $\sigma_{m,n}$.
Then \cite{GW} show that
\[
\ell_M^{(n)} \to Z_M(1), \text{ as } n \to \infty, \text{ in distribution},
\]
where $Z_M$ is as in \eqref{BLPPP}.
For subsequent work on the symmetry, duality, and other
quantities of interest, see \cite{BBM}.

\paragraph{The GUE.}
Let $H$ be an $M\times M$ GUE (Gaussian Unitary Ensemble) random matrix \cite{MEHTA04}.
This is a random element of the set of complex Hermitian $M\times M$ matrices
such that each diagonal element has standard complex Gaussian distribution
and each off-diagonal element has standard real Gaussian distribution.
(A standard complex Gaussian distribution is the law of $Z=X+\sqrt{-1} Y$
where $X, Y$ are
i.i.d.\ Gaussian random variables with $\E X = \E Y =0$,
$\E X^2 = \E Y^2 = 1/2$.)
Moreover, all elements $H_{k,\ell}$, $k \le \ell$, are independent
(and $H_{\ell,k} = \overline H_{k,\ell}$ for $k \le \ell$).
The eigenvalues of $H$ are real random variables and the largest of them
is denoted by $\lambda_M$.
It is clear that $H$ satisfies $H \eqdist UHU^{-1}$ for any unitary
($U^{-1} = U^*$) complex matrix $U$.
It was shown independently by Baryshnikov \cite{BAR2001} and Gravner et al. \cite{GTW2001}, thus answering an open question by
Glynn and Whitt \cite{GW}, that
$Z_M(1)$ satisfies
\begin{equation}
\label{Zlambda}
Z_M(1) \eqdist \lambda_M,
\end{equation}
thus establishing a connection between Brownian last passage percolation
and random matrix theory.

\paragraph{CLT for the largest GUE eigenvalue.}
The law of $\lambda_m$
satisfies $\lambda_m/\sqrt{m} \to 2$.
The CLT for $\lambda_m$ is
\begin{equation}
\label{Horatio}
m^{1/6} (\lambda_m - 2 \sqrt{m})
\xrightarrow{\text{d}} F_2,
\end{equation}
where the distribution function $F_2$ a determinantal form:
\[
F_2(x) = \det(I-A_K)_{L^2[0,x]},
\]
where $A_K$ is the operator on $L^2[0,x]$ with kernel $K$ defined
\[
K(x,y) = \frac{\Ai(x) \Ai'(y) - \Ai'(x) \Ai(y)}{x-y},
\]
and where  $\Ai(x)$ is the principal Airy function,
defined as the solution $y=Ai(x)$ to the linear ODE $y'' - xy=0$
with boundary condtion $y \to 0$ as $x \to \infty$.
Fourier-transforming the ODE, at least formally, easily yields
that $\Ai(x) := \pi^{-1} \int_0^\infty \cos(\tfrac13 \omega^3+\omega x) d\omega  =
\pi^{-1} \lim_{\Omega \to \infty}\int_0^\Omega  \cos(\tfrac13 \omega ^3+\omega x) d\omega $.
Note that $K(x,x)$ is defined as a limit when $y \to x$.
The determinant above is to be understood as the Fredholm determinant
of the operator $A_K$ acting on $L^2[0,x]$.
This was established by Tracy and Widom \cite{TWfred94} and the
distribution $F_2$ is known as the Tracy-Widom law, see also Anderson, Guionnet and Zeitouni \cite{AGZ}.

\paragraph{Barak-Erd\H{o}s graph on $\N \times \N$.}
Consider the graph $\vec G(\N \times \N, p)$, with $0<p<1$.
We equip $\N \times \N$ with the natural partial order;
$(i,j)$ is before $(i',j')$ if $i\le i'$, $j\le j'$ and $(i.j) \neq (i',j')$.
If $(i,j)$ is before $(i,j')$ we put and edge between them with probability $p$,
independently from pair to pair of comparable vertices.
Paths in this graph move in a ``northeast'' direction (including the north and
the east).
We let $L_{n,m}$ be the maximum length of all paths from $(1,1)$ to $(n,m)$.
The induced subgraph of $\vec G(\N \times \N, p)$ on any horizontal
line is a $\vec G(\N, p)$.
Consider the $m$-th line and apply Theorem \ref{barbarossa}.
We obtain
\[
\left(\frac{L_{[nt].m}-Cnt}{\sigma \sqrt{\lambda n}}\right)_{t \ge 0}
\xrightarrow{\text{d}} Z_m, \text{ as } n \to \infty.
\]
On the other hand, by self-similarity and \eqref{Zlambda}, we have
\[
Z_m(t) \eqdist \sqrt{t} \lambda_m,
\]
and $m^{1/6} (\lambda_m - 2\sqrt{m}) \xrightarrow{\text{d}}  F_2$.
It is then natural to if we can obtain a limit for a normalized
$L_{n,m}$ when $n$ and $m$ tend to infinity simultaneously.
To see what kind of scaling we can expect, write \eqref{Horatio} as
\[
m^{1/6} \left(\frac{Z_m(t)}{\sqrt{t}}-2\sqrt{m}\right) \xrightarrow{\text{d}}  F_2,
\text{ as } m \to \infty.
\]
A statement of the form $X(t,m) \xrightarrow[m \to \infty]{\text{(d)}}  X$,
where the distribution of $X(t,m)$ does not depend on the choice of $t>0$,
implies the statement
$X(t, m(t)) \xrightarrow[t \to \infty]{\text{d}}  X$, for any function $m(t)$
such that $m(t) \xrightarrow[t \to \infty]{} \infty$. Hence,
upon setting $m=[t^a]$, we have
\begin{equation*}
\label{Zdef}
t^{a/6} \bigg(\frac{Z_{[t^a]}(t)}{\sqrt{t}}-2\sqrt{t^a}\bigg)
\xrightarrow[t \to \infty]{\text{d}}
F_2.
\end{equation*}
Therefore, it is reasonable to guess that an analogous limit theorem
holds for a centered scaled version of the largest
length $L_{n,[n^a]}$, namely that
\begin{equation}
\label{Llim}
n^{a/6} \bigg(\frac{L_{n,[n^a]}-c_1 n}{c_2 \sqrt{n}}-2 \sqrt{n^a}\bigg)
\xrightarrow[t \to \infty]{\text{d}}
F_2,
\end{equation}
where $c_1, c_2$  are appropriate constants.
Indeed, \eqref{Llim} holds, with $c_1=C=C(p)$ and $c_2 = \lambda \sigma^2$
 and $a$ sufficiently small.
This was proved in \cite[Theorem 6.1]{KT13}
for the more general case of graphs where the edge probabilities
may not be constant, as in Section \ref{secreg}. The smallness of $a$
depends on a condition that involves the edge probabilities.

That paper used the idea of strong coupling with Brownian motions
(Koml\'os-Major-Tusn\'ady \cite[Theorem 4]{KMT75,KMT76}),
as in \cite{BM2005}, a paper dealing with last passage percolation
on $\N \times \N$ with random weights on the vertices.
The difficulty in the proof of the main theorem in \cite{KT13} is that
if we consider the intersection of the sets of skeleton points corresponding
to each line then their intersection is empty.

We also refer the reader to the seminal paper of Johansson \cite{JOHANSSON00}
for last passage percolation on a lattice (with i.i.d.\ exponential weights
on vertices).
There are also exciting connections with processes fundamental
to modern probability theory and mathematical physics, such as the
Kardar-Parizi-Zhang (KPZ) universality class.
We refer to the recent works by Dauvergne {\em et al} \cite{DAUV22,DAUV24} on
the convergence of the Brownian last passage percolation
to the central object in the KPZ class, namely the {\em directed landscape}.

\section{Weighted Barak-Erd\H{o}s graphs}
\label{sec:wei}
We now sketch results on last passage percolation on the Barak-Erd\H{o}s
graph with random weights on its edges; we refer
to \cite{FMS14} and \cite{FK18} for details.
Recall that existence
of an edge $(i,j)$ is encoded by a random variable $\alpha_{i,j}$
that takes value $1$ with probability $p$ or $-\infty$
with probability $1-p$.
Let $u$ be a positive random variable with
distribution function $F(x)=\P(u\le x)$,
and let $u_{i,j}$, $i<j$, be a collection of i.i.d.\ copies of $u$.
If $\alpha_{i,j}=1$ then edge $(i,j)$ exists and has weight $u_{i,j}$.
We now have a weighted random graph that we will denote as $\vec G(\Z, p, u)$.
A path from $i$ to $j$ of length $\ell$ and weight $w$ is
a sequence $(i=i_0 < i_1 < \cdots < i_{\ell-1} <i_\ell=j)$
such that $\alpha_{i_0,i_1} = \cdots = \alpha_{i_{\ell-1},i_\ell}=1$
and $u_{i_0,i_1}+\cdots + u_{i_{\ell-1},i_\ell}=w$.
A {\em geodesic} from $i$ to $j$ is a maximum weight path from $i$ to $j$.

Different phenomena appear depending on whether $\E u^2$ is finite
or not.

\subsection{Finite variance weights}
Assume that $\E u^2 < \infty$.
We study
\[
W_{i,j} :=
\max_{\substack{i= i_0 < i_1 < \cdots < i_\ell = j\\\ell\in \N}}
\left(\sum_{k=1}^\ell \alpha_{i_{k-1},i_k} u_{i_{k-1}, i_k}\right)^+,
\]
the maximum weight of all paths from $i$ and $j$.
If we further define the larger quantity
$\widetilde W_{i,j}
:= \max_{\substack{i\le i_0 < i_1 < \cdots < i_\ell \le j\\\ell\in \N}}
\left(\sum_{k=1}^\ell \alpha_{i_{k-1},i_k} u_{i_{k-1}, i_k}\right)^+$,
the maximum weight of all paths with endpoints between $i$ and $j$
we obtain the subadditive inequality
\[
\widetilde W_{i,k} \le \widetilde W_{i,j} + \widetilde W_{j,k}
+ \max_{i \le x \le j \le y \le k}
u_{x,y}.
\]
We can then see that $\lim_{n \to \infty} \widetilde W_{0,n}/n$ exists
iff the expectation of the latter maximum is finite
which requires that the second moment of $u$ be finite.
Hence $\E u^2 < \infty$ is necessary and sufficient
for the above limit to exist.

To understand this, and to prepare the ground for the central
limit theorem, we consider two random subsets of the integers.
The first is the usual skeleton set $\SSS$ and the second is the set $\RRR_c$ of
{\em $c$-renewal points} where $c$ is a positive constant.
To define this, we first define the events
\begin{align*}
A_i^+  &= \{W_{i,i+n} > cn \text{ for all } n \ge 1\},
\\
A_i^-  &= \{W_{i-n,i} > cn \text{ for all } n \ge 1\},
\\
A_i^{-+}  &= \{\alpha_{i-m, i+n} u_{i-m, i+n}  < c(m+n)
\text{ for all } m, n \ge 1\},
\end{align*}
and then let
\[
\RRR_c = \{i \in \Z:\, A_i^+ \cap A_i^-\cap A_i^{-+}  \text{ occurs}\}.
\]
Clearly, $\RRR_c$ is a stationary and ergodic random set with density
\[
\mu(c,p) = \P(A_0^+ \cap A_0^-\cap A_0^{-+}).
\]
Since
\[
\RRR_c \subset \SSS,
\]
we have $\mu(c,p) \le \lambda(p)$, where $\lambda(p)$ is the
density of $\SSS$; see \eqref{ellq} and Remark \ref{elleuler}.
\begin{lemma}
Assume that $\E u^2 < \infty$ and $0< c < (\E u) (\E W_{\Gamma_1, \Gamma_2})$.
Then $\mu(c,p) > 0$.
Moreover, $\vec G(\Z, p, u)$ regenerates over $\RRR_c$.
\end{lemma}
\begin{proof}[Sketch of proof]
We first observe that $\P(A_0^+) = \P(A_0^-)$.
If we choose $0< c < (\E u) (\E W_{\Gamma_1, \Gamma_2})$
then $\P(A_0^+)>0$.
The finiteness of $\E u^2$
implies the positivity of $\P(A_0^{-+})$.
It can be shown that $A_0^+, A_0^-, A_0^{-+}$ are independent
and so $\mu(c,p) = \P(A_0^+) \P(A_0^-) \P(A_0^{-+})>0$.
For details see
\cite[Lemma 2]{FK18}.
For the last assertion see \cite[Lemma 3]{FK18}.
\end{proof}
This lemma is responsible for the law of large numbers:
\begin{theorem}
\label{CLTW}
Assume that $\E u^2 < \infty$ and $0 < p \le 1$.
Then there is a constant $C$ depending on
$p$ and the law of $u$ such that
\[
\lim_{n \to \infty} \frac{W_{0,n}}{n} =
\lim_{n \to \infty} \frac{\widetilde W_{0,n}}{n} = C
\text{ a.s. and in $L^1$}.
\]
\end{theorem}
We remark that the equality of the two limits is because of the
existence of the $c$-renewal points that have positive density.
The CLT holds provided that the third moment of $u$ is finite:
\begin{theorem}\label{CLT8}
Assume that $\E u^3 < \infty$ and $0 < p \le 1$.
Then the sequence of processes
\[
\left\{\frac{W_{0,[nt]}-C n t}{\sqrt{\lambda n}}, \, t \ge 0\right\}
\]
converges in distribution, as $n \to \infty$, to a
zero mean Brownian motion.
\end{theorem}

Theorems \ref{CLTW} and \ref{CLT8} may be complemented by a result
describing the behavior of the weight of the heaviest edge on
a geodesic path. Assuming that $F$ is continuous ensures that
there is a unique geodesic path from $0$ to $n$. Then we can define
\[
h_n := \text{ maximum weight of all edges on the geodesic path from
$0$ to $n$.}
\]
We also assume that the edge weight $u$ is regularly varying
with index $s$, in the sense that
\begin{equation}
\label{Frv}
\frac{1-F(tx)}{1-F(x)}\to t^{-s},
\text{ as } x\to\infty, \text{ for every } t>0.
\end{equation}
Of course, $s>2$ is needed in order that $\E u^2$ be finite.
When $2<s<3$, one can deduce
that the fluctuations of $W_{0,n}$ are of order larger than
$\sqrt{n}$, and so the central limit theorem cannot be extended to this
case.

\begin{theorem}
\label{longestedge}
Let the edge weight $u$ be a continuous random variable
that is also regularly varying with index $s>2$
in the sense of \eqref{Frv}.
Then we have
\begin{equation} \label{upperlower}
\frac{\log h_n}{\log n} \rightarrow \frac{1}{s-1}
\text{ in probability as } n \rightarrow \infty.
\end{equation}
In particular, if $2<s<3$ then
\[
\frac{\var W_{0,n}}{n} \rightarrow \infty,
\]
and a central limit theorem such as that in Theorem \ref{CLT8} cannot hold.
\end{theorem}
The proof of this can be found in \cite{FMS14}.

\subsection{Infinite variance weights}
\label{infsec}
Assume now that $\E u^2  = \infty$.
Under this condition, $W_{0,n}$ grows faster than linearly.
This can be seen by considering the contribution of the single heaviest
edge in $[0,n]$, and noting that
the expectation of the maximum of
$n^2$ i.i.d.\ random variables with infinite variance
has expectation that grows faster than $n$.
Since $W_{0,n}$ is at least as large as the weight of this single edge,
we have that $\E W_{0,n}/n \to \infty$ as $n\to\infty$, and
from Kingman's subadditive ergodic
theorem we can conclude that in fact $W_{0,n}/n \to\infty$ a.s.

As before, we assume that $u$ is a continuous random variable
such that the regular variation condition \eqref{Frv} holds.
We need $s< 2$ in order that $\E u^2$ be infinite.

New phenomena occur in the infinite variance case. In order to
describe them succinctly and avoid technicalities, we shall
further assume that $p=1$. That is, we only present results
for the $\vec G(\Z, 1, u)$ case.

\paragraph{Finite model}
Let $\vec G_n$ be a graph on $V_n=\left\{0, \frac1n, \ldots, \frac{n-1}{n},
1\right\}$,
edges $E_n = \{(i/n,j/n):\, 0 \le i < j \le n\}$,
and weights $u^{(n)}_e$, $e \in E_n$, that are i.i.d.\ copies of $u$.
We can think of a path $\pi$ of $\vec G_n$ as a collection
of edges $e_1, \ldots, e_\ell$ where the ending point of $e_i$ is the
starting point of $e_{i+1}$ for all $1 \le i < \ell$.
Let $\Pi_n$ be the set of all paths  in $\vec G_n$ from $0$ to $n$
(a set of size $2^{n-1}$).
The maximum weight of all paths in $\Pi_n$ can be written as
\[
W_{0,n} = \max_{\pi \in \Pi_n} \sum_{e \in \pi} u^{(n)}_e.
\]
The latter maximum will not increase if we throw in all {\em admissible} subsets
of the set of edges $E_n$, where we say that a set $A \subset E_n$ is admissible
if every pair of elements of $A$ are non-overlapping edges
in the sense that the endpoints of one are $\le$ the endpoints of the other.
If we let $\mathcal C_n$ be the set of all admissible sets of edges then
\begin{equation}
\label{w0n1}
W_{0,n} = \max_{\pi \in \mathcal C_n} \sum_{e \in A} u^{(n)}_e.
\end{equation}

We next introduce another way to construct $\vec G_n$.
This second construction can be used to define
a corresponding model on a continuous set of vertices
with an appropriately defined maximum path weight $W$
in such a way that a scaled version of $W_{0,n}$
converges to $W$ in distribution.
Note that $E_n$ has size
\[
N=N_n = \binom{n+1}{2}.
\]

Let $E_n = \{e_1, \ldots, e_N\}$ be an enumeration of the edges.
Let $M_{e_1}^{(n)}, M_{e_2}^{(n)}, \ldots, M_{e_N}^{(n)}$  be the order statistics
of the $u^{(n)}_{e_1}, u^{(n)}_{e_2}, \ldots, u^{(n)}_{e_N}$
That is,
$\{M_{e_1}^{(n)}, M_{e_2}^{(n)}, \ldots, M_{e_N}^{(n)}\} =
\{u^{(n)}_{e_1}, u^{(n)}_{e_2}, \ldots, u^{(n)}_{e_N}\}$
(as sets) and $M_{e_1}^{(n)} > M_{e_2}^{(n)} > \cdots > M_{e_N}^{(n)}$.

Let $Y_1^{(n)}, Y_2^{(n)}, \ldots, Y_{N}^{(n)}$ be a random
ordering of $\{e_1, \ldots, e_N\}$ chosen uniformly from all the $N!$
possibilities. {Assign weight $M_{e_i}^{(n)}$ to $Y_i^{(n)}$.}
We then have
\begin{equation} \label{C0n}
\mathcal{C}_{n} = \left\{ A \subset \left\{e_1,\ldots,e_N \right\} :
\text{ for every pair $\{e_i, e_j\} \subset A$ the  $Y_{e_i}^{(n)}$, $Y_{e_j}^{(n)}$
are non-overlapping} \right\}
\end{equation}
and
\begin{equation} \label{w0n2}
W_{0,n} = \max_{A \in \mathcal{C}_{n}} \sum_{e_i \in A} M_{e_i}^{(n)}
\end{equation}
which is equivalent {(the same in distribution)}
to the previous definition of $W_{0,n}$ in \eqref{w0n1}.

\paragraph{Infinite model}
We next define a random weighted graph $\vec G$ on countably infinite random set
of vertices.
Let $W_1,W_2,\ldots$ be a sequence of i.i.d.\
exponential random variables with mean 1 each.
Set
\[
M_k = (W_1 + \cdots + W_k)^{-1/s}, \quad k=1,2,\ldots.
\]
Let $U_1,U_2,\ldots$ and $V_1,V_2,\ldots$ be
two sequences of i.i.d.\ uniform random variables in $[0,1]$.
We further assume that $\{U_i\}$, $\{V_i\}$, $\{W_i\}$ are independent.
The edges of $\vec G$ are taken to be
\[
Y_i = (\min(U_i,V_i),\,\max(U_i,V_i)), \quad i = 1,2, \ldots.
\]
The $i$-th largest weight $M_i$ will be attached
to the $i$-th edge $Y_i$. In analogy to \eqref{C0n} we define
\[
\mathcal{C} = \left\{ A \subset
\left\{ 1,2, \ldots \right\} : Y_i \cap Y_j = \varnothing
\text{ for all pairs } \{i,j\} \subset A \right\}.
\]
In analogy to \eqref{w0n2} we let
\begin{equation} \label{w}
W = \sup_{A \in \mathcal{C}} \sum_{i \in A} M_i.
\end{equation}
A priori the random variable $W$ could be infinite, but
Theorem \ref{main1} below guarantees that it is almost surely finite.

\paragraph{Convergence results}\label{figsection}
The intuition behind the approximation of the finite model by the infinite one
is the following pair of convergence results.
First, for any $k \in \N$ we have
\begin{equation} \label{Yeq}
\left( Y_1^{(n)} , Y_2^{(n)} , \ldots , Y_k^{(n)} \right)
\convd  ( Y_1 , Y_2 , \ldots , Y_k )
\end{equation}
as $n \rightarrow \infty$, where we use the product topology on $([0,1]^2)^k$.

Following \cite{MARTIN06,HM07},
let $b_n =a_{N_n} = F^{(-1)}\left(1 - \frac{1}{N_n}\right)$
and put
\[
\widetilde{M}_i^{(n)} = \frac{M_i^{(n)}}{b_n}.
\]
(As an example, if
$F(x) = 1-x^{-s}$ for $x\geq 1$, then $b_n$ grows like $n^{2/s}$.
More generally, under assumption \eqref{Frv},
$\lim_{n\to\infty} \frac{\log b_n}{\log n} = 2/s$).
Then from classical results
in extreme value theory we have for any
$k \in \mathbb{N}$ that
\begin{equation} \label{Meq}
\left( \widetilde{M}_1^{(n)}, \widetilde{M}_2^{(n)}, \ldots , \widetilde{M}_k^{(n)} \right)
\convd \left( M_1 , M_2 , \ldots , M_k \right)
\text{ as } n \rightarrow \infty.
\end{equation}

In this way both the locations and weights of the heaviest edges
(i.e.\ edges with heaviest weights)
in the discrete model above are approximated
by their equivalents in the continuous model.
It is shown in \cite{FMS14} that the heaviest edges are the ones that bring dominant contribution to the maximum weight.
The precise convergence result is as follows.
\begin{theorem} \label{main1}
The random variable $W$ in \eqref{w} is almost surely finite.
If \eqref{Frv} holds with $s\in (0,2)$, then
$\frac{W_{0,n}}{b_n} \rightarrow W$ in distribution as $n \rightarrow \infty$.
\end{theorem}

For $n$ large, the heaviest edge in the geodesic has length on the order of $n$.
This is in contrast to the behavior in the case $\E u^2<\infty$,
where the important contribution to the maximal weight is given by edges of a
lighter order, see e.g.\ Theorem \ref{CLT8} above.

\section{Analytic properties of charged graphs}
\label{secana}
We now turn to some results concerning the behavior of the last
passage percolation constant $C$ of the Barak-Erd\H{o}s graph
as a function of edge weights. The problem is, in general, hard,
so we consider here the simple model introduced in \cite{FKP18}.
The next section will deal with a more general model.

A charged graph is a graph with possibly negative weights on its edges.
We are interested in last passage percolation on random directed
charged graphs. The charge of a path is the sum of the charges of its edges.
We are interested in the maximum charge of all paths.
If all charges are negative then the maximum of negative quantities is the
negative of a minimum of positive quantities, so the problem becomes that
of first passage percolation,
In view of this, we shall assume that some charges are nonnegative.

For the models of this and the next section we assume that the support of the
charge distribution is not a subset of $(-\infty, 0)$.
It will then turn out that the last passage percolation constant is positive,
so it matters little if we take positive part of the charge of a path
before maximization.

\subsection{The two-weights model}
\label{subsec:twoweight}
Let $x$ be a real number, possibly negative (which can be thought of
as a penalty).
To each pair $(i,j)$ of positive integers with $i<j$ we assign weight
or, rather, charge (since $x$ is allowed to be negative)
\[
w^x_{i,j}
= \begin{cases}
1 , & \text{ with probability } p
\\
x , & \text{ with probability } 1-p
\end{cases}.
\]
Let $\Pi_{i,j}$ be the set of
strictly increasing finite sequences of integers,
$i_0 < i_1 < \cdots < i_\ell$ such that $i_0=i$ and $i_\ell=j$.
An element $\pi \in \Pi_{i,j}$ is a finite path in the complete directed graph
on $\Z$ (that is, a graph such that every $(i,j)$ with $i<j$ is an
edge directed from $i$ to $j$).
So $\Pi_{i,j}$ is a deterministic set of size $2^{j-i}-1$.
Define
\[
w^x(\pi) = \sum_{k=1}^\ell w^x_{i_{k-1},i_k},
\quad \text{ if } \pi=(i_0,\ldots,i_\ell) \in \Pi_{i,j},
\]
and call $w^x(\pi)$ the charge or weight of $\pi$.
We are interested in
\begin{equation}
\label{wxij}
W^x_{i,j} = \max_{\pi=(i_0,\ldots,i_\ell)\in \Pi_{i,j}}
~\sum_{k=1}^\ell w^x_{i_{k-1},i_k}.
\end{equation}
We will use the notation $\vec G(\Z,p,x)$ to denote the directed
charged random graph on $\Z$
that contains as edges all $(i,j) \in \Z \times \Z$ with $i<j$
and which has with i.i.d.\ edge charges distributed as $w^x_{i,j}$.
We study the asymptotic growth rate $C(p,x)$
of $W^x_{i,j}$; see \eqref{Cpx} below.
The Barak-Erd\H{o}s graph is still denoted by $\vec G(\Z,p)$
and, as usual, $C(p)$ is the maximal path growth rate.
The parameter $p$ is fixed throughout this section.
We shall be interested in the behavior of the model when $x$ varies.

We may extend the model by letting $x$ range in $\R \cup\{-\infty, +\infty\}$.
The case $x=+\infty$ is uninteresting as being trivial.
The case $x=-\infty$ formally corresponds to the Barak-Erd\H{o}s graph,
the reason being as follows.
The charge of $i_0, i_1, \ldots, i_\ell$ equals $-\infty$ iff
$w^x_{i_k, i_{k+1}} = -\infty$ for some $k$, and this is equivalent to
$i_0, i_1, \ldots, i_\ell$ not being a path in $\vec G(\Z, p)$.
Hence $(W^{-\infty}_{i,j})^+ = L^{\gl,\gr}_{i,j}$
and $\vec G(\Z, p,-\infty) = \vec G(\Z, p)$.

An alternative way to think of $\vec G(\Z,p,x)$ is by letting all
edges of $\vec G(\Z,p)$ be {\blue blue} and all non-edges be {\red red}.
Blue edges have charge 1; red edges have charge $x$.

\begin{theorem}
For $-\infty < x < \infty$ we have
\begin{equation}
\label{Cpx}
\lim_{n \to \infty} \frac{W^x_{0,n}}{n}
= \lim_{n \to \infty} \frac{(W^x_{0,n})^+}{n}
= \inf_{n \in \N} \frac{\E(W^x_{0,n})^+}{n} =: C(p,x),
\end{equation}
a.s.\ and in $L^1$.
\end{theorem}
The constant $C(p,x)$ is defined through this theorem.
The theorem is proved by using ergodic arguments and Kingman's theorem,
thanks to the superadditive inequality
\[
W^x_{i,k} \ge W^x_{i,j} + W^x_{j,k}, \quad i < j < k.
\]

Further properties of $C(p,x)$ are in Theorem \ref{thmmain} below.
We shall use the following notations.
\begin{align*}
\mathbb X_C &:= \{x \in \R:\, C(p,\cdot) \text{ is not differentiable at } x\}
\\
\Q^* &:= \Q \setminus \Z
\\
\mathbb Y &:= \{q \in \Q^*:\, q < 0\}
\cup\{0\} \cup \{\tfrac12, \tfrac13, \ldots\}
\cup \{2,3,\ldots\}
\end{align*}
\begin{theorem}
\label{thmmain}
$~$
\begin{enumerate}[(i)]
\item Scaling. $C(p,x) = x C(1-p,x)$, for all $x>0$.
\item Continuity at $-\infty$. $\lim_{x \to -\infty} C(p,x) = C(p)$.
\item Convexity. $C(p,x)$ is increasing convex over $x \in \R$ and strictly positive.
\item Asymptotic growth. $\lim_{x \to \infty} C(p,x)/x =
C(1-p,0)=\left(\sum_{n=1}^\infty p^{\frac12 n(n-1)}\right)^{-1}$.
\item Nondifferentiability. $\mathbb X_C = \mathbb Y$.
\end{enumerate}
\end{theorem}

The proof of (v) is the most complex and shall only be sketched below.
Property (i) follows from by comparing $\vec G(\Z, p, x)$ with
$\vec G(\Z, 1-p, 1/x)$.
Property (iii) follows from the fact that $W^x_{i,j}$ is
an increasing convex function of $x$.
Strict positivity is due to $C(p,x) > C(p) > 0$ for all $x$.
The first part of Property (iv) is a consequence of continuity of the convex function
and the scaling property (i).
The last formula of (iv) is due to Dutta \cite{Dutta}.

To deal with (v), we start by further elaborating on
skeleton points.

\subsection{Further structure of inter-skeleton pieces}
As usual, $[m,n]$ stands for the set of all integers $k$ with $m \le k \le n$.
That the set $\SSS$ of skeleton points of $\vec G(\Z,p)$ splits
$\vec G(\Z,p,x)$ into  independent parts is not a surprise;
see Lemma \ref{indg}.
Hence, if $(i_0, \ldots, i_\ell)$ achieves the maximum
in the right-hand side of the definition \eqref{wxij} for $W^x_{i,j}$,
then $\{i_0, \ldots, i_\ell\}$ contains all elements of $\SSS\cap [i,j]$.

Enumerate the elements $\Gamma_k$ of $\SSS$ as in \eqref{gammapoints}.
Let $\vec G(I, p)$ be the Barak-Erd\H{o}s graph on the (possibly
random) set of vertices $I \subset \Z$.

Fix a positive integer $n$.
Consider the event $A^+_{0,n}$ that there is a path in $\vec G(\Z,p)$
from $0$ to every vertex in $(1,n]$
and the event $A^-_{0,n}$ that there is a path from
every vertex in $[0,n-1)$ to $n$; see \eqref{A+-}.
In addition, let
\[
F_{0,n}:= \big\{\text{for all $0<j<n$ there is $0<i<n$
with no path from $\min(i,j)$ to $\max(i,j)$}\big\},
\]
and then set
\[
H_{0,n} := A^+_{0,n} \cap A^-_{0,n} \cap F_{0,n}.
\]
The following is an expression for the distribution of
$\vec G([\Gamma_0,\Gamma_1],p)$ conditional on $\{\Gamma_0=0\}$.
\begin{proposition}
\label{menon}
For any nonnegative
deterministic functional $\phi\big(\vec G([\Gamma_0,\Gamma_1],p\big)$
of $\vec G([\Gamma_0,\Gamma_1],p)$ we have
\begin{equation}
\label{x1}
\E \left\{\phi\big(\vec G([\Gamma_0,\Gamma_1],p\big) \big\vert \Gamma_0=0\right\}
=
\sum_{n=1}^\infty
\E \left\{\phi\big(\vec G([0,n],p\big); H_{0,n} \right\},
\end{equation}
In particular,
\begin{equation}
\label{x2}
\P(\Gamma_1-\Gamma_0=n|\Gamma_0=0) = \P(\Gamma_2-\Gamma_1=n) = \P(H_{0,n}).
\end{equation}
Furthermore,
\begin{equation}
\label{x3}
C(p,x) = \lambda \E W^x_{\Gamma_1,\Gamma_2}
= \lambda \sum_{n=1}^\infty \E ( W^x_{0,n}; H_{0,n}),
\end{equation}
where $\lambda=\lambda(p)$ is as in \eqref{ellq}.
\end{proposition}
For a proof see \cite{FKP18}. The intuition should be clear
because, if, say, $A^+_{0,n} \cap A^-_{0,n}$ holds
(in which case both $0$ and $n$ are in $\SSS$)
but $F_{0,n}$ fails, then there is an element of $\SSS$ strictly
between $0$ and $n$.
The first equality of \eqref{x3} is by standard renewal theory.
The second follows from \eqref{x1}.

\begin{remark}
From \eqref{x2} we can obtain some values of the distribution of
$\Gamma_2-\Gamma_1$:
\[
\P(\Gamma_2-\Gamma_1=n) =
\begin{cases}
p, & \text{ if } n=1\\
0, & \text{ if } n=2\\
p^4(1-p), & \text{ if } n=3\\
p^7 (1-p)^3 + 3 p^5 (1-p)^2, & \text{ if } n=4
\end{cases}
\]
The smallest value of $\Gamma_2-\Gamma_1$ is $1$.
Suppose that $\Gamma_2-\Gamma_1\ge 2$.
Note that $\vec G([\Gamma_1, \Gamma_2],p)$ must necessarily
have the edges $(\Gamma_1, \Gamma_1+1)$ and $(\Gamma_2-1, \Gamma_2)$.
If $\Gamma_2-\Gamma_1$ were allowed to take value $2$
then $\Gamma_1+1 = \Gamma_2-1$ would have been a skeleton point
strictly between $\Gamma_1$ and $\Gamma_2$, which is impossible.
This explains why $\P(\Gamma_2-\Gamma_1=2)=0$.
\end{remark}

\subsection{Criticality and nondifferentiability}
Let $\pi=(i_0,\ldots,i_\ell)$ be a finite strictly increasing
sequence of integers.
Let $N(\pi)$ be the number of $(i_{k-1},i_k)$, $k=1,\ldots, \ell$
that are also edges of $\vec G(\Z,p)$.
Thus, $N(\pi)$ is the number of blue edges of $\pi$.
Similarly, $\overline N(\pi)$ is the number of red edges.
Then
\[
w^x(\pi) = \sum_{k=1}^\ell w^x_{i_{k-1},i_k}
= N(\pi) + x \overline N(\pi).
\]
Define
\[
\Pi^x_{i,j} :=\{\pi \in \Pi_{i,j}:\, w^x(\pi) = W^x_{i,j}\},
\]
a random set of paths with maximal charge.
Let $D^+$ denote right derivative with respect to $x$;
similarly $D^-$ for left derivative.
Using \eqref{x3} and the dominated convergence theorem we obtain
\begin{align*}
D^+C(p,x) &= \lambda \E D^+ W^x_{\Gamma_1,\Gamma_2}
=\lambda \E \max_{\pi  \in\Pi^{x}_{\Gamma_1,\Gamma_2}} \overline N(\pi),
\\
D^-C(p,x) &= \lambda \E D^- W^x_{\Gamma_1,\Gamma_2}
=\lambda \E \min_{\pi  \in\Pi^{x}_{\Gamma_1,\Gamma_2}} \overline N(\pi).
\end{align*}

\begin{remark}
To justify the last equalities in these two displays we simply observe that,
by \eqref{wxij}, $W^x_{\Gamma_1,\Gamma_2}$ is the maximum of
the affine functions $x \mapsto w^x(\pi) = N(\pi) + x \overline N(\pi)$,
over all $\pi \in \Pi_{\Gamma_1,\Gamma_2}$ and there are finitely many such $\pi$.
But if we fix $x$, we may further restrict the maximum over all
$\pi \in \Pi^x_{\Gamma_1,\Gamma_2}$.
More generally, if $T$ is an index set, and if $a_t, b_t$, $t \in T$, are
real numbers, then the function $f(x) = \max_t (a_t+b_t x)$
is convex and thus has right and left derivatives for each $x$,
given by $D^+ f(x) = \lim_{h \downarrow 0} (1/h) (f(x+h)-f(x))
= \sup_{t \in T^x} b_t$, where $T^x = \{t \in T: a_t+b_t x=f(x)\}$;
and similarly for $D^-f(x)$.
In the language of convex analysis \cite{ROCK70},
we say that set of affine functions
$x \mapsto a_t+b_t x$, $t \in T^x$, are supporting lines of the function
$f$ at the point $x$ (the remaining ones are irrelevant) and these are the
only ones responsible for the right and left derivatives of $f$ at $x$.
\end{remark}

\begin{definition}[the sets $\fG_n$ and $\fH_n$]
Let $\fG_n$ be the set of all directed graphs on $[0,n]$,
that is, graphs whose edge directions
compatible with the natural ordering of integers.
Let $\fH_n \subset \fG_n$ contain all $G \in \fG_n$ such that,
for all $0<j<n$,
\begin{enumerate}[\hspace*{3mm}(a)]
\item there is a path in $G$ from $0$ to $n$ containing $j$;
\item for some $i \neq j$ there is no path in $G$ from $\min(i,j)$
to $\max(i,j)$
\end{enumerate}
\end{definition}

For any path $\pi$ in a graph $G$ we let $N_G(\pi)$ be the number
of edges of $\pi$ that are also edges of $G$ and we let
$\overline N_G(\pi)$ the number of edges of $\pi$ that are
not edges of $G$.

\begin{definition}[criticality]
\label{defcrit}
We say that $x \in \R$ is {\em critical} if there is a positive integer $n$
and and a graph $G \in \fH_n$ possessing two paths $\pi_1, \pi_2$
such that
\begin{enumerate}[\hspace*{3mm}(a)]
\item
$N_G(\pi_1) + x \overline N_G(\pi_1) =
N_G(\pi_2) + x \overline N_G(\pi_2)
= \max_{\pi \in \Pi_{0,n}} (N_G(\pi) + x \overline N_G(\pi))$
\item
$\overline N_G(\pi_1)  \neq \overline N_G(\pi_2)$.
\end{enumerate}
\end{definition}
Criticality is a property of real numbers.
We let
\[
\mathbb X_\crit =\{x \in \R: x \text{ is critical}\}.
\]

Using Proposition \ref{menon} we can prove
\begin{theorem}
\label{Xcc}
\[
\mathbb X_C = \mathbb X_\crit.
\]
\end{theorem}
In view of this, Theorem \ref{thmmain}(v) is equivalent to
\begin{equation}
\label{YYY}
\mathbb X_\crit = \mathbb Y,
\end{equation}
which is a deterministic problem of graph-theoretic/combinatorial nature.
Proving this is rather complicated. We shall sketch some aspects of it
below.

\subsection{\texorpdfstring{$\mathbb  Y \subset \mathbb X_\crit$}{Y C Xcrit}}
Let $x \in \R$, $G \in \fG_n$ and $\pi \in \Pi_{0,n}$.
Think of the common edges of $\pi$ and $G$ as blue and the remaining
edges of $\pi$ as red. There are $N_G(\pi)$ blue and $\overline N_G(\pi)$
red edges. The path $\pi$ is called {\em $(x,G)$-maximal} if it
achieves the maximum in
$\max_{\pi \in \Pi_{0,n}} (N_G(\pi) + x \overline N_G(\pi))$.

Showing that $x \in X_\crit$ one proceeds by constructing a $G \in \fH_n$
possessing at least two $(x,G)$-maximal paths with
different number of red edges.
Figure \ref{figsumm} summarizes how this is done for
each element $x$  of $\mathbb Y$, except $x=2,3,\ldots$,
since if we know that $1/k \in \mathbb X_\crit$ ($k=2,3,\ldots$)
then we also know that $k \in \mathbb X_\crit$ by the scaling property.
For example, Figure \ref{figsumm}(a) shows that $0 \in \mathbb X_\crit$.
Indeed, consider the paths $\pi_1 = (0,1,3)$ and $\pi_2 = (0,1,2,3)$.
They both have weight $2$ (which is maximal), however, the first
has no red edges but the second has one red edge.
This means that $0$ satisfies the definition of criticality.

\begin{figure}[ht]
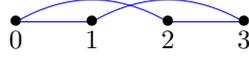
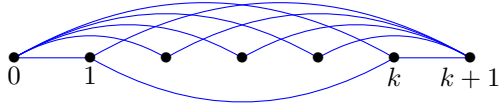
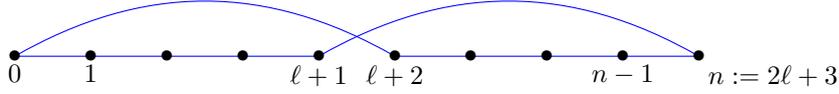
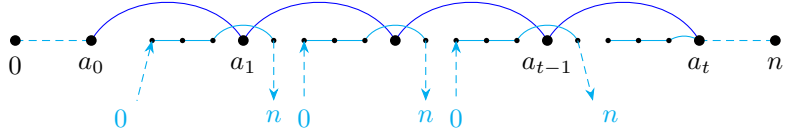

\centering
\begin{subfigure}{.8\textwidth}
\centering
\include{figTex/fig_0crit}
\caption{Each red edge (i.e.\ each non-edge) has charge $0$.
This graph shows that $0 \in \mathbb X_\crit$.}
\end{subfigure}

\begin{subfigure}{.8\textwidth}
\centering
\include{figTex/fig_+crit}
\caption{Let $k \in \{2,3,\ldots\}$. Each non-edge has charge $1/k$. This graph shows that $1/k \in \mathbb X_\crit$.}
\end{subfigure}

\begin{subfigure}{.8\textwidth}
\centering
\include{figTex/fig_negcrit}
\caption{Let $-\ell \in \{-1,-2,\ldots\}$.
Each non-edge has charge $-1/\ell$.
This graph shows that $-\ell \in \mathbb X_\crit$.}
\end{subfigure}

\begin{subfigure}{.8\textwidth}
\centering
\include{figTex/fig_lastgraph}
\caption{Let $x <0$, rational, but not an integer. Each non-edge has charge $x$.
Write $x=-\ell+(s/t)$, $\ell, s, t \in \N$, $t>1$, $\gcd(s,t)=1$, $s<t$.}
\end{subfigure}
\caption{Graphs exhibiting that $\mathbb Y \subset \mathbb X_\crit$.}
\label{figsumm}
\end{figure}

The most complex case is that of the last line of Figure \ref{figsumm}.
We explain the construction of the graph.
We first need the following number-theoretic lemma that can be thought of as
placing $n$ balls in $N$ bins ``as uniformly as possible''.
Let $\lfloor x \rfloor$, $\lceil x \rceil$ be the lower, respectively upper,
parts of $x$.
\begin{lemma}[Corollary to Sturm's lemma \cite{LOT}]
\label{sturmlem}
Let $N, n$ be positive integers,  $N\ge n$. Then
there exists a unique finite sequence $v=(v_1,\ldots,v_N)$ of elements
of $\{0,1\}$ such that, for all $0\le i<j\le N$,
the following hold:
\begin{align*}
\sum_{i < k \le j} v_k \in \bigg\{ \bigg\lfloor (j-i) \frac{n}{N}
\bigg\rfloor, \,  \bigg\lceil (j-i) \frac{n}{N} \bigg\rceil \bigg\} , \quad
 \sum_{i < k \le N} v_k = \bigg\lfloor (N-i) \frac{n}{N}  \bigg\rfloor,\quad
 v_1=1.
\end{align*}
\end{lemma}
We say that the sequence $(v_1, \ldots, v_N)$ defined
through this lemma is the {\em $(N,n)$--balanced} sequence.
For example, with $N=7$ $n=4$ we find
$(v_1, \ldots, v_7) = (1,1,0,1,0,1,0)$

Next, given $x<0$, $x \in \Q^*$, write
\[
x=-\ell+\frac{s}{t},
\]
uniquely, where $\ell, s, t$ are positive integers, $t>1$,
$\gcd(s,t)=1$, $s<t$,
and let $(v_1, \ldots, v_t)$ be the $(t,t-s)$--balanced sequence.
Construct a graph $G$ on $[0,n=3m]$ where
\begin{align*}
m &: = t(\ell+3)-(s+1),
\end{align*}
by letting the first and last parts consist of edges joining consecutive
integers and making the middle part as shown in the figure.
The vetices $a_0, \ldots, a_t$ are defined by
\begin{align*}
a_0 &:=m, \\
a_1 &:= a_0+ (\ell+1) + v_1, \\
a_j &:= a_{j-1} + (\ell+2) + v_j, \quad j=2,\ldots, t.
\end{align*}
Finally, $0$ is connected by an edge to some vertices as show in
the figure and some vertices are connected to $n$ directly.
For example, with $x=-11/7$ we have $x=-2+(3/7)$,
that is, $\ell=2$, $s=3$, $t=7$, $m=31$, $n=93$.
The $(7,4)$ balanced sequence is $(v_1, \ldots, v_7) = (1,1,0,1,0,1,0)$.
We also have
$a_0=31$,
$a_1-a_0= \ell+1+v_1=4$,
$a_2-a_1 = \ell+2+v_2 = 5$,
$a_3-a_2=4$,
$a_4-a_3=5$,
$a_5-a_4=4$,
$a_6-a_5=5$,
$a_7-a_6=4$.
That this graph belongs to $\fH_n$ is not obvious, neither it is obvious
that $x \in \mathbb X_\crit$. The details are omitted.

\subsection{\texorpdfstring{$\mathbb X_\crit \subset \mathbb Y$}{Xcrit C Y}}
We will show that $\mathbb Y^c \subset \mathbb X_\crit^c$.
Note that all irrationals are elements of $\mathbb Y^c$.

First assume that $x \in \mathbb Y^c$ is irrational.
Since, by Proposition \ref{Xcc}, $\mathbb X_\crit = \mathbb X_C$,
we show that $x \in \mathbb X_C^c$.
The set of points at which $x \mapsto W^x_{\Gamma_1,\Gamma_2}$ fails
to be differentiable
is included in the set of points $x$ for which there are two paths $\pi_1, \pi_2$
from $\Gamma_1$ to $\Gamma_2$
such that $w^x(\pi_1)=w^x(\pi_2)$ with
$\overline N(\pi_1) \neq \overline N(\pi_2)$.
This implies that $(\overline N(\pi_2)-\overline N(\pi_1)) x
=  N(\pi_1)- N(\pi_2)$, which means that $x$ is rational.
Hence $W^x_{\Gamma_1,\Gamma_2}$ is differentiable at rational points $x$
and, by the dominated convergence theorem, the same is true for $C(p.x)$.
We thus proved that if $x$ is irrational then $x \in \mathbb X_C^c$.

Next let $x \in \mathbb Y^c$ but not irrational.
Showing that $x \in \mathbb X_\crit^c$ one proceeds by showing that
for every $n \ge 3$ and every $G \in \fH_n$ there is a unique
$(x,G)$-maximal path or that every $(x,G)$-maximal path has the same
$\overline N_G(\pi)$ (=number of red edges).

If $x=1$ then
every edge and every non-edge of a graph $G\in \fG_n$ has charge $1$.
Hence the maximum length of all paths is $n$ and this is achieved
exactly one path, the path $(0,1,2,\ldots,n)$.
So $1 \in \mathbb X_\crit^c$.

We know that negative rationals and the number $0$ are critical.
So assume that $x>0$. We know that $1$ is not critical.
So assume that $x\in (0,1) \cup (1,\infty)$.
By the scaling property, we only need to consider $x \in (0,1)$.
Since $1/2,1/3,\ldots$ are critical, we assume that $x \in (0,1) \setminus
\{1/2, 1/3, \ldots\}$.

\begin{proposition}
Every $0<x<1$ that is not the reciprocal of an integer is not critical.
\label{propnot}
\end{proposition}

The proof of this is rather lengthy and relies on special properties
of maximal paths that can be found in Lemmas 4.5, 4.6, 4.7, 4.8, 4.9 and
4.10 of \cite{FKP18}.
To state them, we adopt some terminology.
Call an edge {\em short} if
its endpoints are successive integers; otherwise, call it {\em long}.
Say that edge $e=(i,j)$ is {\em  nested} in $e'=(i',j')$
if $i' \le i < j \le j'$ and $e\ne e'$.
Let $\pi, \pi' \in \Pi_{0,n}$.
We say that the interval $[i,j] \subset [0,n]$ is {\em $(\pi,\pi')$-special}
if the set of vertices $k \in [i,j]$ that belong to both $\pi$ and $\pi'$
consists of $i$ and $j$ only.

\begin{lemma}
\label{shortfact}
If $0<x<2$ then every maximal path contains all short blue edges.
\end{lemma}

\begin{lemma}
\label{longfact}
If $0<x<2$
then every long edge of a maximal path $\pi$ is blue
(in other words, every red edge of $\pi$ must be short).
\end{lemma}

\begin{lemma}
\label{nestfact}
If $x>0$ then no blue edge of a maximal path can be nested in a
different blue edge of another
maximal path.
\end{lemma}

\begin{lemma}
\label{specialfact}
Let $0<x<2$. Then for every pair $\pi, \pi' \in \Pi_{0,n}$ of maximal paths such that
$\overline N_G(\pi)\neq \overline N_G(\pi')$
there is a $(\pi,\pi')$-special interval $I$ such that
\[
\overline N_G(\pi|_{I})
\neq \overline N_G(\pi'|_{I})
\]
and
\[
N_G(\pi|_{I})-N_G(\pi'|_{I}) \in\{-1,1\}.
\]
\end{lemma}

The proof of the lemmas are omitted, but we give the proof of
Proposition \ref{propnot}:

\begin{proof}[Proof of Proposition \ref{propnot}]
We prove the contrapositive: if $0<x<1$ is critical then $x=1/m$ for
some integer $m$.
So suppose that $x$ is critical and $0<x<1$.
Then there is $n\ge 3$ and $G \in \fH_{n}$
(edges of $G$ are called blue and non-edges red) and two maximal paths $\pi_1$,
$\pi_2$ with different number of red edges: $\overline N_G(\pi_1)\neq \overline N_G(\pi_2)$.
By the Lemma \ref{specialfact}, there is a $(\pi_1, \pi_2)$-special interval $[i,j]$
such that  $N_G(\pi_1|_I)-N_G(\pi_2|_I)\in \{1,-1\}$.
Since
\[
0=w^x_G(\pi_1|_I)-w^x_G(\pi_2|_I) = (\overline N_G(\pi_1|_I)- \overline N_G(\pi_2|_I))x
+ (N_G(\pi_1|_I)-N_G(\pi_2|_I)),
\]
it follows that
\[
x=\frac{1}{|\overline N_G(\pi_1|{[i,j]})- \overline N_G(\pi_2|{[i,j]})|},
\]
and hence the reciprocal of a positive integer.
\end{proof}

\begin{remark}
We considered here differentiability properties
for a last passage percolation problem.
Differentiability properties of a first passage percolation
model were studied in \cite{SZ}.
\end{remark}

\section{Perfect simulation aspects of charged graphs}
\label{sec:perfectSimu}

We now turn our attention to charged graphs with edge charge distribution $F$:
to every pair of $(i,j)$ of integers, with $i<j$, assign a charge
$w_{i,j}$ with distribution $F$, independently.
For a path $\pi=(i_0, \ldots, i_\ell)$, that is, a finite increasing sequence
of integers, assign charge $w(\pi) = w_{i_0,i_1} +\cdots+w_{i_{\ell-1},i_\ell}$,
and let
\begin{equation}
\label{Wij}
W_{i,j} = \sup \{w(\pi): \text{ $\pi$ is a path from $i$ to $j$}\}.
\end{equation}
We use the notation $\vec G(\Z, F)$ for this random charged graph.
Thus, $\vec G(\Z, p\delta_1+(1-p)\delta_x)$ is another notation
the graph dealt with in
Section \ref{secana}, while $\vec G(\Z, p\delta_1+(1-p)\delta_{-\infty})$
is another notation for the Barak-Erd\H{o}s graph $\vec G(\Z, p)$.
Superadditivity inequality implies that there is a deterministic
$C(F)$ such that
\[
\frac{W_{0,n}}{n} \to C(F), \quad\text{ a.s. as } n \to \infty.
\]
In Section \ref{secana} we studied $C(p,x) \equiv
C(p\delta_1+(1-p)\delta_x)$ as a function of $x$.
To deal with $C(F)$ as a function of $F$ analytically
is beyond the scope of this survey. We will, however, explain how
to approach this problem from an algorithmic view point, one
that allows to draw conjectures by simulating the graph
as accurately as possible.

As explained at the start of Section \ref{secana}, we will assume that
$\esssup w_{i,j} > 0$,
else we are dealing with a first passage percolation problem.
We will also assume, for convenience, that this essential supremum is finite
so, without loss of generality, let
\begin{equation}
\label{esssup}
\esssup w_{i,j} = \inf\{z: F((z,\infty))=0\}=1.
\end{equation}
See Remark \ref{unbsup} below that explains that the finiteness of $\esssup
w_{i,j}$ is not a problem insofar as the perfect simulation is concerned.
Comparing $F$ with the distribution $p \delta_{1/2} + (1-p) \delta_{-\infty}$
with $p=F([1/2,1])$ we obtain that
\[
C(F)>0.
\]
The goal of this section is the construction of a random variable
whose expectation is $C(F)$ and such that this random variable
can be {\em perfectly simulated}.
A survey of perfect simulation
can be found in \cite{KEN05}. Its
relation to the so-called backwards-coupling was studied in \cite{FOSTWE98}. It
belongs to the broader area of coupling methods for stochastic recursions that
may entirely lack the Markovian property \cite{BORFOS92,COMFERFER,FK03}.

To do this we discuss an auxiliary Markovian
 particle system, called Max Growth System (MGS),
that is an analog of the Infinite Bin Model.
Our reference throughout this section is \cite{FKMR23}.

\subsection{The Max Growth System (MGS)}
The content of this subsection is purely deterministic.
Consider point measures on the set $\R\cup\{-\infty\}$ that will
be referred to as ``space'':
\[
\NN:= \{ \text{locally finite point measures $\nu$ on } \R\cup\{-\infty\} \text{ such that }\nu(\R_+) < \infty\}.
\]
For example, $\nu=2 \delta_1 + \delta_{-1.5} + 3 \delta_{-4} + \delta_{-\infty}
 \in \NN$.
We think of $\nu$ as consisting of $7$ particles, such that
two of them are at $1$, one at $-1.5$, three at $-4$ and one at $-\infty$;
so we can equivalently represent $\nu$ by the locations of
its particles $(1 , 1, -1.5, -4,-4, -4, -\infty)$ arranged in decreasing order.
In general, we let $\nu_1 \ge \nu_2 \ge \cdots$
be the locations of the particles of $\nu$.
So the function
\[
\NN \ni \nu \mapsto \nu_k := \text{ location of $k$th largest particle of $\nu$}
\]
is well-defined for each $k$ and $\nu = \sum_{k\ge 1} \delta_{\nu_k}$.
We also let
\[
\|\nu\| := \nu(\R \cup\{-\infty\}),
\quad
\inf \nu := \nu_{\|\nu\|},
\]
with $\inf \nu =-\infty$ if $\|\nu\|=\infty$.
Next let
\[
\WW := \big\{w=(w_1, w_2,\ldots):\,
\sup_{k \ge 1} w_k\leq 1,\;
w_k \in \R \cup\{-\infty\} \text{ for all } k\big\},
\]
and define
\[
\fm(\nu, w) := \sup_{k \ge 1} (\nu_k+w_k), \quad \nu \neq 0, \quad w \in \WW.
\]
The map responsible for the dynamics of the MGS is
defined by
\[
\Psi_w \nu := \nu + \delta_{\fm (\nu,w)}.
\]
Composing these maps, for possibly different $w$ each time,
gives a trajectory in $\NN$.
More precisely, the MGS starting at ``time'' $T$ from state $\nu(T) \in \N$
and driving sequence $w(t)$, $t>T$, is the sequence defined by
\[
\nu(t) = \Psi_{w(t)} \nu(t-1), \quad t > T,
\]
that is,
\[
\nu(t) = \Psi_{w(t)} \comp \Psi_{w(t-1)} \comp \cdots \comp \Psi_{w(T+1)} \nu(T).
\]
Let $\NN_0$ consist of those $\nu \in \NN$ with $\nu_1=0$.
Define the map
\[
\sigma : \NN \to \NN_0, \quad  \sigma \nu = \sum_{k\ge 1} \delta_{\nu_k-\nu_1}
\]
that places the origin of space at $\nu_1$.
We will need the following easily verified properties.
\begin{align*}
&\sigma \Psi_w = \sigma \Psi_w \sigma,
\\
&\fm(\sigma \nu, w)  = \fm(\nu,w) - \nu_1.
\end{align*}

We next consider an MGS $\nu(t)$, $t \ge 0$, and show that under certain conditions
on the driving sequence the quantity
\[
\mathfrak{M}(t) \equiv \fm(\nu(t-1), w(t))
\]
will eventually not
depend on the choice of the initial state $\nu(0)$.
\begin{lemma}[Decoupling property]
\label{lem:triangular}
Fix $\ell \in [0,1)$.
Consider two MGSs, $\nu$, $\tilde \nu$, starting at time $0$
with $\nu(0) \in \NN_0$ and $\tilde \nu(0)=\delta_0$.
Assume that the driving sequence $w(t)$, $t>0$, is the same for both.
Let $\mathfrak{M}(t) = \fm(\nu(t-1), w(t))$, $\tilde \fM(t)=\fm(\tilde{\nu}(t-1), w(t))$.
If
\begin{align}
\nu_2(0) & \le -\ell,
\label{trip-}
\\
\bar{w}(t) &:= \max \{ w_1(t), \ldots, w_t(t)\} \geq 1 - \ell
\quad \text{ for all $1 \leq t \leq n$,}
\label{trip0}
\end{align}
then
\begin{equation}
\label{concl}
\fM(t) =  \tilde\fM(t) \quad \text{ for all $1 \leq t \leq n$.}
\end{equation}
\end{lemma}
\begin{proof}
We have $\nu_1(0)=0 \ge -\ell \ge \nu_2(0)$, by assumption.
We will show by induction that, for all $n \in \N$,
\begin{equation}
\label{indn}
\text{if $\bar{w}(1), \ldots, \bar{w}(n)  \ge 1-\ell$ } \text{ then }
\begin{cases}
\fM(n) =  \tilde\fM(n)
\\
\nu(n)|_{\R_+} = \tilde \nu(n)|_{\R_+}
\\
\nu(n)(\R_+)=n+1
\\
\nu_{n+2}(n) \le -\ell
\end{cases}
\end{equation}
Let $n=1$ and assume that $\bar{w}(1) = w_1(1) \ge 1-\ell$.
We have
$\fM(1) = \max\{\nu_1(0) + w_1(1) , \nu_2(0) + w_2(1), \ldots\}
= w_1(1)$ because $\nu_1(0) + w_1(1) = w_1(1) \ge 1-\ell$
and since $1 \ge w_j(1)$, $-\ell \ge \nu_j(0)$ for all $j \ge 2$,
we have $w_1(1) \ge w_j(1)+\nu_j(0)$ for all $j \ge 2$ and so
only the first term survives in the maximum.
But then $\fM(1) = w_1(1)=\tilde \fM(1)$.
Since $\nu(1) = \nu(0) + \delta_{\fM(1)}$ we have that its
restriction on $\R_+$ equals $\delta_0+\delta_{w_1(1)}$
and $\nu_3(1) = \nu_2(0)\le -\ell$. So \eqref{indn} holds when $n=1$.

Assume next \eqref{indn} holds for some $n$. We show its veracity for $n+1$.
We work under the assumption that $\bar{w}(1), \ldots, \bar{w}(n), \bar{w}(n+1)  \ge 1-\ell$.
By \eqref{indn} we have
$\nu(n)|_{\R_+} = \tilde \nu(n)|_{\R_+}$ (containing $n+1$ particles
on $\R_+$)
and
$\nu_{n+1}(n-1) \le -\ell$.
So
\begin{align*}
\fM(n+1) &= \max_{j \ge 1}\{ \nu_j(n)+w_j(n+1) \}
\\
&= \max_{j \le n+1}\{ \nu_j(n)+w_j(n+1) \}
\vee \max_{j > n+1}\{ \nu_j(n)+w_j(n+1) \}
\\
&= \max_{j \le n+1}\{ \tilde \nu_j(n)+w_j(n+1) \}
\vee \max_{j > n+1}\{ \nu_j(n)+w_j(n+1) \}
\\
&= \max_{j \le n+1}\{ \tilde \nu_j(n)+w_j(n+1) \} = \tilde \fM(n+1).
\end{align*}
The reason that we dropped the second maximum is that
$\max_{j \le n+1}\{ \tilde \nu_j(n-1)+w_j(n) \}
\ge \max_{j \le n+1} w_j(n) \ge 1-\ell$
while $\max_{j > n+1}\{ \nu_j(n-1)+w_j(n) \} \le
\nu_{n+1}(n-1) + \max_{j > n+1} w_j(n) \le -\ell+1$.
We also have $\nu(n+1) = \nu(n) + \delta_{\fM(n+1)}
= \nu(n) + \delta_{\tilde \fM(n+1)}$
so $\nu(n+1)|_{\R_+} = \tilde \nu(n+1)|_{\R_+}$
and they have $n+2$ particles on $\R_+$,
while $\nu_{n+3}(n+1) = \nu_{n+2}(n) \le -\ell$.
\end{proof}

\begin{remark}
Lemma \ref{lem:triangular} and the conclusion \eqref{concl}
can easily be modified if we need to start the MGS
from an arbitrary time $T$ rather than $T=0$.
The quantity that replaces $\overline w(t)$ will be
\begin{equation}
  \label{tripT}
  \bar{w}(T;t) := \max \{ w_1(T+t), \ldots, w_t(T+t)\}.
\end{equation}
This is needed in Theorem \ref{thm1} below.
\end{remark}

\begin{remark}
The decoupling property is essentially responsible for
producing, in a stochastic version of the MGS, the so-called
renovating events that constitute a nice way for exhibiting
stability. The original theory can be found in
\cite{BOR78,BOR80,BOR98,BORFOS92,BORFOS94} and its extension for
functional of stochastic dynamical systems in \cite{FK18}.
\end{remark}

\subsection{The MGS of a random charged graph, coupling and stationarity}
\begin{definition}
Given a probability measure $Q$ on $\WW$
define the random process MGS$(Q)$ by
$\nu(t) = \Psi_{w(t)} \nu(t-1)$, $t \ge 1$, with
$w(1), w(2), \ldots$ being independent with common law $Q$,
with $\nu(0) \in \NN$ arbitrary.
Clearly, MGS$(Q)$ is Markovian.
In the particular case where the components of $w(1)$ are i.i.d.\
with common distribution $F$ we shall be using the notation
MGS$(Q_F)$.
\end{definition}

Let $\vec G(\Z, F)$ be a random charged graph with charge distribution $F$
satisfying \eqref{esssup}.
Let
\[
W_n := W_{0,n},
\]
in the notation of \eqref{Wij}.
In particular, $W_0=0$, $W_1=w_{0,1}$, $W_2=(w_{0,1}+w_{1,2}) \vee w_{1,2}$.
Let
\begin{equation}
\label{nugraph}
\nu^G(n) := \sum_{k=0}^n \delta_{W_k}, \quad n \ge 0.
\end{equation}
Then $\nu^G(n) \in \NN$ and has $n+1$ particles.
Clearly, $\nu^G(n) = \nu^G(n-1) + \delta_{W_n}$,
with $W_n = \max_{0 \le k \le n-1} (W_k+w_{k,n})
= \max_k (\nu^G_k(n-1)+ w_{\pi(k),n})$ where $\pi$ is a
permutation on $\{0,\ldots,n-1\}$ that puts
$W_0, \ldots, W_{n-1}$ in decreasing order.
Hence $W_n \eqdist \max_k (\nu^G_k(n-1)+ w_{k,n})$
and so if we consider the MGS$(Q_F)$ defined by
\[
\nu(n) = \nu(n-1) + \delta_{\fm(\nu(n-1), w(n))},
\]
starting with $\nu(0)=\delta_0$,
we see that
\begin{equation}
\label{numark}
\nu^G \eqdist \nu,
\end{equation}
not only component-wise, but also as processes.
Therefore, letting
\begin{equation}
\label{lppMaxWeight}
M_n = \max_{0 \leq k \leq n} W_k
\end{equation}
we have
\[
(M_n-M_{n-1} ,  n \ge 1)
\eqdist (\fm(\sigma \nu(n-1), w(n))^+, n \ge 1),
\]
which implies that $M_n/n$ is a sample mean:
\[
\frac{M_n}{n} = \frac{1}{n} \sum_{j=1}^n (M_j - M_{j-1})
\eqdist \frac{1}{n} \sum_{j=1}^n \fm(\sigma \nu(j-1),w(j))^+,
\]
where the last equality in distribution is at the level of sequences.
Since
\[
C(F)=\lim_{n \to \infty} \frac{W_n}{n}
= \lim_{n \to \infty} \frac{M_n}{n} \text{ a.s.},
\]
and since $\int_0^\infty x F(dx) < \infty$, we have
\[
C(F)= \lim_{n \to \infty} \frac{\E M_n}{n}
= \lim_{n \to \infty} \frac{1}{n}\sum_{j=1}^n
\E \fm(\sigma \nu(j-1), w(j))^+ .
\]
\begin{remark}
Let us make the reasonable ansatz that the quantity inside the last
expectation converges in distribution to some random variable:
\begin{equation}
\label{claimm}
\fm(\sigma \nu(n-1),w(n))^+ \convd \bar \fm^{+} \text{ as } n \to \infty.
\end{equation}
Then, a poor man's estimate of $C(F)$ would be to simply to
compute sample means $$n^{-1} \sum_{j=1}^n \E \fm(\sigma \nu(j-1), w(j))^+$$
for large $n$.
But the sample means are not from i.i.d.\ copies of the limiting
variable and thus the sampling is ``not perfect''.
To achieve perfect sampling, we must first show that the non-Markovian process
$\fm(\sigma \nu(n-1),w(n))$, $n=1,2,\ldots$, has a stationary
version
(we say that a process has a stationary version if it is eventually
equal to a stationary process; the ``eventually'' means that there
is a finite random time after which the two processes are equal a.s.)
from which we can sample i.i.d.\ copies constructively.
\end{remark}

\begin{theorem}
\label{thm1}
Suppose that $F$ satisfies \eqref{esssup}.
Let $\nu$ be an MGS$(Q_F)$ with driving sequence $w(t)$, $t \in \Z$,
starting from an arbitrary state at an arbitrary, possibly random, time.
Then there exists a stationary process $(\bar{\fm}(t), t \in \Z)$ such that
\[
\fm(\sigma\nu(t-1),w(t)) = \bar{\fm}(t) \text{ for $t$ large enough, a.s.}
\]
In particular $\E( \bar{\fm}(0)^+) = C(F)$.
\end{theorem}

\begin{proof}
Let $\ell \in [0,1)$ be such that $p:=F([1-\ell,1]) >0$,
which is possible due to \eqref{esssup}.
For $k \in \Z$, consider the event
\[
R_k:=
\{ w_1(k) \geq \ell\} \cap
\bigcap\limits_{j=1}^\infty
\big\{\max(w_1(k+j),\ldots,w_j(k+j)) \geq 1 - \ell\big\}.
\]
It is clear from its definition that $(R_k,  k \in \Z)$ is
a stationary sequence of events with
\[
 \P(R_k) = \P(R_0) = F([\ell,1])\prod_{j=1}^\infty (1 - (1 - p)^j) > 0.
\]
Consider the stationary random set $J := \{ k \in \Z : R_k \text{ holds}\}$.
Since $\P(R_k)> 0$, we have, by ergodicity (more specifically by the Poincar\'e
recurrence theorem),
$\inf J = -\infty$ and $\sup J = \infty$ a.s. We enumerate the elements of $J$ by
\[
  \cdots < T_{-1} < T_0 \leq 0 < T_1 < T_2 < \cdots
\]
Hence the $R_{T_i}$ are all events of probability $1$.
We define $\tilde \nu(t)$, $t \in \R$, by letting, for all $j$,
\[
\tilde \nu (T_j) = \delta_0, \qquad
\tilde{\nu}(t) = \sigma \sum_{i \in \Z}
\mathbbm{1}_{\{T_i < t < T_{i+1}\}}
\Psi^{t}_{w(t)}
\comp \cdots \comp
\Psi^{T_i+1}_{w(T_i+1)}
\delta_0, \quad t \not = T_j.
\]
On the other hand, thanks to Lemma \ref{lem:triangular}, the
process $\fm(\sigma \nu(t-1), w(t))$, $t > T_i$, is algebraically
independent of $\nu(T_i)$.
This shows, see \cite{FK03}, that the stationary process
$\bar m$ defined by
\[
\bar{\fm}(t) = \fm(\sigma \tilde \nu(t-1),w(t)) , \quad t \in \Z,
\]
satisfies
$\bar{\fm}(t) = \fm(\sigma \nu(t-1),w(t))$ for all $t$ large enough, a.s.
Next, using that
\[
   C(F) = \lim_{n \to \infty} \frac{\E(M_n)}{n} = \lim_{n \to \infty}
\frac{1}{n} \sum_{j=1}^n \E(\fm(\sigma \nu(j-1),w(j))^+) \quad \text{a.s.},
\]
and using the eventual equality between $\fm(\sigma \nu(t-1),w(t))$ and
$\bar{\fm}(t)$, we have
\[
  C(F) = \lim_{n \to \infty} \frac{1}{n} \sum_{j=1}^n \E(\bar{\fm}(j)^+) =
\E(\bar{\fm}(0)^+),
\]
by stationarity and ergodicity of the sequence.
\end{proof}

\subsection{Perfect simulation}
Recall the quantity $\overline w$ defined by \eqref{tripT}.
\begin{theorem}[Perfect simulation]
Define
\[
  T^* := \sup\{ t \le -1:\, {w}_1(t) \geq \ell, \min_{1\leq j \leq |t|} \overline w(t;j) \geq 1-\ell\}.
\]
Then $|T^*|< \infty$ a.s., and
\[
\bar{\fm}(0)
= \fm\left( \sigma
\Psi^{T^*+1}_{w(T^*+1}) \comp \cdots \comp \Psi^{0}_{w(0)}
\delta_0; w(0) \right)
\text{ a.s.}
\]
\end{theorem}

\begin{proof}
We recall that $(T_{-j}, j \in \N)$ are the negative elements of the random set $J$, with $T_{-1} > -\infty$. We remark that
\[
  w_1(T_{-1}) \geq \ell, \quad \bar{w}(T_{-1};j) \geq 1 - \ell \text{ for all $j> 0$},
\]
therefore $T_* \geq T_{-1}$, proving its finiteness.

Moreover, since
\[
w_1(T^*) \geq \ell, \min_{1 \leq j \leq |t|} \bar{w}(T^*;j) \geq 1-\ell,
\]
by Lemma \ref{lem:triangular}, the quantity
$\fm\left( \sigma
\Psi^{T^*+1}_{w(T^*+1}) \comp \cdots \comp \Psi^{0}_{w(0)}
\nu; w(0)\right)$
does not algebraically depend on the value of
$\nu \in \NN_0$. As a result, it is equal to $\bar{\fm}(0)$.
\end{proof}

\begin{center}
\begin{algorithm}[ht]
  \caption{Construction of a variable of law $\bar{\fm}(0)$.}
  \label{algo:main}
  \SetAlgoLined
  Fix $t=0$ and $J=1$\;
  Generate the variable $w_1(0)$\;
  Fix Stopping = False\;
  \While{Stopping = False}{
  \While{$\max_{1\le j \leq J} w_j(t) < 1 - \ell$}{Increase $J$ by $1$\;Generate the variable $w_J(t)$\;}
  \While{$J > 1$}{
  Decrease $J$ by $1$ and $t$ by $1$\;
  Generate $w_1(t), \ldots w_J(t)$\;
  \While{$\max_{1\le j \leq J} w_j(t) < 1 - \ell$}{Increase $J$ by $1$\;Generate the variable $w_J(t)$\;}}
  Decrease $t$ by $1$\;
  Generate $w_1(t)$\;
  Fix Stopping = $\{w_1(t) \geq \ell\}$\;
  }
  Fix $\nu = \delta_0$\;
  \For{$s$ from $t+1$ to $-1$}{Generate the variables $w_1(s),\ldots,w_{||\nu||}(s)$ \; Set $\fm = \max \{ \nu_j + w_j(s) \text{ for $1 \leq j\leq\Vert\nu\Vert$}\}$\; Add $\delta_\fm$ to $\nu$\;}
  Set $\fm = \max \{ \nu_j + w_j(0) \text{ for }1 \leq j\leq\Vert\nu\Vert\}$\;
  \textbf{Return:} $\fm - \nu_1$\;
\end{algorithm}
\end{center}

We now describe more precisely the perfect simulation algorithm. Let $F$ be a
probability distribution satisfying \eqref{esssup}, we fix $\ell \in
[0,1)$ such that $F([1-\ell,1]) \in (0,1)$. The algorithm requires the
construction of an array of i.i.d.\ random variables with common distribution $F$
until the random variable $T^*$ can be constructed.

To construct $T^*$ as well as $\bar{\fm}(0)$ from the sequence $\{w_j(t), j \in
\N, t \in \Z\}$, one only needs to consider a.s. finitely many elements of this
set, as $\{T^* = t\}$ is a measurable function of
\[
\{ w_1(t)\}\cup\{w_j(t+k), 1 \leq j \leq k \leq |t| \}
\]
and $\bar{\fm}(0)$ is a measurable function of
\[
\{ w_1(T^*)\}\cup\{ w_{j}(T^*+k), 1 \leq j \leq k \leq |T^*| \}.
\]
Therefore, we can explore triangular arrays of the form
\[
  \{ w_1(t)\}\cup  \{w_j(t+k), 1 \leq j \leq k \leq |t|\},
\]
progressively decreasing $t$ until time $T^*$ is detected. Once this random
variable is known, we construct the random variable $\bar{\fm}(0)$ using the
procedure described in Theorem \ref{thm1} from the previously discovered random
variables. A possible implementation is described in Algorithm~\ref{algo:main}.
We show a graphical representation of a run of Algorithm~\ref{algo:main} in
Figure \ref{fig:triangulararray}.

\begin{figure}[ht]
\centering
\includegraphics[width=6cm]{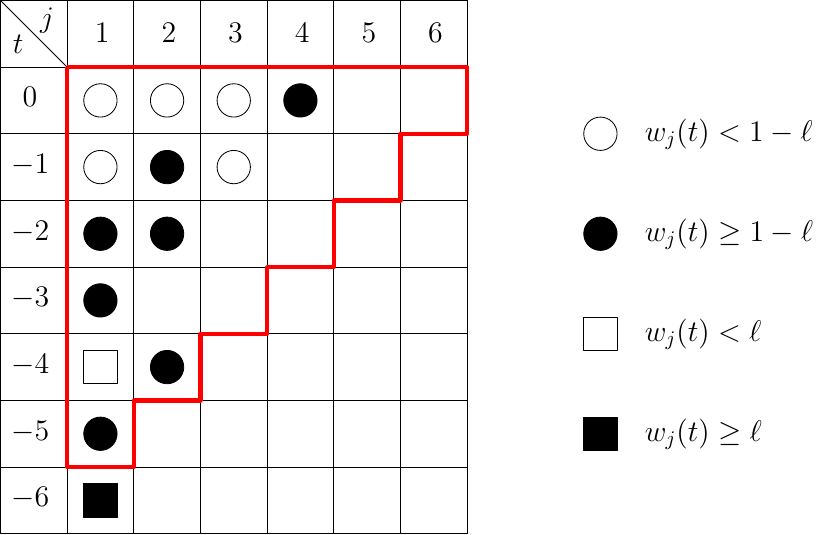}
\caption{Illustration of the execution of Algorithm \ref{algo:main} on an
example, in the case where $\ell<1-\ell$. The variables sampled until the
Boolean variable Stopping becomes True are pictured by black/white squares and
disks. One searches for the first time $T^*$ such that every line of index
$T^*+1\leq t \leq 0$ has at least one black disk between columns $1$ and $t-T^*$
and such that there is a black square in position $(T^*,1)$. The full triangular
array of variables used in the construction of $\nu$ is enclosed by a red
boundary.}
\label{fig:triangulararray}
\end{figure}

We observe that this algorithm has a complexity of $(T^*)^2$, as it is the
number of steps needed to generate the variable $\bar{\fm}(0)$. It is worth
noting that $-T^*$ can be constructed as the first hitting time of $0$ of the
Markov chain $(X_n)$ with initial state
\[
X_0=\min\{j\geq1, w_j(0)\geq 1-\ell\}
\]
and with transition probabilities defined for all $j \geq 2$ and $i \geq j$ by
\[
    P(j,j-1) = 1 - (1-p)^{j-1} \text{ and } P(j,i) = p(1-p)^{i-1}
\]
where $p = \P(w_1(0) \geq 1-\ell)$, with
\[
  P(1,0) = \P(w_1(0) \geq \ell), \quad \P(1,1) = \P(1 - \ell \leq w_1(0) < \ell), \quad P(1,j) = p(1-p)^{j-1} \text{ for }j \geq 2.
\]
The quantity $X_n$ corresponds to the value of the variable $J$ at the end of
the period when $t=-n$ in Algorithm \ref{algo:main}. In the example shown in
Figure \ref{fig:triangulararray}, we have
\[
(X_0,X_{-1},X_{-2},X_{-3},X_{-4},X_{-5},X_{-6})=(4,3,2,1,2,1,0).
\]
Note that $T^*$ has exponential tails.

The choice of the parameter $\ell$ may have an important effect on the behavior
of the average complexity $\E((T^*)^2)$ of the algorithm. We plotted $\ell
\mapsto \E((T^*)^2)$ in Figure \ref{fig:averageComplexity}, when the charge
distribution is given by $F(\mathrm d x) = \ind{x \leq 1 } e^{x-1} \mathrm d x$.
Additionally, as $p \to 0$, the quantity $\E((T^*)^2)$ grows to $\infty$. We
estimated $\E((T^*)^2)$ for $F = p\delta_1+(1-p)\delta_{-\infty}$ and plotted
this quantity as a function of $p$ in Figure \ref{fig:overComplexity}.

\begin{figure}[ht]
\centering
\includegraphics[width=6cm]{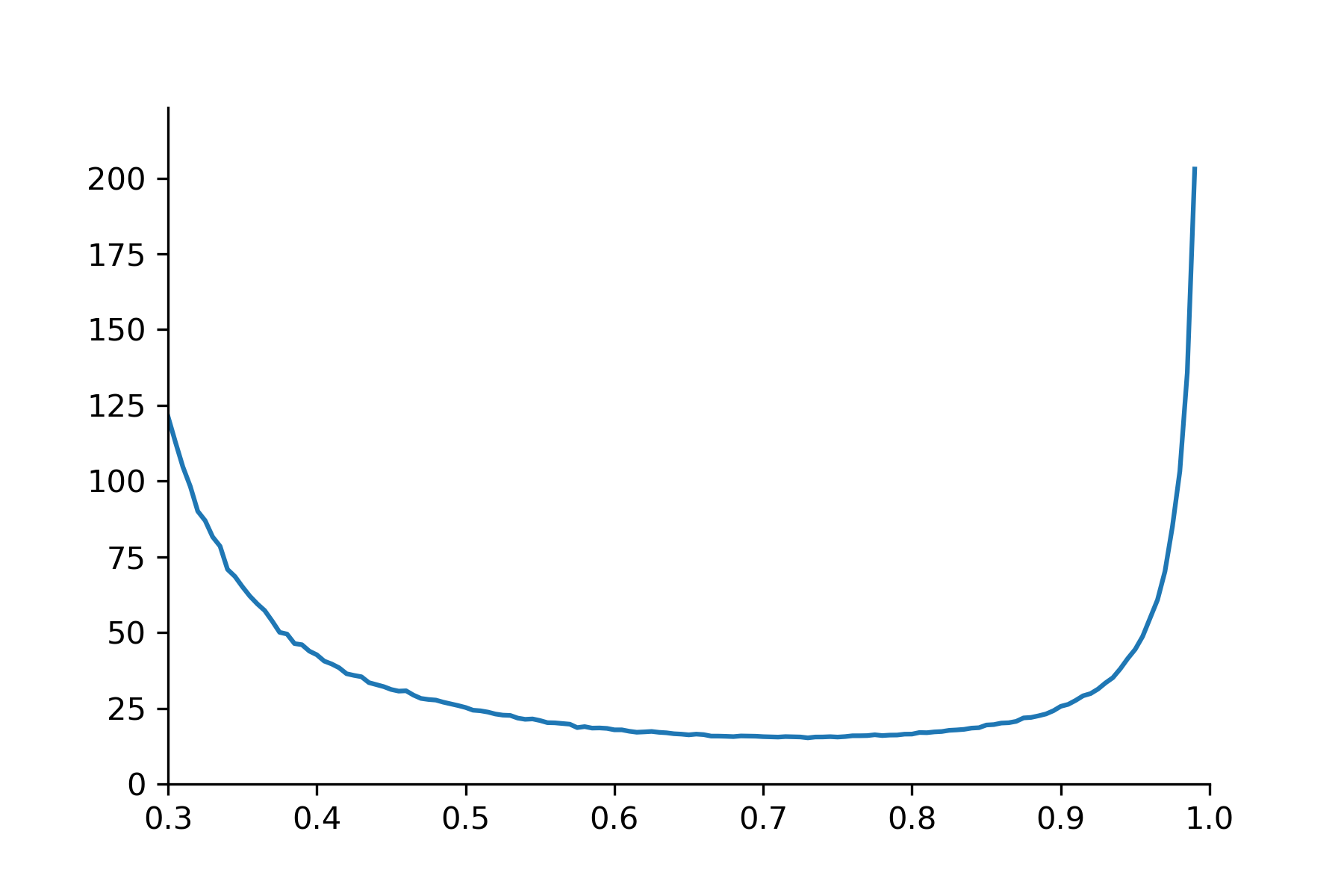}
\caption{Dependency in the parameter $\ell$ of the complexity of Algorithm
\ref{algo:main} with a charge distribution $F(\mathrm d x) = \ind{x \leq 1 }
e^{x-1} \mathrm d x$. The figure was obtained with a Monte Carlo simulation of
$N=10^4$ copies of $T^*$ for $100$ different values of $\ell$. For this charge
distribution, the
Monte Carlo simulations give $C(F) = 0.4432\pm0.0006$.}
\label{fig:averageComplexity}
\end{figure}

We observe in Figure \ref{fig:averageComplexity} that different choices of the
value $\ell$ can have a dramatic impact on the efficiency of Algorithm
\ref{algo:main}. Choosing a value $\ell$ too small has the effect of making the
first appearance of a triangular event too late. On the other hand, if $\ell$ is
too big then with high probability, one will have $w_1(T) \leq \ell$, and thus
the first ``successful'' triangular event will appear much later. For the
distribution $F$ we chose, it appears that an optimal choice of $\ell$ seems to
be around $\ell = 0.7$, which balances between these two extremes.

\begin{figure}[ht]
\centering
\includegraphics[width=6cm]{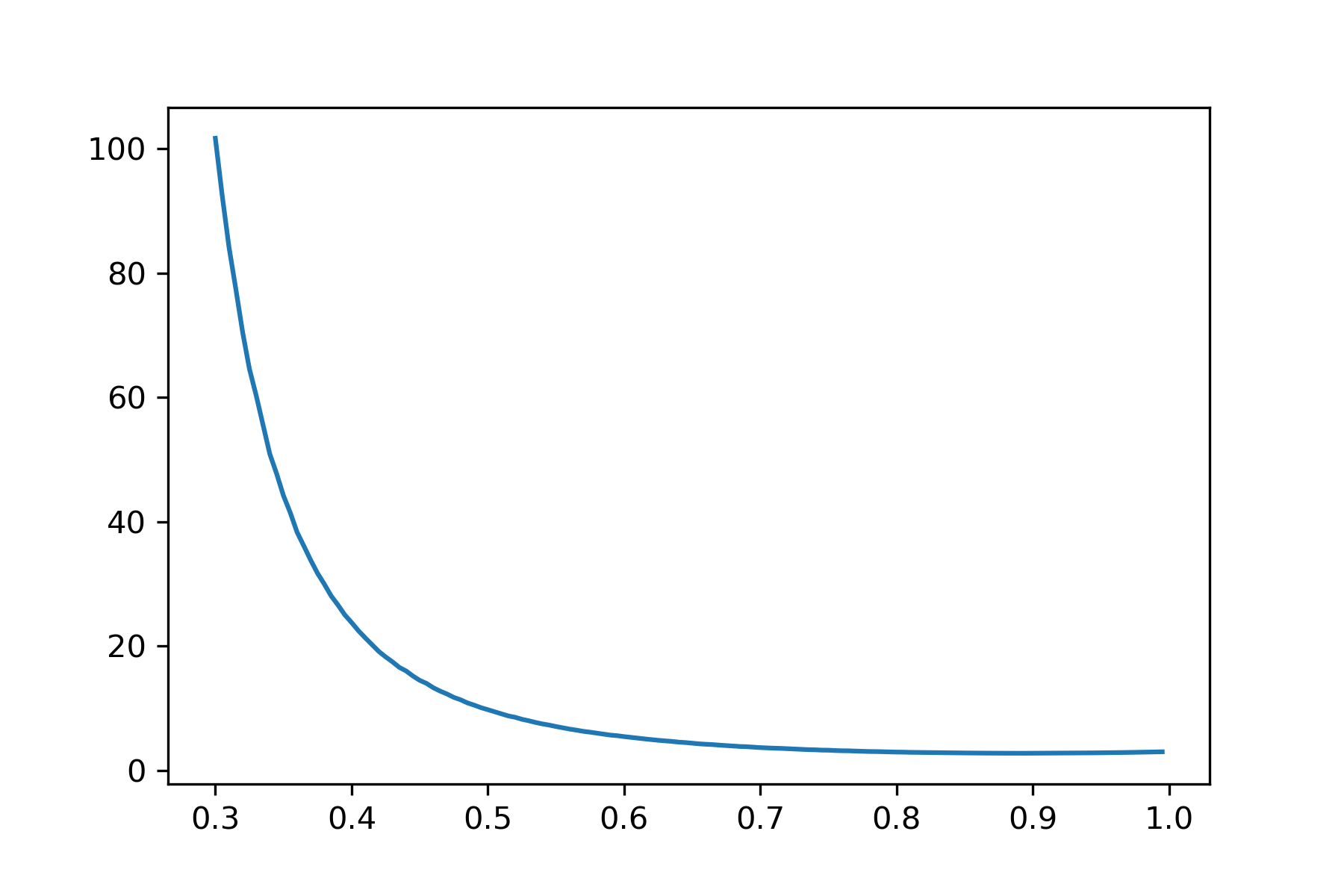}
\caption{Dependency in the parameter $p$ of the complexity of Algorithm
\ref{algo:main} applied to the detection of the longest path in the
Barak-Erd\H{o}s graph with parameter $p$. Figure obtained through Monte Carlo
simulation of $N = 10^5$ copies of $T^*$ for 120 different values of $p$.}
\label{fig:overComplexity}
\end{figure}

We observe in Figure \ref{fig:overComplexity} that if $F$ puts a large mass on
the negative half-line, the complexity of Algorithm \ref{algo:main} can become
quite large. The function $p \mapsto \E(T^*)^2$ grows at least exponentially in
$1/p$ as $p \to 0$ in the Barak-Erd\H{o}s graph, but we were not able to obtain
a good estimate of this rate of increase.

\begin{remark}[Perfect simulation when the charge distribution has an infinite
essential supremum]
\label{unbsup}
We explained how to perfectly simulate the random variable
$\bar{\fm}(0)^+$ whose expectation is  $C(F)$ under the assumption
that the supremum of the support of $F$ is $1$ or any finite number for that
matter.
To perfectly simulate a random variable whose expectation is $C(F)$ when
the supremum of the support of $F$ is  $\infty$, we use an idea by Glynn and
Rhee \cite{GlynnRhee}.
Denote by $X_n$ the variable $\bar{\fm}(0)^+$ when $F=F_n$ has $n$ as
the supremum of its support
and note that $X_n \to X$ a.s.\ and in $L^1$ for some random variable $X$.
Then $C(F_n) = \E X_n \to \E X = C(F)$, as $n \to \infty$.
We may not be able to perfectly simulate $X$. However,
if we let $\nu$ be a positive random integer such that $p_n=\P(\nu=n)>0$
for all $n$ and set $Y= (X_\nu-X_{\nu-1})/p_\nu$ then $Y$ can be perfectly
simulated. Moreover, $\E Y = C(F)$.
This is because $\E Y= \sum_{n=1}^\infty p_n \E (Y|\nu=n)
=  \sum_{n=1}^\infty \E(X_n-X_{n-1}) = \lim_{n \to \infty} \E X_n = \E X$.
\end{remark}

\begin{remark}
For the special case of a Barak-Erd\H{o}s graph, a simpler
perfect simulation algorithm is explained in \cite{FK03}.
\end{remark}

\section{Additional remarks and open problems}
The last passage percolation constant $C(p)$ for $\vec G(\Z, p)$ has been
one of the main concerns in this survey.
We showed how very good approximations can be obtained analytically
by relating $\vec G(\Z, p)$ to the IBM$(p)$ model.
The function $C(p)$ is not convex. We conjecture however that $C(p)/p$ is.

The passage to the IBM enables using branching processes techniques
that help explain the slow rate of convergence of $C(p)/p$ to $e$
as $p \to 0$ and its relation to the Brunet-Derrida behavior.
Such phenomena should be present when the geometric$(p)$, $0<p<1$,
distributions in the IBM
are replaced by a family of more general distributions.

The CLTs are generally available for the length of the longest or
heaviest path but, as explained, some moment conditions are needed
in case random weights are added to the edges.
There is one thing that is not totally clear in the CLTs and this is
the variance parameter. In, e.g., the case of $\vec G(\Z,p)$, we have
that $(L_n-C n)/\sqrt{\sigma^2 \lambda n}$ converges in
distribution to a standard normal with $\sigma^2$ being equal to
the variance of $L_{\Gamma_1, \Gamma_2}-C(\Gamma_2-\Gamma_2)$. This
quantity is unknown and not even bounds are available. The situation becomes
more complex in presence of weights.

CLTs are also available for the IBM and can take various forms.
For example, for the IBM$(p)$ model, it has been shown that
$\sqrt{n} \left(n^{-1} \sum_{k=0}^{[A_n t]} X_n(k)-t\right)$, $0 \le t \le 1$,
converges in law to a Brownian motion if $A_n = nC$ and to a Brownian bridge
if $A_n = L_n$. See \cite[Thm.\ 9.2]{FK03}.

When considering random directed graphs on partially ordered vertex sets,
the CLT limits are not necessarily Gaussian. The Brownian last passage
percolation process $Z(t) = \max_{0 \le s \le t} (X(s) + Y(t)-Y(s))$
pops up for example in the case of the vertex set $\Z \times \{1,2\}$.
We know nothing about the random variable $S=S(t)$ that achieves this maximum.
This is of interest as it would be related to the path that achieves the
maximum.

Concerning longest paths of $\vec G(\Z \times \Z, p)$ on a
window $[0,n] \times [0,m]$, as $n, m=m(n) \to \infty$,
the limit theorem of \eqref{Llim} is not optimal. We conjecture that
the Tracy-Widom limit is also possible when $m(n)$ grows with $n$ linearly.
One can of course ask for analogous results in the
$\vec G(\Z^d , p)$ case, when $d \ge 3$. It is not clear at all what
kind of laws would replace the Tracy-Widom distribution here.

The recent paper of Terlat \cite{terlat} makes progress in the behavior
of $C(F)$ when $F$ is a distribution whose support may include negative
numbers with a possible atom on $-\infty$. Concerning however the
type of limit theorems that one can obtain when $F$ is heavy-tailed,
such as those of Section \eqref{infsec}, the field is open.

Concerning other types of behavior of maximal lengths or weights,
e.g., in the large deviations sense, we mention that,
when weights have a distribution supported on the integers,
normal and moderate large deviations were obtained in \cite{KLMR21}.
In addition local limit theorems were obtained when the weight distribution
is non-lattice.

Distributional results for maximal/heaviest weighted paths are not available
beyond results around CLTs. However, in the sparse case, one can
use the recursive nature of the (possibly weighted) PWIT to come
up with functional equations for such quantities. See
e.g.\ \cite[eq.\ (13)]{FK18}.

Recall that $C(p,x)$ of Section \ref{secana},
is the last passage percolation constant when the weight law is
$p \delta_1 + (1-p) \delta_x$ for all edges $(i,j)$ with $i<j$.
There is only one solvable model we are aware of, that of $C(p,0)$
as in \cite{Dutta}. Can there be other ones?

If we replace $p \delta_1 + (1-p) \delta_x$ above
by $F_{p,x}=p Q+(1-p) \delta_x$ where $Q$ is a probability measure on $(0,\infty)$
we expect that the behavior of $p\mapsto C(F_{p,x})$ can be obtained
by techniques similar to those of \cite{MR16,MR19}, so long as
$\int y^2 Q(dy) < \infty$.
As a function of $x$, the $C(F_{p,x})$ quantity is continuous and
convex and it can be seen that $C(F_{p,x})$ is not differentiable at $x$
iff $Q$ has an atom $y$ such that $x/y \in \mathbb X_\crit$.
See \cite[Sec.\ 5]{FKP18}.

\section*{Acknowledgements}

We thank the referees for their insightful suggestions that helped us improve the exposition.

\end{document}

%% file: figTex/configuration.tex
\begin{tikzpicture}[scale = 0.6]
\draw (-1,0) -- (7,0);
\draw (-.5,2) node{$\cdots$};
\draw (0,0) -- (0,4.5);
\draw (1,0) -- (1,4.5);
\draw (2,0) -- (2,4.5);
\draw (3,0) -- (3,4.5);
\draw (4,0) -- (4,4.5);
\draw (5,0) -- (5,4.5);
\draw (6,0) -- (6,4.5);
\draw (0.5,0.5) node{$8$} circle (0.4);
\draw (0.5,1.4) node{$9$} circle (0.4);
\draw (1.5,0.5) node{$7$} circle (0.4);
\draw (2.5,0.5) node{$3$} circle (0.4);
\draw (2.5,1.4) node{$4$} circle (0.4);
\draw (2.5,2.3) node{$5$} circle (0.4);
\draw (2.5,3.2) node{$6$} circle (0.4);
\draw (3.5,0.5) node{$1$} circle (0.4);
\draw (3.5,1.4) node{$2$} circle (0.4);
\end{tikzpicture}

%% file: figTex/projectedMarkovA.tex
\begin{tikzpicture}[state/.style={circle,draw,font=\scriptsize},
        > = Stealth,
        auto,
        prob/.style = {inner sep=2pt,font=\scriptsize},
        loop right/.style={out=30,in=330,loop},
        ]
        \node[state]  (a) at (0,0) {\phantom{[}$\emptyset$\phantom{]}};
        \node[state]  (b) at (1.5,0)   {$[1]$};
        \path[->]   (a) edge[very thick, color=red, bend left] node[prob]{$1,2$} (b)
                    (b) edge[out=330,in=30,loop,very thick,color=red] node[below right,prob]{$1$} (b)
                    (b) edge[bend left] node[prob]{$2$} (a);
\end{tikzpicture}

%% file: figTex/projectedMarkovB.tex
\begin{tikzpicture}[state/.style={circle,draw,font=\scriptsize},
        > = Stealth,
        auto,
        prob/.style = {inner sep=2pt,font=\scriptsize},
        loop right/.style={out=30,in=330,loop},
        ]
        \node[state]  (a) at (0,0) {\phantom{[}$\emptyset$\phantom{]}};
        \node[state]  (b) at (1.5,1)  {$[1]$};
        \node[state]  (c) at (1.5,-1) {$[2]$};
        \node[state]  (d) at (3,0) {$[1,1]$};
        \path[->]   (a) edge[very thick, color=red] node[prob]{$1,2,3$} (b)
                    (b) edge[bend left] node[right,prob]{$2,3$} (c)
                    (b) edge[very thick, color=red,bend left] node[prob]{$1$} (d)
                    (c) edge node[prob]{$3$} (a)
                    (c) edge[very thick, color=red,bend left] node[prob]{$1,2$} (b)
                    (d) edge node[prob]{$2$} (c)
                    (d) edge node[prob]{$3$} (b)
                    (d) edge[out=330,in=30,loop,very thick,color=red] node[below right,prob]{$1$} (d);
\end{tikzpicture}

%% file: figTex/berank.tex
\begin{tikzpicture}[xscale=1.4,state/.style={circle,draw,font=\scriptsize},
        > = Stealth,
        auto,
        prob/.style = {inner sep=2pt,font=\scriptsize},
        loop right/.style={out=30,in=330,loop},
        ]
        \node[state] (a) at (0,0) {$\kappa_0$};
        \node[state] (b) at (1,0) {$\kappa_1$};
        \node[state] (c) at (2,0) {$\kappa_2$};
        \node[state] (d) at (3,0) {$\kappa_3$};
        \node[state] (e) at (4,0) {$\kappa_4$};
        \node[state] (f) at (5,0) {$\kappa_5$};
        \node[state] (g) at (6,0) {$\kappa_6$};
        \node[state] (h) at (7,0) {$\kappa_7$};
        \node[state] (i) at (8,0) {$\kappa_8$};
        \node[state,densely dashed] (j) at (9,0) {$\kappa_9$};

        \path[-]
        (a) edge[bend left] (f)
        (b) edge[bend right] (i)
        (d) edge[bend left] (g)
        (a) edge[very thick, color=red] (b)
        (a) edge[very thick, color=red,bend left] (c)
        (a) edge[very thick, color=red, bend left] (d)
        (e) edge[very thick, color=red,bend left] (g)
        (c) edge[very thick, color=red,bend right] (e)
        (c) edge[very thick, color=red,bend right] (i)
        (d) edge[very thick, color=red,bend left] (f)
        (a) edge[very thick, color=red,bend left] (h)
         ;
         \draw[blue] (0,0.7) node {$0$};
         \draw[blue] (1,0.7) node {$1$};
         \draw[blue] (2,0.7) node {$1$};
         \draw[blue] (3,0.7) node {$1$};
         \draw[blue] (4,0.7) node {$2$};
         \draw[blue] (5,0.7) node {$2$};
         \draw[blue] (6,0.7) node {$3$};
         \draw[blue] (7,0.7) node {$1$};
         \draw[blue] (8,0.7) node {$2$};
         \draw[red] (0,-.7) node {$9$};
         \draw[red] (1,-.7) node {$8$};
         \draw[red] (2,-.7) node {$7$};
         \draw[red] (3,-.7) node {$6$};
         \draw[red] (4,-.7) node {$4$};
         \draw[red] (5,-.7) node {$3$};
         \draw[red] (6,-.7) node {$1$};
         \draw[red] (7,-.7) node {$5$};
         \draw[red] (8,-.7) node {$2$};
         \draw (0,-1.5) node {$q^8$};
         \draw (1,-1.5) node {$q^7$};
         \draw (2,-1.5) node {$q^6$};
         \draw (3,-1.5) node {$q^5$};
         \draw (4,-1.5) node {$q^3$};
         \draw (5,-1.5) node {$q^2$};
         \draw (6,-1.5) node {$1$};
         \draw (7,-1.5) node {$q^4$};
         \draw (8,-1.5) node {$q$};

\end{tikzpicture}

%% file: figTex/scenaria3.tex
\begin{tikzpicture}[state/.style={thick,red,circle,draw,font=\scriptsize},
        > = Stealth,
        auto,
        prob/.style = {inner sep=2pt,font=\scriptsize},
        loop right/.style={out=30,in=330,loop},
        ]
        \node[state] (a) at (0,0) {$\kappa_0$};
        \draw (-.4,.4) node{$\emptyset$};
        \node[state] (b) at (-2,-1) {$\kappa_1$};
        \draw (-2.5,-1) node{$1$};
        \node (c) at (-1,-1) {};
        \node (d) at (0,-1) {};
        \node[state] (e) at (1,-1) {$\kappa_2$};
        \draw (.5,-1) node{$4$};
        \node (f) at (2,-1) {};
        \node (g) at (-2.7,-2) {};
        \node (h) at (-2,-2) {};
        \node (i) at (-1.3,-2) {};
        \node[state] (j) at (1,-2) {$\kappa_3$};
        \draw (0.4,-2) node{$41$};
        \shade[top color=blue,bottom color=white] (c) +(0,.0) --+ (-.3,-.7) --+ (.3,-.7) -- cycle;
        \shade[top color=blue,bottom color=white] (d) +(0,0) --+ (-.3,-.7) --+ (.3,-.7) -- cycle;
        \shade[top color=blue,bottom color=white] (f) +(0,.0) --+ (-.3,-.7) --+ (.3,-.7) -- cycle;
        \shade[top color=blue,bottom color=white] (g) +(0,.0) --+ (-.3,-.7) --+ (.3,-.7) -- cycle;
        \shade[top color=blue,bottom color=white] (h) +(0,.0) --+ (-.3,-.7) --+ (.3,-.7) -- cycle;
        \shade[top color=blue,bottom color=white] (i) +(0,.0) --+ (-.3,-.7) --+ (.3,-.7) -- cycle;

        \path[-]
        (a) edge[thick,red] (b)
        (a) edge[blue] (-1,-1)
        (a) edge[blue] (0,-1)
        (a) edge[thick, red] (e)
        (a) edge[blue] (2,-1)
        (e) edge[thick,red] (j)
        (b) edge[blue] (-2.7,-2)
        (b) edge[blue] (-2,-2)
        (b) edge[blue] (-1.3,-2);
\end{tikzpicture}

%% file: figTex/fig_0crit.tex
\begin{tikzpicture}
  \node (a) at (0,0) {};
  \node (b) at (1,0) {};
  \node (c) at (2,0) {};
  \node (d) at (3,0) {};

  \path[-,blue]
  (a.center) edge (b.center)
  (a.center) edge[bend left] (c.center)
  (b.center) edge[bend left] (d.center)
  (c.center) edge (d.center);

  \draw (a) node{$\bullet$} node[below]{0};
  \draw (b) node{$\bullet$} node[below]{1};
  \draw (c) node{$\bullet$} node[below]{2};
  \draw (d) node{$\bullet$} node[below]{3};
\end{tikzpicture}

%% file: figTex/fig_+crit.tex
\begin{tikzpicture}
  \node (a) at (0,0) {};
  \node (b) at (1,0) {};
  \node (c) at (2,0) {};
  \node (d) at (3,0) {};
  \node (e) at (4,0) {};
  \node (f) at (5,0) {};
  \node (g) at (6,0) {};

  \path[-,blue]
  (a.center) edge (b.center)
  (f.center) edge (g.center)
  (a.center) edge[bend left] (c.center)
  (a.center) edge[bend left] (d.center)
  (a.center) edge[bend left] (e.center)
  (a.center) edge[bend left] (f.center)
  (b.center) edge[bend left] (g.center)
  (c.center) edge[bend left] (g.center)
  (d.center) edge[bend left] (g.center)
  (e.center) edge[bend left] (g.center)
  (b.center) edge[bend right] (f.center);

  \draw (a) node{$\bullet$} node[below]{0};
  \draw (b) node{$\bullet$} node[below]{1};
  \draw (f) node{$\bullet$} node[below]{$k$};
  \draw (g) node{$\bullet$} node[below]{$k+1$};
  \draw (c) node{$\bullet$};
  \draw (d) node{$\bullet$};
  \draw (e) node{$\bullet$};
\end{tikzpicture}

%% file: figTex/fig_negcrit.tex
\begin{tikzpicture}
  \node (a) at (0,0) {};
  \node (b) at (1,0) {};
  \node (c) at (2,0) {};
  \node (d) at (3,0) {};
  \node (e) at (4,0) {};
  \node (f) at (5,0) {};
  \node (g) at (6,0) {};
  \node (h) at (7,0) {};
  \node (i) at (8,0) {};
  \node (j) at (9,0) {};

  \path[-,blue]
  (a.center) edge (b.center)
  (b.center) edge (c.center)
  (c.center) edge (d.center)
  (d.center) edge (e.center)

  (f.center) edge (g.center)
  (g.center) edge (h.center)
  (h.center) edge (i.center)
  (i.center) edge (j.center)

  (a.center) edge[bend left] (f.center)
  (e.center) edge[bend left] (j.center);

  \draw (a) node{$\bullet$} node[below]{0};
  \draw (b) node{$\bullet$} node[below]{1};
  \draw (e) node{$\bullet$} node[below] {$\ell+1$};
  \draw (f) node{$\bullet$} node[below]{$\ell+2$};
  \draw (i) node{$\bullet$} node[below] {$n-1$};
  \draw (j) node{$\bullet$} node[below right]{$n := 2\ell + 3$};
  \draw (c) node{$\bullet$};
  \draw (d) node{$\bullet$};
  \draw (g) node{$\bullet$} ;
  \draw (h) node{$\bullet$} ;
\end{tikzpicture}

%% file: figTex/fig_lastgraph.tex
\begin{tikzpicture}[> = Stealth]
  \node (0) at (0,0) {};
  \node (a0) at (1,0) {};
  \node (a01) at (1.8,0) {};
  \node (a02) at (2.2,0) {};
  \node (a03) at (2.6,0) {};

  \node (a1) at (3,0) {};
  \node (a10) at (3.4,0) {};
  \node (a11) at (3.8,0) {};
  \node (a12) at (4.2,0) {};
  \node (a13) at (4.6,0) {};

  \node (abis) at (5,0) {};
  \node (abis0) at (5.4,0) {};
  \node (abis1) at (5.8,0) {};
  \node (abis2) at (6.2,0) {};
  \node (abis3) at (6.6,0) {};

  \node (atm) at (7,0) {};
  \node (atm0) at (7.4,0) {};
  \node (atm1) at (7.8,0) {};
  \node (atm2) at (8.2,0) {};
  \node (atm3) at (8.6,0) {};

  \node (at) at (9,0) {};
  \node (n) at (10,0) {};

  \path[-,blue]
  (a0.center) edge[bend left=60] (a1.center)
  (a1.center) edge[bend left=60] (abis.center)
  (abis.center) edge[bend left=60] (atm.center)
  (atm.center) edge[bend left=60] (at.center);

  \draw [cyan,densely dashed] (0.center) -- (a0.center);
  \draw[cyan,densely dashed] (at.center) -- (n.center);

  \path[cyan]
    (a01.center) edge (a02.center)
    (a02.center) edge (a03.center)
    (a03.center) edge[bend left=60] (a10.center);

  \path[cyan]
    (a11.center) edge (a12.center)
    (a12.center) edge (a13.center)
    (a13.center) edge[bend left=60] (abis0.center);

  \path[cyan]
    (abis1.center) edge (abis2.center)
    (abis2.center) edge (abis3.center)
    (abis3.center) edge[bend left=60] (atm0.center);

  \path[cyan]
    (atm1.center) edge (atm2.center)
    (atm2.center) edge (atm3.center)
    (atm3.center) edge[bend left] (at.center);

\fill (0) circle[radius=2pt];
\fill (a0) circle[radius=2pt];
\fill (a1) circle[radius=2pt];
\fill (abis) circle[radius=2pt];
\fill (atm) circle[radius=2pt];
\fill (at) circle[radius=2pt];
\fill (n) circle[radius=2pt];

\fill (a01) circle[radius=1pt];
\fill (a02) circle[radius=1pt];
\fill (a03) circle[radius=1pt];
\fill (a10) circle[radius=1pt];
\fill (a11) circle[radius=1pt];
\fill (a12) circle[radius=1pt];
\fill (a13) circle[radius=1pt];
\fill (abis0) circle[radius=1pt];
\fill (abis1) circle[radius=1pt];
\fill (abis2) circle[radius=1pt];
\fill (abis3) circle[radius=1pt];
\fill (atm0) circle[radius=1pt];
\fill (atm1) circle[radius=1pt];
\fill (atm2) circle[radius=1pt];
\fill (atm3) circle[radius=1pt];

\draw (0) + (0,-.1) node[below] {0};
\draw (a0)+ (0,-.1) node[below] {$a_0$};
\draw (a1)+ (0,-.1) node[below] {$a_1$};
\draw (atm)+ (0,-.1) node[below] {$a_{t-1}$};
\draw (at)+ (0,-.1) node[below] {$a_{t}$};
\draw (n)+ (0,-.1) node[below] {$n$};

\draw[<-,cyan,densely dashed] (a01.center) -- +(-.2,-.8) node[below left] {0};
\draw[->,cyan,densely dashed] (a10.center) -- +(.0,-.8) node[below] {$n$};

\draw[<-,cyan,densely dashed] (a11.center) -- +(-.0,-.8) node[below] {0};
\draw[->,cyan,densely dashed] (abis0.center) -- +(.0,-.8) node[below] {$n$};
\draw[<-,cyan,densely dashed] (abis1.center) -- +(-.0,-.8) node[below] {0};
\draw[->,cyan,densely dashed] (atm0.center) -- +(.2,-.8) node[below right] {$n$};
%
%
%
\end{tikzpicture}